\documentclass[]{amsart} 

\usepackage{amssymb}

\usepackage{graphicx}

\usepackage{xcolor}
\newtheorem{theorem}{Theorem}[section]
\newtheorem{lemma}[theorem]{Lemma}
\newtheorem{corollary}[theorem]{Corollary}

\theoremstyle{definition}
\newtheorem{definition}[theorem]{Definition}
\newtheorem{example}[theorem]{Example}
\newtheorem*{algorithm}{Stabilized LF-LTS Galerkin FE Algorithm}

\theoremstyle{remark}
\newtheorem{remark}[theorem]{Remark}

\numberwithin{equation}{section}

\begin{document}

\title[Stabilized leapfrog based local time-stepping]{Stabilized leapfrog based local time-stepping method for the wave equation}

\author{Marcus J. Grote}
\address{Department of Mathematics and Computer Science, University of Basel, Spiegelgasse 1, 4051 Basel, Switzerland}
\email{marcus.grote@unibas.ch}
\thanks{}

\author{Simon Michel}
\address{Department of Mathematics and Computer Science, University of Basel, Spiegelgasse 1, 4051 Basel, Switzerland}
\email{simon.michel@unibas.ch}
\thanks{}

\author{Stefan A. Sauter}
\address{Institute for Mathematics, University of Zurich, Winterthurerstrasse 190, 8057 Zurich, Switzerland}
\email{stas@math.uzh.ch}
\thanks{}



\date{}

\begin{abstract}
Local time-stepping methods permit to overcome the severe stability constraint on explicit methods
caused by local mesh refinement without sacrificing explicitness. In \cite{DiazGrote09},
a leapfrog based explicit local time-stepping (LF-LTS) method was 
proposed for the time integration of second-order wave equations. Recently, optimal convergence rates 
were proved for a conforming FEM discretization, albeit under a CFL stability condition
where the global time-step, $\Delta t$, depends on the smallest elements in the mesh \cite{grote_sauter_1}. 
In general one cannot improve upon that stability constraint, 
as the LF-LTS method may become unstable at certain
discrete values of $\Delta t$. To remove those critical values of $\Delta t$,
we apply a slight modification (as in recent work on LF-Chebyshev methods \cite{CarHocStu19}) to the original LF-LTS method which nonetheless
preserves its desirable properties: it is fully explicit, second-order accurate, satisfies a three-term (leapfrog like) recurrence relation, and conserves the energy. The new stabilized LF-LTS method 
also yields optimal convergence rates
for a standard conforming FE discretization, yet under a CFL condition where $\Delta t$ no longer depends on the mesh size inside the locally refined region. 
\end{abstract}

\maketitle

\section{Introduction\label{SecIntro}}
For the time integration of second-order wave equations,
the leapfrog (LF) method \cite{HLWActa} probably remains to this day the most popular numerical method. Based on a centered finite difference approximation of the second-order time derivative, 
it is second-order accurate, explicit, time-reversible, symplectic and, for linear problems, conserves 
(a discrete version of) the energy for all time \cite{HLW}. 
For the spatial discretization of partial differential equations, finite element methods (FEM) provide a flexible approach, which easily
accomodates a varying mesh size or polynomial degree. Not only do FEM permit the use of high-order polynomials, necessary to capture the oscillatory nature of wave phenomena and keep numerical dispersion (``pollution error'') minimal \cite{BabuskaSauter97}, but they 
are also apt at locally resolving small geometric features or material interfaces. Hence the combined
FEM and LF based numerical discretization of second-order wave equations has proved a versatile and highly efficient approach, be it in acoustics, electromagnetics or elasticity.

Local mesh refinement, however, can cause a severe bottleneck 
for the LF method, or any other standard explicit scheme, due to the stringent CFL stability condition
on the time-step imposed by the smallest element in the mesh. Even when the locally refined
region consists of a few small elements only, those elements will dictate
a tiny time-step throughout the entire computational domain for stability.
To overcome the crippling effect of a few small elements, without sacrificing its
high efficiency or explicitness elsewhere, various local time integration strategies were proposed in
recent years which use smaller time-steps or an alternative method inside the regions of local mesh refinement. Hence, inside the ``coarse'' part of the mesh, where most of the elements reside, the standard LF method is used as is. Inside the ``fine'' part, however, where the remaining smaller elements are located, the time
integration is either implicit or explicit.

Following the classical IMEX (implicit-explicit) approach, a first locally implicit method for a DG discretization of Maxwell's equations was proposed by Piperno \cite{Piperno}; it combines the LF (or St\o rmer-Verlet) method in the coarse part with the Crank-Nicolson (CN) method in the fine part of the mesh and was further analyzed in Dolean et al. \cite{DFFL10}.
Independently, Verwer proposed a similar second-order CNLF method for Maxwell's equations \cite{Ver2011}.
From a comparison of those two schemes, Descombes et al. ``advise the use of the IMEX method from \cite{Ver2011}'' to avoid any order reduction \cite{DLM13}.
Hochbruck and Sturm proved optimal error estimates for the CNLF scheme from  \cite{Ver2011} when combined with a centered or an upwind discontinuous Galerkin (DG) FE discretization of Maxwell's equations \cite{HS16,HS19}. 
In \cite{ChabassierImperiale2015}, Chabassier and Imperiale proposed fourth-order energy-preserving IMEX schemes for the wave equation. Here, the computational domain is divided by a fixed artificial boundary
 into a coarse and fine region with a Lagrange multiplier along the interface.

Even earlier,  Collino et al. proposed a LF based local time-stepping (LTS) method for the wave equation, 
which also employs the explicit LF method in the fine part yet with a smaller step-size~\cite{ColFouJol3,ColFouJol1}; error estimates were derived
in~\cite{ColFouJol2, JolyRodriguez} for the one-dimensional wave equation. Stability of 
this second-order method, however, is guaranteed by enforcing the conservation of a discrete energy which requires at every time-step the solution of a linear
system for the shared unknowns at the interface; hence, the overall scheme is not fully explicit. 

A fully explicit second-order LTS method based on St\o rmer-Verlet and a DG space discretization was also proposed for Maxwell's equations by Piperno \cite{Piperno}, which was rewritten in a somewhat more
efficient time-staggered form by Montseny et al. \cite{MPFC08}. In \cite{DiazGrote09}, Diaz and Grote proposed LF based local time-stepping (LTS)
methods of arbitrarily high accuracy for second order wave equations. Hence, when combined with a mass-lumped conforming \cite{COHEN,mass_lumping_2d} or discontinuous Galerkin FE discretization \cite{GSS06} in space, the resulting method is truly explicit and inherently parallel; it was successfully applied to 3D seismic wave propagation \cite{MZKM13}. A multilevel version was later proposed \cite{DiazGrote15} and achieved high parallel efficiency on an HPC architecture \cite{Rietmann2017}.

Optimal convergence rates for the LF-LTS method from \cite{DiazGrote09} were recently 
derived for a conforming FEM discretization, albeit under a CFL condition
where the global time-step $\Delta t$ in fact depends on the smallest elements in the mesh \cite{grote_sauter_1}. 
In doing so, the inner loop over the $p$ local LF steps of size $\Delta t / p$ was rewritten
in terms of a single global time-step $\Delta t$, which involves Chebyshev polynomials.
In general one cannot improve upon the stability constraint on 
$\Delta t$, as the LF-LTS method may become unstable at certain
discrete values of $\Delta t$. Although those instabilities only matter in special situations and for long time simulation, they nonetheless
thwart any attempt to guarantee that the numerical solution remains bounded for all time independently of $p$, that is under a CFL condition imposed by the coarse mesh only. In fact, similar instabilities can also occur in 
leapfrog-Chebyshev (LFC, discrete Gautschi-type) methods without added stabilization \cite{GilbertJoly,CarHocStu19}, which are closely related to LTS schemes. 
As a remedy, Carle, Hochbruck and Sturm \cite{CarHocStu19} 
introduced an important class of stabilized LFC methods based on 
stabilized Chebyshev polynomials together with a special starting value. 

The reformulation of the original LF-LTS method \cite{DiazGrote09} using Chebyshev polynomials in \cite{grote_sauter_1} is key for its subsequent stabilization.
More concretely, by replacing the Chebyshev polynomials by their stabilized analogues, as in \cite{CarHocStu19} for LFC methods,  
we shall devise a stabilized version of the original LF-LTS method \cite{DiazGrote09}, which
completely removes the potentially unstable behavior at discrete time-steps while preserving
all the desirable properties of the original method: it is fully explicit, proceeds by a three-term recurrence relation 
and conserves (a discrete version of) the energy; hence, the (leapfrog-like) structure
of the resulting algorithm remains unchanged. The same stabilized LF-LTS algorithm
was developed independently and in parallel by the group of M. Hochbruck \cite{CarHoc}.
Here we develop a convergence analysis for the fully discrete stabilized LF-LTS method
under a CFL condition where $\Delta t$
no longer depends on the mesh size inside the locally refined region. 

The rest of our paper is structured as follows. In Section 2, we first introduce the notation underlying a conforming Galerkin FE discretization of the wave equation together with mass-lumping. Next, we present the
Galerkin FE formulation of the new stabilized LF-LTS method, where the coefficients in the inner
loop are replaced by those from stabilized Chebyshev polynomials. 
In Section 3, we prove continuity, symmetry and coercivity of the resulting stabilized bilinear form and introduce a new CFL stability condition, which no longer depends on the mesh size inside the locally refined region. 
Then, we derive an error equation and prove optimal error estimates in the $L^2$-norm.
Finally, in Section 4, we illustrate the optimal stability
and convergence properties of the stabilized LF-LTS method and also compare it to 
the original LF-LTS approach from \cite{DiazGrote09}.

We note that the idea of replacing standard Chebyshev polynomials by their stabilized version 
for added stability is well-known in the parabolic 
context  \cite{HundsVerwSommRKC,Verw96}. It was also recently used
to stabilize LFC methods \cite{CarHocStu19} and a Lagrange multiplier based
LTS approach \cite{ChabassierImperiale2019}.

\section{Galerkin Discretization with leapfrog Based Local
Time-Stepping\label{SecGalDisc}}

\subsection{The Wave Equation}

Let $\Omega\subset\mathbb{R}^{d}$ be a bounded Lipschitz domain with boundary
$\Gamma=\partial\Omega$ and assume there exists a partition $\Gamma=\Gamma
_{D}\cup\Gamma_{N}$, $\Gamma_{D}\cap\Gamma_{N}=\emptyset$, where $\left\vert
\Gamma_{D}\right\vert >0$. For $1\leq q\leq\infty$, let $L^{q}\left(
\Omega\right)  $ be the standard Lebesgue space with norm $\left\Vert
\cdot\right\Vert _{L^{q}\left(  \Omega\right)  }$. For $q=2$, the scalar
product in $L^{2}\left(  \Omega\right)  $ is denoted by $\left(  \cdot
,\cdot\right)  $ and the norm by $\left\Vert \cdot\right\Vert
:=\left\Vert \cdot\right\Vert _{L^{2}\left(  \Omega\right)  }$. For
$k\in\mathbb{N}_{0}$ and $1\leq q\leq\infty$, let $W^{k,q}\left(
\Omega\right)  $ denote the classical Sobolev spaces equipped with the norm
$\left\Vert \cdot\right\Vert _{W^{k,q}\left(  \Omega\right)  }$. For $q=2$,
these spaces are Hilbert spaces and denoted by $H^{k}\left(  \Omega\right)
:=W^{k,2}\left(  \Omega\right)  $ with scalar product $\left(  \cdot
,\cdot\right)  _{H^{k}\left(  \Omega\right)  }$ and norm $\left\Vert
\cdot\right\Vert _{H^{k}\left(  \Omega\right)  }:=\left(  \cdot,\cdot\right)
_{H^{k}\left(  \Omega\right)  }^{1/2}:=\left\Vert \cdot\right\Vert
_{W^{k,2}\left(  \Omega\right)  }$. 
The space of Sobolev functions in $H^{1}\left(  \Omega\right)  $
which vanish on the Dirichlet part of the boundary $\partial\Omega$ is denoted
by $H_{D}^{1}\left(  \Omega\right)  :=\left\{  w\in H^{1}\left(
\Omega\right)  :\left.  w\right\vert _{\Gamma_{D}}=0\right\}  $. For
$\Gamma_{D}=\Gamma$, we use the shorthand $H_{0}^{1}\left(  \Omega\right)  $.
Throughout this paper we restrict ourselves to
function spaces over the field of real numbers.

Let $V\subset H^{1}\left(  \Omega\right)  $ denote a closed subspace of
$H^{1}\left(  \Omega\right)  $, such as $V=H^{1}\left(  \Omega\right)  $ or
$V=H_{0}^{1}\left(  \Omega\right)  $, and $a:V\times
V\rightarrow\mathbb{R}$ denote a bilinear form, which is symmetric, continuous, and coercive:%
\begin{subequations}
\label{wellposed}
\end{subequations}%
\begin{equation}
a\left(  u,v\right)  =a\left(  v,u\right)  \qquad\forall u,v\in V 
\tag{\ref{wellposed}a}
\label{wellposeda}%
\end{equation}
and%
\begin{equation}
\left\vert a\left(  u,v\right)  \right\vert \leq C_{\operatorname*{cont}%
}\left\Vert u\right\Vert _{H^{1}\left(  \Omega\right)  }\left\Vert
v\right\Vert _{H^{1}\left(  \Omega\right)  }\qquad\forall u,v\in V 
\tag{\ref{wellposed}b}
\label{wellposedb}%
\end{equation}
and%
\begin{equation}
a\left(  u,u\right)  \geq c_{\operatorname*{coer}}\left\Vert u\right\Vert
_{H^{1}\left(  \Omega\right)  }^{2}\qquad\forall u\in V, 
\tag{\ref{wellposed}c}
\label{wellposedc}%
\end{equation}
with $0 < c_{\operatorname*{coer}} \leq C_{\operatorname*{cont}}$. 

We now consider the wave equation, where, for simplicity, we restrict ourselves to a homogeneous right-hand side:  For given $u_{0}\in V,v_{0}\in L^{2}\left(  \Omega\right)  $, find
$u:\left[  0,T\right]  \rightarrow V$ such that%
\begin{equation}
\left(  \ddot{u},w\right)  +a\left(  u,w\right)  = 0
\quad\forall w\in V,t>0 \label{waveeq}%
\end{equation}
with initial conditions%
\begin{equation}
u\left(  0\right)  =u_{0}\quad\text{and\quad}\dot{u}\left(  0\right)  =v_{0}.
\label{waveeqic}%
\end{equation}
It is well known that (\ref{waveeq})--(\ref{waveeqic}) is well-posed for
sufficiently regular $u_{0}$ and $v_{0}$  \cite{LionsMagenesI}. In fact,
the weak solution $u$ can be shown to be continuous in time, that is, $u\in
C^{0}(0,T;V) \cap C^{1}(0,T;L^{2}\left(  \Omega\right)  )$ -- see
[\cite{LionsMagenesI}, Chapter III, Theorems 8.1 and 8.2] for details -- which
implies that the initial conditions (\ref{waveeqic}) are well-defined.

\begin{example}
\label{Exmodel problem}The classical second-order wave equation in strong form
is given by%
\begin{equation}%
\begin{split}
\partial_t^2 u - \nabla\cdot(c^{2}\nabla u)  &  = 0 \qquad\;\,\mbox{in }\Omega
\times(0,T),\\
u  &  =0\qquad\;\,\mbox{on }\Gamma_{D}\times(0,T),\\
\frac{\partial u}{\partial\nu}  &  =0\qquad\;\,\mbox{on }\Gamma_{N}%
\times(0,T),\\
u|_{t=0}  &  =u_{0}\qquad\mbox{in }\Omega,\\
\partial_t u|_{t=0}  &  =v_{0}\qquad\mbox{in }\Omega,\\
&
\end{split}
\label{model problem}%
\end{equation}
where the velocity field $c(x)$ satisfies $0 < c_{\min} \leq c(x) \leq c_{\max}$.
In this case, we have $V:=H_{D}^{1}\left(  \Omega\right)  $ and $a\left(  u,v\right)  :=\left(  c^{2}\nabla u,\nabla v\right)  $.
\end{example}

\subsection{Galerkin Finite Element Discretization}

For the semi-discretization in space of (\ref{waveeq}), we shall employ the Galerkin finite element
method and thus first need to introduce some notation. For a spatial
dimension $d\in\left\{  1,2,3\right\}  $, we assume that the bounded Lipschitz domain
$\Omega\subset\mathbb{R}^{d}$ is an interval for $d=1$, a polygonal domain for
$d=2$, and a polyhedral domain for $d=3$. Now, let $\mathcal{T}:=\left\{  \tau
_{i}:1\leq i\leq N_{\mathcal{T}}\right\}  $ denote a conforming (i.e.: no
hanging nodes), simplicial finite element mesh for $\Omega$ with
\[
h_{\tau} := \operatorname*{diam}\tau,\quad h:=\max_{\tau\in\mathcal{T}}h_{\tau}%
\]
and $\rho_{\tau}$ the diameter of the largest inscribed ball in
$\tau$. Without loss of generality, we may assume that there is fixed $h_{0}>0$
such that%
\begin{equation}
h\leq h_{0}. \label{hsmallerh0}%
\end{equation}
Clearly, the choice $h_{0}=\operatorname*{diam}\Omega$ is always possible. As a
convention, the simplices $\tau\in\mathcal{T}$ are closed sets. The shape
regularity constant $\gamma$ of the mesh $\mathcal{T}$ is defined by%
\[
\gamma\left(  \mathcal{T}\right)  := \left\{
\begin{array}[c]{ll}%
1,  & d=1,\\
\max\limits_{\tau\in\mathcal{T}}\left\{\dfrac{h_{\tau}}{\rho_{\tau}}\right\}, & d=2,3.
\end{array}
\right.
\]

Let $\mathbb{P}_{m}$ denote the space of $d$-variate polynomials of maximal
total degree $m$. For a subset $\omega\subset\mathbb{R}^{d}$, we write
$\mathbb{P}_{m}\left(  \omega\right)  $ for $\left.  \mathbb{P}_{m}\right\vert
_{\omega}$. For $m\in\mathbb{N}$, we define the standard continuous, piecewise
polynomial finite element space by%
\[
S_{\mathcal{T}}^{m}:=\left\{  u\in C^{0}\left(  \Omega\right)  \mid\forall
\tau\in\mathcal{T}:\left.  u\right\vert _{\tau}\in\mathbb{P}_{m}\left(
\tau\right)  \right\}  
\]
and the space of piecewise constant functions by%
\[
S_{\mathcal{T}}^{0}:=\left\{  u\in L^{1}\left(  \Omega\right)  \mid\forall
\tau\in\mathcal{T}:\left.  u\right\vert _{\tau}\in\mathbb{P}_{0}\left(
\tau\right)  \right\}  .
\]
Although functions in $S_{\mathcal{T}}^{m}$ are $H^{1}$-conforming, they do not satisfy
any boundary conditions. 

The finite element space $S$, which we will use for the spatial Galerkin discretization, may equal $S_{\mathcal{T}}^{m}$ or
lie between two of the spaces
$S_{\mathcal{T}}^{m}$. To be more precise, let%
\[
\hat{\tau}:=\left\{  \mathbf{x}=\left(  x_{i}\right)  _{i=1}^{d}\in
\mathbb{R}_{\geq0}^{d}:\sum_{i=1}^{d}x_{i}\leq1\right\}
\]
be the reference simplex and
$\phi_{\tau}:\widehat{\tau}\rightarrow\tau$ denote an affine pullback for $\tau\in\mathcal{T}$. We
assume that the space $S$ is given by%
\begin{subequations}
\label{defSnew}
\end{subequations}%
\begin{equation}
S=\left\{  u\in V\mid\forall\tau\in\mathcal{T}:u\circ\phi_{\tau}\in P\right\}, 
\tag{\ref{defSnew}a}
\label{defSnewa}%
\end{equation}
for some fixed polynomial space $P$ which satisfies%
\begin{equation}
\mathbb{P}_{m}\left(  \hat{\tau}\right)  \subset P\subset\mathbb{P}_{m^{\prime}}\left(  \hat{\tau}\right)  
\tag{\ref{defSnew}b}
\label{defSnewb}%
\end{equation}
for a maximal integer $m\geq1$ and a minimal integer $m^{\prime}$. As a consequence, $S_{\mathcal{T}}^{m}\cap V\subset S\subset
S_{\mathcal{T}}^{m^{\prime}}\cap V$ obviously holds.

Then, the semi-discrete wave equation is given by: find $U_{S}:\left[
0,T\right]  \rightarrow S$ such that%
\begin{subequations}
\label{spacediscL2}
\end{subequations}%
\begin{equation}
\left(  \ddot{U}_{S},v\right)  +a\left(  U_{S},v\right)  =0\quad\forall v\in
S,\, t>0. 
\tag{\ref{spacediscL2}a}
\label{spacediscL2a}%
\end{equation}
To formulate initial conditions for $U_{S}$, the given data $u_{0}$, $v_{0}$ in (\ref{model problem}) must be mapped to the finite element space $S$. 
For any integer $s$ with $2\leq s\leq m+1$, any $0\leq\mu\leq s$, and any real number
$q$ with $2\leq q\leq\infty$, we assume there exists a bounded linear operator
$r_{S}:W^{s,q}\left(  \Omega\right)  \rightarrow S$ with the property:%
\begin{equation}
\left(  \sum_{K\in\mathcal{T}}\left\Vert r_{S}v-v\right\Vert _{W^{\mu
,q}\left(  K\right)  }^{q}\right)  ^{1/q}\leq Ch^{s-\mu}\left\Vert
v\right\Vert _{W^{s,q}\left(  \Omega\right)  }\qquad\forall v\in
W^{s,q}\left(  \Omega\right)  \cap V,\label{interpolationestimate}%
\end{equation}
for a constant $C$ independent of $v$ and $h$.

Once the operator $r_{S}$ is chosen, the initial conditions read%
\begin{equation}
U_{S}\left(  0\right)  =r_{S}u_{0}\qquad\text{and\qquad}\dot{U}_{S}\left(
0\right)  =r_{S}v_{0}
\tag{\ref{spacediscL2}b}
\label{spacediscL2b}%
\end{equation}
and we henceforth assume that $u_{0}, v_{0} \in 
H^{2}\left(  \Omega\right)  $.

\begin{example}
Since the space dimension satisfies $d\in\left\{  1,2,3\right\}  $, Sobolev's
embedding theorem implies $H^{2}\left(  \Omega\right)  \hookrightarrow
C^{0}\left(  \overline{\Omega}\right)  $ so that $r_{S}$ can be chosen as the
standard nodal interpolant into $S_{\mathcal{T}}^{m}\cap V$ (cf. \cite[Thm. 3.1.5]{CiarletPb}). The inclusion $S_{\mathcal{T}}^{m}\cap V\subset
S$ implies that $r_{S}$ also maps into $S$ and from \cite[Thm. 3.1.5]{CiarletPb}, we conclude that (\ref{interpolationestimate}) holds.
\end{example}

\subsection{Mass-Lumping\label{RemMasslumping}}

The basis representation of the $L^{2}$ scalar product in (\ref{spacediscL2a})
leads to a matrix, which is usually called the \emph{mass matrix}. 
For fully explicit time integration, it is crucial to replace this matrix by a \emph{lumped matrix},
which is positive definite, diagonal and preserves the optimal rates of convergence; it can also be interpreted as a
quadrature approximation of the local $L^{2}$ scalar products $\left(
\cdot,\cdot\right)  _{L^{2}\left(  \tau\right)  }$.

For $S$, we choose a local nodal basis $b_{z}$, $z\in\Sigma$, where
$\Sigma$ is a set of discrete nodal points in $\overline{\Omega}$
which serve as counting indices for the basis, such that%
\begin{equation}
\operatorname*{supp} b_{z} \subset \bigcup \left\{  \tau \in \mathcal{T}: z \in \tau \right\}  . 
\tag{\ref{defSnew}c}
\label{defSnewc}%
\end{equation}
Hence, every function $u\in S$ has the (unique) representation%
\begin{equation}
u=\sum_{z\in\Sigma}u_{z}b_{z}, \qquad u_{z} \in \mathbb{R}
\label{ubasisrep}%
\end{equation}
for some coefficient vector $\mathbf{u}=\left(  u_{z}\right)  _{z\in\Sigma}$.
By convention, $u\in S$ and $\left(  u_{z}\right)
_{z\in\Sigma}$ are related by
(\ref{ubasisrep}) whenever they appear in the same context.

\begin{definition}
Let $S$ satisfy (\ref{defSnew}). The space $S$ spanned by the basis $\left(
b_{z}\right)  _{z\in\Sigma}$ is \emph{mass-lumping admissible}, if there exists a
positive definite diagonal matrix $\mathbf{D}_{\Sigma}=\operatorname*{diag}%
\left[  d_{z}:z\in\Sigma\right]  $ such that the scalar product%
\begin{equation}
\left(  u,v\right)  _{\mathcal{T}}:=\sum_{z\in\Sigma}d_{z}u_{z}v_{z}%
\quad\forall u,v\in S
\label{uvtscalarp}
\end{equation}
induces a norm on $S$,
\begin{equation}
\left\Vert u\right\Vert _{\mathcal{T}}:=\left(  u,u\right)  _{\mathcal{T}%
}^{1/2}, \label{utnorm}%
\end{equation}
which is equivalent to the $L^{2}\left(  \Omega\right)  $-norm: there exists a
constant $C_{\operatorname*{eq}}>0$ such that%
\begin{equation}
C_{\operatorname*{eq}}^{-1}\left\Vert u\right\Vert _{\mathcal{T}}%
\leq\left\Vert u\right\Vert \leq C_{\operatorname*{eq}}\left\Vert u\right\Vert
_{\mathcal{T}}\quad\forall u\in S\text{.} \label{ceqCeq}%
\end{equation}
In addition, we require that the mass-lumped scalar product is exact up to degree
$m+m^{\prime}-2$, or more precisely, that it satisfies (cf. \cite[Cond. $\mathcal{A}$,
after (4.2)]{mass_lumping_2d})%
\begin{equation}
\left(  u,v \right)_{\mathcal{T}} = \left(  u,v\right)  
\quad \forall u \in S_{\mathcal{T}}^{k},
\, \forall v\in S_{\mathcal{T}}^{k^{\prime}},%
\; k+k^{\prime} \leq \max \left\{ m + m^{\prime} - 2, m^{\prime}+1 \right\}.
\label{ml_exactnesscondition}
\end{equation}

\end{definition}

We also associate to the restricted bilinear form the linear operator $A^{S}:S\rightarrow S$ defined by%
\begin{equation}
\left(  A^{S}u,v \right)_{\mathcal{T}} = a\left( u,v \right) \qquad\forall u,v\in S.
\label{defAS}
\end{equation}

\begin{remark}
In the one-dimensional case $d=1$, it is well known that the choice $S=S_{\mathcal{T}}^{m}\cap V$ is mass-lumping admissible, if $\Sigma$ corresponds to the Gauss-Lobatto quadrature points and $\mathbf{D}_{\Sigma}$ to the resulting quadrature approximation on each interval of the mesh \cite{COHEN}.

In the two-dimensional case $d=2$, the lowest order choice $S=S_{\mathcal{T}}^{1}\cap V$ is also mass-lumping admissible, if $\Sigma$ corresponds to the set of all triangle vertices except for those on the Dirichlet part of the boundary.
Then, the lumped mass matrix can be obtained via a quadrature scheme which employs the vertices of the triangles as quadrature points. 
For $m=2$ and $m=3$, the space $S_{\mathcal{T}}^{m}\cap V$ is not mass-lumping admissible for quadrature schemes defined on single triangles.
However, the space can be enriched by certain "bubble functions" in $S_{\mathcal{T}}^{m^{\prime}}$ with $m^{\prime} = m+1$ to obtain a finite element space $S$ and a quadrature scheme which is mass-lumping admissible \cite{mass_lumping_2d}. 
Alternatively, for $m=2$, the \textquotedblleft hybrid\textquotedblright\ quadrature rules as in \cite[Sect. 4]{GHS2} can be used for the standard space $S_{\mathcal{T}}^{m}\cap V$ provided the mesh satisfies certain connectivity assumptions on its graph.

In the three-dimensional case $d=3$, again the lowest order choice $S=S_{\mathcal{T}}^{1}\cap V$ is inherently mass-lumping admissible while for $m=2,3,4$, 
mass-lumping for enriched finite element spaces are discussed in \cite{mass_lumping3D}.
\end{remark}

Since local time-stepping methods use different time-steps in different parts of the domain, 
we now partition the finite element mesh $\mathcal{T}$ into a quasi-uniform (or \emph{coarse}) part $\mathcal{T}_{\operatorname*{c}}$ of mesh size
$h_{\operatorname*{c}}:=h$ and a locally refined (or \emph{fine}) part $\mathcal{T}_{\operatorname*{f}}:=\mathcal{T}\backslash\mathcal{T}_{\operatorname*{c}}$ of
mesh size $h_{\operatorname*{f}}:=\max_{\tau\in\mathcal{T}_{\operatorname*{f}}}h_{\tau}$, 
i.e., $\mathcal{T}=\mathcal{T}_{\operatorname*{c}}\cup \mathcal{T}_{\operatorname*{f}}$ and $\mathcal{T}_{\operatorname*{c}} \cap \mathcal{T}_{\operatorname*{f}}=\emptyset$. 
Similarly, we define coarse and locally
refined regions of the domain\footnote{By $\operatorname*{int}\left(M\right)  $ we denote the open interior of a subset $M\subset\mathbb{R}^{d}$.}, see Fig. \ref{FigMeshNotation}, as
\[
\Omega_{\operatorname*{c}} := \operatorname*{int}\left( {\textstyle\bigcup\nolimits_{\tau\in\mathcal{T}_{\operatorname*{c}}}} \tau\right)  , 
\quad 
\Omega_{\operatorname*{f}} := \operatorname*{int}\left( {\textstyle\bigcup\nolimits_{\tau\in\mathcal{T}_{\operatorname*{f}}}} \tau\right) , \quad 
\Omega_{\operatorname*{f}}^{+} := \operatorname*{int}\left( \Omega_{\operatorname*{f}}\cup\left( {\textstyle\bigcup\nolimits_{\tau\in\mathcal{T}_{\operatorname*{c}}, \tau\cap\overline{\Omega_{\operatorname*{f}}}\neq\emptyset}} \,\tau \right)  \right)  .
\]
Hence, $\Omega_{\operatorname*{f}}^{+} $ contains $\Omega_{\operatorname*{f}}$ and all elements 
directly adjacent to it. 
We note that $\Omega_{\operatorname*{f}}$ and $\Omega_{\operatorname*{c}}$ are disjoint, while their union covers all of $\Omega$ except for the coarse/fine interface, and that $\Omega_{\operatorname*{c}} \cup\Omega_{\operatorname*{f}}^{+} = \Omega$,
since the interface between $\Omega_{\operatorname*{f}}$ and $\Omega_{\operatorname*{c}}$ is contained in $\Omega_{\operatorname*{f}}^{+} $.

We also split the degrees of freedom associated with the fine or coarse parts of the mesh, respectively, as
\[
\Sigma_{\operatorname*{f}} := \Sigma\cap\overline{\Omega_{\operatorname*{f}}}
\quad\text{and\quad}
\Sigma_{\operatorname*{c}} := \Sigma\backslash\Sigma_{\operatorname*{f}},
\]
and introduce the corresponding FE subspaces 
\[
S_{\operatorname*{c}%
}:=\operatorname{span}\left\{  b_{z}:z\in\Sigma_{\operatorname*{c}}\right\}
\quad\mbox{and}\quad S_{\operatorname*{f}}:=\operatorname{span}\left\{  b_{z}:z\in
\Sigma_{\operatorname*{f}}\right\} . 
\] 
Since the support of functions
in $S_{\operatorname*{c}}$ is contained in $\overline{\Omega_{\operatorname*{c}}}$ and the support of functions in $S_{\operatorname*{f}}$ is contained in $\overline{\Omega_{\operatorname*{f}}^{+}}$ due to (\ref{defSnewc}), we have the direct sum decomposition: 
for every $u\in S$ there exist unique $u_{\operatorname*{c}}\in S_{\operatorname*{c}}$ and $u_{\operatorname*{f}}\in S_{\operatorname*{f}}$ such that
\begin{equation}
u=u_{\operatorname*{c}} + u_{\operatorname*{f}}. 
\label{uadddecomp}
\end{equation}
Hence, we can define the projections $\Pi_{\operatorname*{c}}^{S}: S \rightarrow S_{\operatorname*{c}}$ 
and $\Pi_{\operatorname*{f}}^{S}:S \rightarrow S_{\operatorname*{f}}$ by
\begin{equation}
\Pi_{\operatorname*{c}}^{S}u:=u_{\operatorname*{c}}\quad\text{and}\quad
\Pi_{\operatorname*{f}}^{S}u:=u_{\operatorname*{f}}. 
\label{defPifPic}%
\end{equation}
In fact, the decomposition (\ref{uadddecomp}) is
orthogonal with respect to the $\left(  \cdot,\cdot\right)  _{\mathcal{T}}$
scalar product -- see Lemma
\ref{LemEstPif} below --, which will be essential to prove sharp bounds for the eigenvalues of our discrete bilinearform in Theorem \ref{Coreigsys}.

\begin{figure}
\centering
\includegraphics[width=0.7\textwidth]{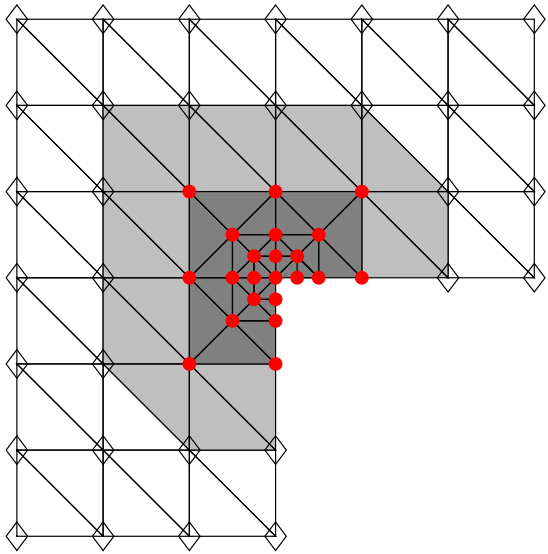}
\caption{A locally refined mesh for $d=2$ and $m=1$. 
The white and light grey triangles belong to $\Omega_{\operatorname*{c}}$, whereas the dark grey triangles belong to $\Omega_{\operatorname*{f}}$; 
$\Omega^{+}_{\operatorname*{f}}$ consists of both dark and light grey triangles. 
Vertices associated to $\Sigma_{\operatorname*{c}}$ are marked with black diamonds, 
vertices associated to  $\Sigma_{\operatorname*{f}}$ with red dots.}
\label{FigMeshNotation}
\end{figure}

With the definitions ans notations introduced above, the semi-discrete wave equation \emph{with mass-lumping} then is given by:
find $u_{S}:\left[  0,T\right]  \rightarrow S$ such that%
\begin{subequations}
\label{spacedisc}
\end{subequations}%
\begin{equation}
\left(  \ddot{u}_{S},v\right)_{\mathcal{T}} + a\left(  u_{S},v\right)  = 0
\quad\forall v\in S,t>0 
\tag{%
\ref{spacedisc}%
a}
\label{spacedisca}%
\end{equation}
with initial conditions%
\begin{equation}
u_{S}\left(  0\right)  = r_{S} u_{0} 
\quad \text{and\quad} 
\dot{u}_{S}\left(0\right)  = r_{S} v_{0}.
\tag{%
\ref{spacedisc}%
b}
\label{spacediscb}%
\end{equation}

\subsection{Stabilized LF-LTS Galerkin FE Formulation}
\label{sec:stablflts}
In \cite{DiazGrote09}, LF based local time-stepping (LTS) methods of arbitrarily high order for wave equations were proposed.
Optimal convergence rates for the original second order LF-LTS method were recently 
derived for a conforming FEM discretization, albeit under a CFL condition
where the global time-step $\Delta t$ in fact depends on the smallest elements in the mesh \cite{grote_sauter_1}. 

We now present the fully discrete space-time Galerkin FE formulation of the stabilized leapfrog based local time-stepping LF-LTS scheme. 
To do so, let the (global) time-step $\Delta t=T/N$ and $u_{S}^{\left(  n\right)  }$ denote the FE approximation at time $t_{n}=n\Delta t$.
Given the numerical solution $u_S^{(n-1)}$ and $u_S^{(n)}$ at times $t_{n-1}$ and $t_{n}$, 
the LF-LTS method then computes the numerical solution $u_{S}^{(n+1)}$ of \eqref{spacedisc} at $t_{n+1}$ by using a smaller time-step $\Delta\tau=\Delta t/p$ inside $\Omega_{\operatorname*{f}}^{+}$, 
where $p\in\mathbb{N}$ denotes the ``coarse'' to ``fine'' time-step ratio. 
Clearly, if the maximal velocity in the coarse and the fine regions differ significantly, the choice of $p$ should also reflect that variation and instead denote the local CFL number ratio. 

For the definition of the stabilized LF-LTS method, we will employ stabilized Chebyshev polynomials  \cite{HV03}, also used in \cite{CarHocStu19} for the LFC methods. They are based on 
Chebyshev polynomials of the first kind \cite{Rivlin}, denoted by $T_{p}$, and a (small) stabilization parameter $0 \leq \nu \leq 1/2$.
The smaller $\nu$, the smaller the amount of stabilization added to the LTS method.
For $\nu = 0$, no stabilization is applied and the LTS method reduces to the original scheme from \cite{DiazGrote09,grote_sauter_1}.

Next, let $P_{p,\nu}\in\mathbb{P}_{p}$ denote the polynomial
\begin{equation}
P_{p,\nu}\left( x \right)  := 2 \left( 1 - \frac{ T_{p} \left( \delta_{p,\nu} - \frac{x}{\omega_{p,\nu}} \right) }{ T_{p} \left( \delta_{p,\nu} \right) } \right)  
\quad\text{with }
\left\{
\begin{aligned}
\omega_{p,\nu} &:= 2 \,\frac{ T_{p}^{\prime} \left(  \delta_{p,\nu} \right) }{ T_{p} \left( \delta_{p,\nu} \right) },\\
\delta_{p,\nu} &:= 1+\nu/p^{2},
\end{aligned}
\right.  
\label{defpolys1}%
\end{equation}
which correspond to shifted and rescaled \emph{stabilized (aka damped) Chebyshev} polynomials \cite[(2.3c)]{CarHocStu19}.
Since $P_{p,\nu}\left( 0 \right) = 0$ and $P_{p,\nu}^{\prime}\left( 0 \right) = 1$, we may rewrite $P_{p,\nu}$ as%
\[
P_{p,\nu}\left( x \right) = x \left( 1 - \frac{2}{p^{2}} \sum_{j=1}^{p-1} \alpha_{j}^{p,\nu} \left( \frac{x}{p^{2}} \right)^{j} \right)  
\]
for some coefficients $\alpha_{j}^{p,\nu}$. As the coefficients $\alpha_{j}^{p,\nu}$ are never needed for the algorithm, we omit an explicit representation here.
We also define the reduced polynomials $P_{p,\nu}^{\Delta t}, Q_{p,\nu}^{\Delta t} \in \mathbb{P}_{p-1}$ by
\begin{equation}
P_{p,\nu}^{\Delta t}\left(  x\right)  := 
\frac{P_{p,\nu}\left(  \Delta t^{2}x\right)}{\Delta t^{2}x}
\quad\text{and\quad}
Q_{p,\nu}^{\Delta t}\left(  x\right)  := 1 - P_{p,\nu}^{\Delta t}\left(  x\right).
\label{defQpnu}%
\end{equation}
Note that the polynomials $Q_{p,\nu}^{\Delta t}$, $P_{p,\nu}^{\Delta t}$ can be applied to any linear operator which is an endomorphism. 

In Fig. \ref{FigPpnue}, we display the polynomials $P_{p,\nu}\left(p^2 x\right)$ for $x \in [0,4]$ and different values of $p$ and $\nu$. 
For $\nu = 0$, we observe that their values also lie within $[0,4]$, with $p-1$ points of tangency at the upper or lower bounds. 
In contrast for $\nu=0.3$, the values of $P_{p,\nu}\left(p^2 x\right)$ stay strictly inside $(0,4)$ for $x\in [0,4-\epsilon]$, for some $\epsilon > 0$.
\begin{figure}
\centering
\includegraphics[width = 0.49\textwidth]{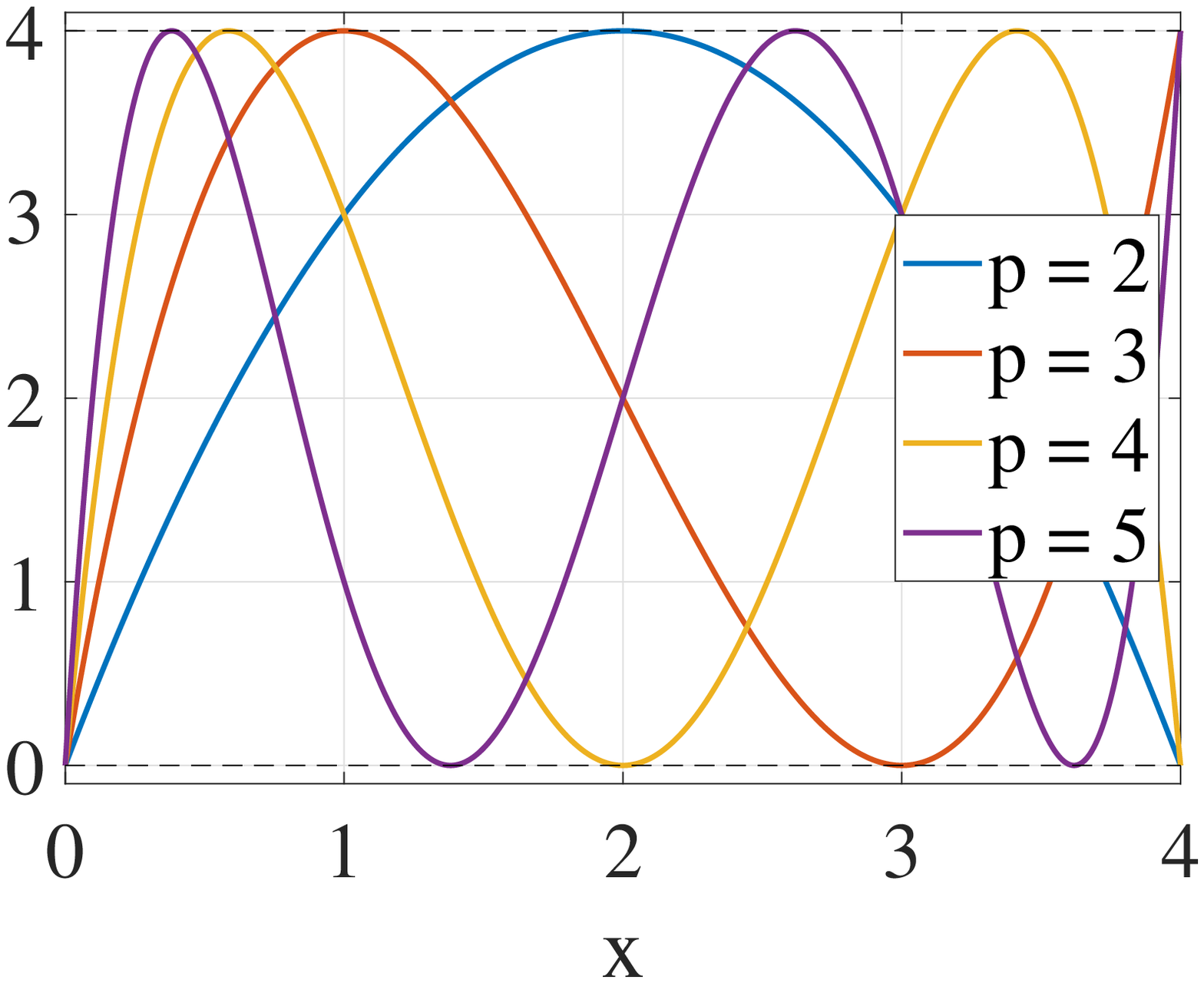}
\includegraphics[width = 0.49\textwidth]{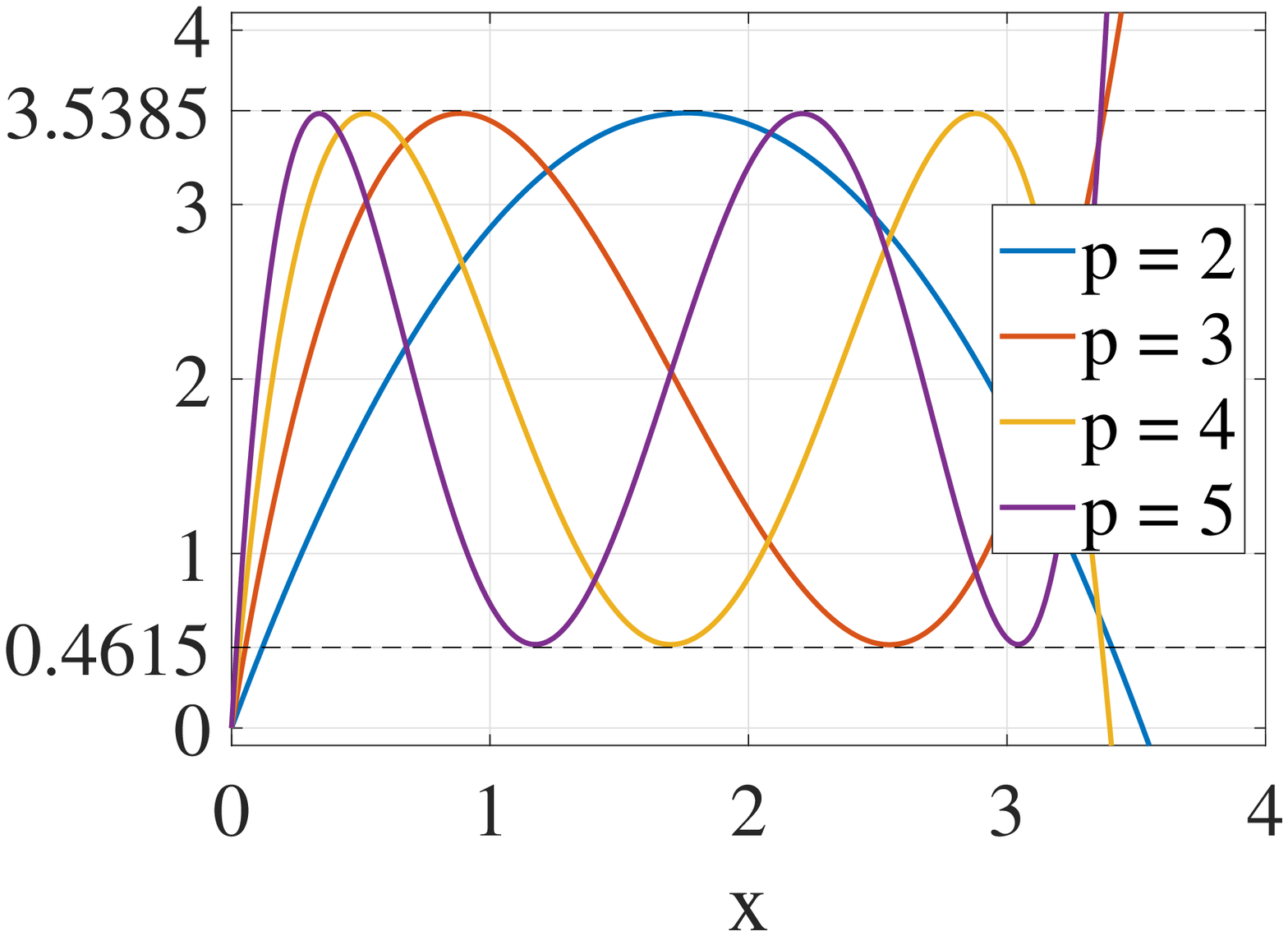}
\caption{Graphs of $P_{p,\nu}\left(p^2 x\right)$ for different $p$ and $\nu=0$ (left) or $\nu=0.3$ (right). Dashed lines indicate the bounds from Lemma \ref{LemPpnue}.}
\label{FigPpnue}
\end{figure}

Now, let $a^{p,\nu}:S\times S\rightarrow\mathbb{R}$  denote the stabilized discrete bilinear form
\begin{equation}
a^{p,\nu} \left( u,v \right) := a\left( u - Q_{p,\nu}^{\Delta t} \left( \Pi_{\operatorname*{f}}^{S} A^{S} \right) u,v \right)  \qquad \forall u,v\in S
\label{defapv}%
\end{equation}
with associated (linear) operator $A^{S,p,\nu}: S \rightarrow S$%
\begin{align}
A^{S,p,\nu}  &  :=A^{S}\left(  I-Q_{p,\nu}^{\Delta t}\left(  \Pi
_{\operatorname*{f}}^{S}A^{S}\right)  \right)  =A^{S}P_{p,\nu}^{\Delta
t}\left(  \Pi_{\operatorname*{f}}^{S}A^{S}\right) \label{defASp}\\
&  =\left(  A^{S}\right)  ^{1/2}P_{p,\nu}^{\Delta t}\left(  \left(
A^{S}\right)  ^{1/2}\Pi_{\operatorname*{f}}^{S}\left(  A^{S}\right)
^{1/2}\right)  \left(  A^{S}\right)  ^{1/2}.\nonumber
\end{align}

The stabilized LF-LTS method is then defined for $n\geq 1$ as
\begin{equation}%
\begin{aligned}
\left( u_{S}^{\left(  n+1\right)} - 2 u_{S}^{\left(  n\right)  } + u_{S}^{\left(  n-1\right)},w \right)  _{\mathcal{T}} 
&= -\Delta t^{2} a^{p,\nu}\left( u_{S}^{\left( n\right) },w \right)   &  &\forall w\in S,\\
u_{S}^{\left(  0\right)  } &= r_S u_{0}\\
u_{S}^{\left(  1\right)  } &= r_S u_{0} + \Delta t \, r_S v_{0} - \frac{\Delta t^{2}}{2} A^{S}  u_{0}
\end{aligned}
\label{leap_frog_lts_fd}%
\end{equation}
The same algorithm was also developed independently and in parallel by the group of M. Hochbruck \cite{CarHoc}.
For $p=1$, the bilinear forms $a^{p,\nu}$ and $a$ coincide because $\omega_{1,\nu} = 2 / \delta_{1,\nu}$ and therefore $Q_{1,\nu}^{\Delta t}(x)$ is identically zero; 
then (\ref{leap_frog_lts_fd}) corresponds to the standard leapfrog scheme.

\begin{remark}
In (\ref{leap_frog_lts_fd}), the term $A^S u_{0}$ in the third equation could be replaced by $A^{S,p,\nu}  u_{0}$
thereby allowing for local time-stepping already during the very first time-step. 
In that case, the analysis below also applies but requires a minor change, namely, replacing
$A^{S}$ by $A^{S,p,\nu}$ in (\ref{initerr1}) and (\ref{initerr2}) below. 
This modification neither affects the stability nor the convergence rate of the overall stabilized LF-LTS scheme.
\end{remark}
Given the constants
\begin{equation}
\beta_k := \frac{T_{k-1}\left( \delta_{p,\nu} \right)}{T_{k+1}\left( \delta_{p,\nu} \right)} 
\quad \text{and} \quad 
\beta_{k+1/2} := \frac{ T_{k}\left( \delta_{p,\nu} \right) }{ T_{k+1}\left( \delta_{p,\nu} \right) }, \quad 1 \leq k \leq p-1,
\label{Def_beta}
\end{equation}
we list the full stabilized LF-LTS algorithm.
Note that $\beta_k$ and $\beta_{k+1/2}$ only need to be computed once for any particular choice of $p$ and $\nu$.

\begin{algorithm}
Let $n \geq 1$.
\begin{enumerate}
\item Given $u_S^{(n-1)}$, $u_S^{(n)}$, set $z_{S,0}^{(n)} := u_S^{(n)}$ and compute $w_S^{(n)}$ as
\[
w_S^{(n)} = A^{S} \Pi_{\operatorname*{c}}^{S} z_{S,0}^{(n)}.
\]
\item Compute
\[
z_{S,1}^{(n)} = z_{S,0}^{(n)} - 
\frac{1}{2} \left( \frac{\Delta t}{p} \right)^2 \left( \frac{2 p^2}{\omega_{p,\nu} \delta_{p,\nu} } \right) 
\left( w_S^{(n)} + A^{S} \Pi_{\operatorname*{f}}^{S} z_{S,0}^{(n)} \right).
\]
\item For $k = 1,\ldots,p-1,$ compute
\begin{align*}
z_{S,k+1}^{(n)} &= \left( 1 + \beta_{k} \right) z_{S,k}^{(n)} - \beta_{k} z_{S,k-1}^{(n)} \\
&\quad - \left( \frac{\Delta t}{p} \right)^2 \left( \frac{2 p^2}{\omega_{p,\nu}} \right) \beta_{k+ 1/2} 
\left( w_{S}^{(n)} + A^{S} \Pi_{\operatorname*{f}}^{S} z_{S,k}^{(n)} \right).
\end{align*}
\item Compute
\[
u_{S}^{(n+1)} = - u_{S}^{(n-1)} + 2 z_{S,p}^{(n)}.
\]
\end{enumerate}
\end{algorithm}

The above algorithm is equivalent to (\ref{leap_frog_lts_fd}) -- see Theorem \ref{TheoEquivLTSAlgo} in Appendix \ref{secAppB}.

If the target space $S_{\operatorname*{f}}$ of $\Pi_{\operatorname*{f}}$ is relatively small, the overall cost will be dominated by the computation of $w_S^{(n)}$, which requires a single operation with $A^{S} \Pi_{\operatorname*{c}}^{S}$ per time-step $\Delta t$. 
All further operations with $A^{S} \Pi_{\operatorname*{f}}^{S}$ involve only those unknowns that are associated with $\Omega_{\operatorname*{f}}^{+}$.

Inside $\Omega_{\operatorname*{c}}$, away from the coarse/fine interface, the above algorithm reduces to the standard LF method with time step $\Delta t$, regardless of $p$ or $\nu$. 
Indeed, those degrees of freedom inside $\Omega_{\operatorname*{c}}$
are not affected by any terms involving  $A^{S} \Pi_{\operatorname*{f}}^{S} z_{S,k}^{(n)} $. By summing 
up all local time-steps $k = 1, \dots, p-1$ and applying standard Chebyshev recursions, the update
inside $\Omega_{\operatorname*{c}}$ reduces to 
\[
z_{S,p}^{(n)} = u_S^{(n)} - \frac{\Delta t^2}{2} A^S u_S^{(n)}.
\]
Hence when combined with Step 4 in the above algorithm, the update of the numerical solution inside $\Omega_{\operatorname*{c}}$ simply corresponds to a standard LF step of size $\Delta t$.

Moreover, the above notation immediately translates into the discrete matrix-vector-type notation previously used in \cite[section 3.1]{DiazGrote09} by recognizing that the operators $A^{S}$, $\Pi_{\operatorname*{f}}^{S}$ correspond to the matrices $A$ and $P$ used there. 
In \cite{DiazGrote09}, the diagonal matrix $P$, with entries equal to zero or one, identifies the unknowns associated to the locally refined region; 
hence, the projection $\Pi_{\operatorname*{c}}^{S}$ corresponds to the matrix $I - P$. 

\begin{remark}
For $\nu=0$, the above algorithm coincides with the original LF-LTS algorithm from \cite{DiazGrote09,grote_sauter_1}, since $\delta_{p,0} = 1$, and therefore $\beta_k = \beta_{k+1/2} = 1$ and $\omega_{p,0} = 2p^2$ because $T_p^{\prime}(1) = p^2$, see \cite{Rivlin}.
\end{remark}

\begin{remark}
\label{RemarkLFCLTS}
If $\Pi^S_{\operatorname*{f}}$ corresponds to the identity, i.e. a smaller time-step is applied everywhere, the above algorithm reduces to the leapfrog-Chebyshev (LFC) formulation from \cite[Corollary 3]{GilbertJoly} for $\nu=0$ and to the stabilized LFC method \cite{CarHocStu19} for $\nu > 0$. In contrast to 
the LFC method from \cite{CarHocStu19}, which uses the same stabilized Chebyshev polynomials in its formulation, the original derivation in \cite{DiazGrote09} of the LF-LTS Algorithm is based on the initial operator splitting $A^S = A^S \Pi^S_{\operatorname*{c}} + A^S \Pi^S_{\operatorname*{f}}$, where both terms 
are nonsymmetric and singular. 
However, when rewritten in leapfrog manner, the formulation (\ref{defapv})--(\ref{leap_frog_lts_fd}) yields a symmetric bilinear form $a^{p,\nu}$.
\end{remark}

\section{Stability and Convergence Analysis\label{SecStabCons}}

\subsection{Estimates of the Bilinearform\label{SecSetting}}

The time-steps $\Delta t$ and $\Delta\tau = \Delta t / p$, $p\in\mathbb{N}$, used in $\mathcal{T}_{\operatorname*{c}}$ and $\mathcal{T}_{\operatorname*{f}}$, respectively, are each determined by the smallest element in either part of the mesh.
Hence, the ``coarse''-to-``fine'' time-step ratio, $p$, must satisfy
\[
\dfrac{ \min\left\{  h_{\tau}:\tau\in\mathcal{T}_{\operatorname*{c}}\right\} }{h_{\tau}} \leq p \qquad\forall\tau\in\mathcal{T}.
\]
Given the quasi-uniformity constant for the \emph{coarse} mesh $\mathcal{T}_{\operatorname*{c}}$,
\[
C_{\operatorname*{qu},\operatorname*{c}}:=\frac{h_{\operatorname*{c}}}%
{\min\left\{  h_{\tau}:\tau\in\mathcal{T}_{\operatorname*{c}}\right\}  },
\]
the ratio between the largest and smallest elements of the mesh is thus bounded by
\begin{equation}
\dfrac{h_{\operatorname*{c}}} {h_{\tau}} \leq
C_{\operatorname*{qu},\operatorname*{c}} p \qquad\forall\tau\in\mathcal{T}. 
\label{hfphc}%
\end{equation}

\begin{lemma}
\label{Leminvineq}
Let (\ref{hsmallerh0}), (\ref{ceqCeq}), and (\ref{hfphc}) be
satisfied. There exists a constant $C_{\operatorname*{inv},\operatorname*{loc}%
}>0$, which only depends on $\gamma\left(  \mathcal{T}\right)  $ and $m^{\prime}$, such
that for all $\tau\in\mathcal{T}$%
\begin{equation}
\left\Vert \nabla u\right\Vert _{L^{2}\left(  \tau\right)  }\leq
C_{\operatorname*{inv},\operatorname*{loc}}h_{\tau}^{-1}\left\Vert
u\right\Vert _{L^{2}\left(  \tau\right)  },\qquad\forall u \in%
S_{\mathcal{T}}^{m^{\prime}}. \label{defcinv}%
\end{equation}
The global version for the coarse mesh part of $\mathcal{T}$ involves the
quasi-uniformity constant $C_{\operatorname*{qu},\operatorname*{c}}:$%
\begin{equation}
\left\Vert u\right\Vert _{H^{1}\left(  \Omega_{\operatorname*{c}}\right)
}\leq C_{\operatorname*{inv}} h_{\operatorname*{c}}%
^{-1}\left\Vert u\right\Vert _{\mathcal{T}}\quad\forall u\in S
\label{globalinvcoarse}%
\end{equation}
for $C_{\operatorname*{inv}} := C_{\operatorname*{eq}}%
\sqrt{h_{0}^{2}+C_{\operatorname*{qu},\operatorname*{c}}^{2}%
C_{\operatorname*{inv},\operatorname*{loc}}^{2}}$.

The global version on the entire domain is%
\begin{equation}
\left\Vert u\right\Vert _{H^{1}\left(  \Omega\right)  }\leq
C_{\operatorname*{inv}}\frac{p}{h_{\operatorname*{c}}}\left\Vert u\right\Vert
_{\mathcal{T}}\quad\forall u\in S. \label{globalinvinequ2}%
\end{equation}
\end{lemma}

\begin{proof}
It is well known that the functions in $S_{\mathcal{T}}^{m^{\prime}}$ satisfy the
inverse inequality (\ref{defcinv}) (for a proof we refer, e.g., \cite[(3.2.33)
with $m^{\prime}=1$, $q=r=2$, $l=0$, $n=d$.]{CiarletPb}\footnote{There is a misprint in
\cite{CiarletPb} $m-1$ should be replaced by $m-\ell$; see also \cite[(4.5.3)
Lemma]{scottbrenner3}.}).

For the global version on the coarse mesh, we note that
\begin{align*}
\left\Vert \nabla u\right\Vert _{L^{2}\left(  \Omega_{\operatorname*{c}%
}\right)  }^{2}  &  =\sum_{\tau\in\mathcal{T}_{\operatorname*{c}}}\left\Vert
\nabla u\right\Vert _{L^{2}\left(  \tau\right)  }^{2}\leq
C_{\operatorname*{inv},\operatorname*{loc}}^{2}\sum_{\tau\in\mathcal{T}%
_{\operatorname*{c}}}h_{\tau}^{-2}\left\Vert u\right\Vert _{L^{2}\left(
\tau\right)  }^{2}\\
&  \leq C_{\operatorname*{inv},\operatorname*{loc}}^{2}C_{\operatorname*{qu}%
,\operatorname*{c}}^{2}h_{\operatorname*{c}}^{-2}\sum_{\tau\in\mathcal{T}%
_{\operatorname*{c}}}\left\Vert u\right\Vert _{L^{2}\left(  \tau\right)  }%
^{2}=C_{\operatorname*{inv},\operatorname*{loc}}^{2}C_{\operatorname*{qu}%
,\operatorname*{c}}^{2}h_{\operatorname*{c}}^{-2}\left\Vert u\right\Vert
_{L^{2}\left(  \Omega_{\operatorname*{c}}\right)  }^{2}\\
&  \leq C_{\operatorname*{eq}}^{2}C_{\operatorname*{inv},\operatorname*{loc}%
}^{2}C_{\operatorname*{qu},\operatorname*{c}}^{2}h_{\operatorname*{c}}%
^{-2}\left\Vert u\right\Vert _{\mathcal{T}}^{2},
\end{align*}
which immediately implies (\ref{globalinvcoarse}).

For the bound on the entire domain, we obtain from (\ref{hfphc}), (\ref{defcinv}) that
\[
\begin{aligned}
\left\Vert \nabla u\right\Vert _{L^{2}\left(  \Omega\right)  }^{2} 
&\leq
C_{\operatorname*{inv},\operatorname*{loc}}^{2}\sum_{\tau\in\mathcal{T}%
}h_{\tau}^{-2}\left\Vert u\right\Vert _{L^{2}\left(  \tau\right)  }^{2} \\
&\leq \frac{C_{\operatorname*{inv},\operatorname*{loc}}^{2} C_{\operatorname*{qu}%
,\operatorname*{c}}^{2} p^{2}}%
{h_{\operatorname*{c}}^{2}%
}\left\Vert u\right\Vert ^{2} 
\leq C_{\operatorname*{eq}}^{2} \frac
{C_{\operatorname*{inv},\operatorname*{loc}}^{2} C_{\operatorname*{qu}%
,\operatorname*{c}}^{2} p^{2}}{h_{\operatorname*{c}}^{2}} \left\Vert u\right\Vert
_{\mathcal{T}}^{2}.
\end{aligned}
\]
Hence, since $p\geq 1$, we conclude that
\[
\left\Vert u\right\Vert _{H^{1}\left(  \Omega\right)  }
\leq
C_{\operatorname*{eq}} \sqrt{1+\frac
{C_{\operatorname*{inv},\operatorname*{loc}}^{2} C_{\operatorname*{qu}%
,\operatorname*{c}}^{2} p^{2}}{h_{\operatorname*{c}}^{2}}}\left\Vert u\right\Vert _{\mathcal{T}}
\leq
\frac{p}{h_{\operatorname*{c}}} C_{\operatorname*{eq}} \sqrt{h_{0}^{2}+ C_{\operatorname*{inv},\operatorname*{loc}}^{2} C_{\operatorname*{qu}%
,\operatorname*{c}}^{2} }
\left\Vert u\right\Vert _{\mathcal{T}}.
\]\end{proof}

\begin{lemma}
\label{LemEstPif}
The decomposition $u = u_{\operatorname*{c}} + u_{\operatorname*{f}}$ defined in (\ref{uadddecomp}), (\ref{defPifPic}) is orthogonal with respect to the $\left( \cdot,\cdot \right)_{\mathcal{T}}$ scalar product, i.e.%
\begin{equation}
\left\Vert \Pi^{S}_{\operatorname*{f}} u \right\Vert_{\mathcal{T}}^2
+ \left\Vert \Pi^{S}_{\operatorname*{c}} u \right\Vert_{\mathcal{T}}^2
= \left\Vert u \right\Vert_{\mathcal{T}}^2
\qquad\forall u \in S_{\mathcal{T}}^{m}. 
\label{Rfest}%
\end{equation}
Let (\ref{wellposedb}), (\ref{hsmallerh0}), (\ref{ceqCeq}), and (\ref{hfphc}) be satisfied. 
For $v \in S_{\mathcal{T}}^{m}$ it holds%
\[
\left\Vert A^{S}v\right\Vert _{\mathcal{T}}\leq C_{\operatorname*{cont}%
}C_{\operatorname*{inv}}\frac{p}{h_{\operatorname*{c}}}\left\Vert v\right\Vert
_{H^{1}\left(  \Omega\right)  }\leq C_{\operatorname*{cont}}%
C_{\operatorname*{inv}}^{2}\left(  \frac{p}{h_{\operatorname*{c}}}\right)
^{2}\left\Vert v\right\Vert _{\mathcal{T}}.
\]
The operator $\Pi^{S}_{\operatorname*{f}}$ is self-adjoint with respect to the $(\cdot,\cdot)_{\mathcal{T}}$ scalar product, positive semi-definite, and satisfies
\begin{equation}
\left\Vert \Pi^{S}_{\operatorname*{f}}A^{S}v\right\Vert _{\mathcal{T}}\leq
C_{\operatorname*{cont}}C_{\operatorname*{inv}}^{2}\left(  \frac
{p}{h_{\operatorname*{c}}}\right)  ^{2}\left\Vert v\right\Vert _{\mathcal{T}}.
\label{RfASest}%
\end{equation}
In addition, if (\ref{wellposedc}) holds, the operator $A^{S}$ satisfies the estimate%
\begin{equation}
\left(  A_{S} v,v \right)  _{\mathcal{T}}\geq c_{\operatorname*{coer}}\left\Vert
v \right\Vert _{H^{1}\left(  \Omega\right)  }^{2}\geq\frac
{c_{\operatorname*{coer}}}{C_{\operatorname*{eq}}^{2}}\left\Vert v \right\Vert_{\mathcal{T}}^{2}\qquad \forall v\in S. 
\label{lowevAs}%
\end{equation}
\end{lemma}

\begin{proof}
Orthogonality of the decomposition and the identity in (\ref{Rfest}) directly follow from definitions (\ref{uvtscalarp}), (\ref{utnorm}) of the scalar product $(\cdot,\cdot)_{\mathcal{T}}$ and the corresponding norm.

Using the definition of $A^S$ in (\ref{defAS}) and (\ref{wellposedb}), we then conclude that
\[
\left\Vert A^{S} v \right\Vert_{\mathcal{T}}
= \sup_{w \in S \setminus \left\{ 0 \right\}} \frac{a\left( v,w \right)}{\left\| w \right\|_{\mathcal{T}}}
\overset{(\ref{wellposedb})}{\leq} C_{\operatorname*{cont}} \left\Vert v \right\Vert_{H^{1}\left( \Omega \right)} 
\sup_{w\in S \setminus \left\{ 0 \right\}} \frac{ \left\Vert w \right\Vert_{H^{1}\left( \Omega \right)} }{ \left\Vert w \right\Vert_{\mathcal{T}} }.
\]
From the inverse inequality (\ref{globalinvinequ2}) applied to $w$, we thus obtain%
\[
\left\Vert A^{S}v\right\Vert _{\mathcal{T}}\leq C_{\operatorname*{cont}%
}C_{\operatorname*{inv}}\frac{p}{h_{\operatorname*{c}}}\left\Vert v\right\Vert
_{H^{1}\left(  \Omega\right)  }.
\]
Applying again (\ref{globalinvinequ2}), but now to $v$, yields the second inequality.
Its combination with (\ref{Rfest}) results in (\ref{RfASest}).
Next, since
\[
\left( \Pi_{\operatorname*{f}}^S u, v \right)_{\mathcal{T}}
= \left( u_{\operatorname*{f}}, v_{\operatorname*{c}} + v_{\operatorname*{f}} \right)_{\mathcal{T}}
= \left( u_{\operatorname*{f}}, v_{\operatorname*{f}} \right)_{\mathcal{T}}
= \left( u, \Pi_{\operatorname*{f}}^S v \right)_{\mathcal{T}},
\]
we conclude that $\Pi_{\operatorname*{f}}^S$ is self-adjoint.
Similarly, for any $u \in S$ we have
\[
\left( \Pi_{\operatorname*{f}}^{S} u, u \right)_{\mathcal{T}} 
= \left(u_{\operatorname*{f}}, u_{\operatorname*{c}} + u_{\operatorname*{f}} \right)_{\mathcal{T}} 
= \left( u_{\operatorname*{f}}, u_{\operatorname*{f}} \right) \geq 0,
\]
which implies that $\Pi_{\operatorname*{f}}^S$ is positive semi-definite.

For the lower estimate we employ (\ref{wellposedc}) and (\ref{ceqCeq}) to
obtain%
\[
\left(  A_{S}u,u\right)  _{\mathcal{T}}=a\left(  u,u\right)  \geq
c_{\operatorname*{coer}}\left\Vert u\right\Vert _{H^{1}\left(  \Omega\right)
}^{2}\geq c_{\operatorname*{coer}}\left\Vert u\right\Vert ^{2}\geq
\frac{c_{\operatorname*{coer}}}{C_{\operatorname*{eq}}^{2}}\left\Vert
u\right\Vert _{\mathcal{T}}^{2}.
\]
\end{proof}

\begin{remark}
Lemma \ref{Leminvineq} and \ref{LemEstPif} hold for any $m \geq 1$ with implied constant $C_{\operatorname*{inv}}$, $C_{\operatorname*{inv},\operatorname*{loc}}$, and $C_{\operatorname*{eq}}$. 
Via the inclusion $S\subset S_{\mathcal{T}}^{m^{\prime}}$, these estimates carry over to the finite element space $S$ as in (\ref{defSnew}).
\end{remark}

To derive theoretical properties about the bilinear form $a^{p,\nu}$, we shall require
various technical results about Chebyshev polynomials, some of them 
proved in \cite{CarHocStu19} and further estimates provided in Appendix \ref{AppendixA}. Moreover, we
assume the CFL-type stability condition on the ratio $\Delta t/h_{\operatorname*{c}}$:
\begin{equation}
\left(  3+\frac{C_{\operatorname*{cont}}}{c_{\operatorname*{coer}}}\right) 
C_{\operatorname*{cont}} C_{\operatorname*{inv}}^{2}
\left(  \frac{\Delta t}{h_{\operatorname*{c}}}\right)^{2}  
  \leq\frac{\nu}{\nu+1}, 
\label{CFLtotal}%
\end{equation}
with $\delta_{p,\nu}$, $\omega_{p,\nu}$ as in (\ref{defpolys1}).

\begin{remark}
Some of our theory can be proved under the weaker CFL condition
\begin{equation}
C_{\operatorname*{cont}} C_{\operatorname*{inv}}^{2}
\left( \frac{\Delta t}{h_{\operatorname*{c}}} \right)^{2}  
  \leq\left( 2+\nu/p^{2} \right) \frac{ \omega_{p,\nu} }{ p^{2} }.
\label{condcfl2}
\end{equation}
Indeed, from Lemma \ref{Lembound} it follows that
\begin{equation*}
\left(  2+\nu / p^2 \right)  2p^{2} \operatorname*{e}\nolimits^{-\nu} 
\leq
\left(  2+\nu/p^{2} \right)  \omega_{p,\nu}.
\end{equation*}
Hence, the upper bound in (\ref{condcfl2}) is bounded below independently of $p$ and (\ref{condcfl2}) is therefore implied by the
stronger condition%
\begin{equation}
C_{\operatorname*{cont}}C_{\operatorname*{inv}}^{2}\left(  \frac{\Delta
t}{h_{\operatorname*{c}}}\right)  ^{2}\leq4\operatorname*{e}\nolimits^{-\nu}. \label{CFLtotalc}%
\end{equation}
Since 
\[
\frac{\nu}{\nu+1} \leq 4 \operatorname*{e}\nolimits^{-\nu}
\]
for $0 \leq \nu \leq 1/2$, (\ref{CFLtotal}) always implies (\ref{CFLtotalc}) and hence (\ref{condcfl2}).
Both (\ref{condcfl2}) and (\ref{CFLtotal}) thus correspond to $p$ independent CFL stability conditions on $\Delta t$, with (\ref{condcfl2}) implied by (\ref{CFLtotal}).
In fact, the \emph{proof} of our current stability analysis deteriorates in the limit $\nu \to 0$. 
However, this limit is of no practical interest, in particular, since numerical experiments show that the stabilized method behaves in a robust manner with respect to small but fixed, postitive values of $\nu$.
Note also that our method for the case $\nu=0$ results in the original LF-LTS method, which is stable albeit under the $p$-dependent CFL condition presented in \cite{grote_sauter_1}.
\end{remark}

\begin{theorem}
\label{TheoCont} 
Let (\ref{wellposed}), (\ref{hsmallerh0}), (\ref{ceqCeq}), and (\ref{hfphc}) be satisfied. 
Assume that the \emph{CFL condition}  (\ref{condcfl2}) holds. 
Then, {the bilinear form $a^{p,\nu}\left( \cdot,\cdot \right)$ is continuous,
\[
\left\vert a^{p,\nu}\left(  u,v\right)  \right\vert \leq
C_{\operatorname*{cont}} 
\left\Vert u\right\Vert _{H^{1}\left(  \Omega\right)  }\left\Vert v\right\Vert
_{H^{1}\left(  \Omega\right)  } \quad\forall u,v\in S
\]
and symmetric.}
Moreover, it satisfies the coercivity estimate%
\[
a^{p,\nu}\left(  u,u\right)  \geq\frac{c_{\operatorname*{coer}}}{p^{2}}%
\frac{\nu}{\left(  2+\nu\right)  ^{2}}\left\Vert u\right\Vert _{H^{1}\left(
\Omega\right)  }^{2}\quad\forall u\in S\text{.}%
\]

\end{theorem}

\begin{proof}
\textbf{Continuity: }
{If $p=1$, the two bilinear forms $a^{p,\nu}\left(
\cdot,\cdot\right)  $ and $a\left(  \cdot,\cdot\right)  $ coincide and the
result trivially follows. 
Thus, we now assume that $p\geq2$. }Let $u,v\in S$.
By definition of $a^{p,\nu}$ and (\ref{defASp}) we have
\[
a^{p,\nu}\left(  u,v\right)  =\left(  \left(  A^{S}\right)  ^{1/2}P_{p,\nu
}^{\Delta t}\left(  \left(  A^{S}\right)  ^{1/2}\Pi_{\operatorname*{f}}%
^{S}\left(  A^{S}\right)  ^{1/2}\right)  \left(  A^{S}\right)  ^{1/2}%
u,v\right)  _{\mathcal{T}}.
\]
Now, let $\check{u}:=\left(  A^{S}\right)  ^{1/2}u$ and $\check{v}:=\left( A^{S}\right)  ^{1/2}v$. 
By Lemma \ref{LemEstPif}, $\left( A^S \right)^{1/2} \Pi_{\operatorname*{f}}^S \left( A^S \right)^{1/2}$ is self-adjoint with respect to $\left(\cdot,\cdot\right)_{\mathcal{T}}$ and positive semi-definite, with real valued, non-negative spectrum $\sigma\left(  \left(  A^{S}\right)  ^{1/2}\Pi_{\operatorname*{f}}^{S}\left(  A^{S}\right)  ^{1/2}\right)$. 
Hence, $\lambda\in\sigma\left( \left(  A^{S}\right)^{1/2} \Pi_{\operatorname*{f}}^{S} \left(  A^{S}\right)^{1/2}\right)$, 
if there exists $w_{\lambda}\in S \setminus \left\{0\right\}$ such that%
\[
\left(  \left(  A^{S}\right)  ^{1/2}\Pi_{\operatorname*{f}}^{S}\left(
A^{S}\right)  ^{1/2}w_{\lambda},v\right)  _{\mathcal{T}}=\lambda\left(
w_{\lambda},v\right)  _{\mathcal{T}}\quad\forall v\in S.
\]
Then
\begin{equation}
a^{p,\nu}\left(  u,v\right)  =\left(  P_{p,\nu}^{\Delta t}\left(  \left(
A^{S}\right)  ^{1/2}\Pi_{\operatorname*{f}}^{S}\left(  A^{S}\right)
^{1/2}\right)  \check{u},\check{v}\right)  _{\mathcal{T}}\leq\rho_{p,\nu
}\left\Vert \check{u}\right\Vert _{\mathcal{T}}\left\Vert \check{v}\right\Vert
_{\mathcal{T}} \label{acontH1est}%
\end{equation}
for%
\[
\rho_{p,\nu}=\sup_{\lambda\in\sigma\left(  \left(  A^{S}\right)  ^{1/2}%
\Pi_{\operatorname*{f}}^{S}\left(  A^{S}\right)  ^{1/2}\right)  }\left\vert
P_{p,\nu}^{\Delta t}\left(  \lambda\right)  \right\vert.
\]
Since by Lemma \ref{LemEstPif}, the maximal eigenvalue
$\lambda_{\max}$ satisfies%
\[
\lambda_{\max}\leq C_{\operatorname*{cont}}C_{\operatorname*{inv}}^{2}\left(
\frac{p}{h_{\operatorname*{c}}}\right)  ^{2},
\]
we conclude from CFL condition (\ref{condcfl2}) that%
\begin{align*}
\rho_{p,\nu} &\leq \sup_{0 \leq \lambda \leq \lambda_{\max}} \left| P_{p,\nu}^{\Delta t} \left(\lambda\right) \right| 
= \sup_{0 \leq x \leq \Delta t^2 \lambda_{\max}} \left| \frac{P_{p,\nu}\left(  x\right)  }{x} \right| \\
&\overset{(\ref{condcfl2})}{\leq} \sup_{0\leq x\leq \left(1+\delta_{p,\nu}\right)\, \omega_{p,\nu}}\left\vert \frac{P_{p,\nu}\left(  x\right)  }{x}\right\vert
\leq \sup_{0\leq x\leq2\delta_{p,\nu}\omega_{p,\nu}}\left\vert \frac{P_{p,\nu}\left(  x\right)  }{x}\right\vert .
\end{align*}
We use Lemma \ref{LemEstPpnue} to get%
\[
\left\vert \frac{P_{p,\nu}\left(  x\right)  }{x}\right\vert \leq
1 
\quad\forall x\in\left(  0,2\delta_{p,\nu
}\omega_{p,\nu}\right)  .
\]

Inserting this upper bound into (\ref{acontH1est}) results in%
\[
a^{p,\nu}\left(  u,v\right)  \leq 
\left\Vert \check{u}\right\Vert _{\mathcal{T}}\left\Vert \check{v}\right\Vert
_{\mathcal{T}}.
\]
Furthermore, we have%
\begin{equation}
\left\Vert \check{u}\right\Vert _{\mathcal{T}}^{2}=\left\Vert \left(
A^{S}\right)  ^{1/2}u\right\Vert _{\mathcal{T}}^{2}=\left(  A^{S}u,u\right)
_{\mathcal{T}}=a\left(  u,u\right)  \leq C_{\operatorname*{cont}}\left\Vert
u\right\Vert _{H^{1}\left(  \Omega\right)  }^{2} \label{ucheckest2}%
\end{equation}
and end up with the continuity estimate%
\[
a^{p,\nu}\left(  u,v\right)  \leq C_{\operatorname*{cont}} 
\left\Vert u\right\Vert _{H^{1}\left(  \Omega\right)
}\left\Vert v\right\Vert _{H^{1}\left(  \Omega\right)  }.
\]

\textbf{Symmetry. }Since $A^{S}$ and $\Pi
_{\operatorname*{f}}^S$ are both self-adjoint with respect to the $\left(
\cdot,\cdot\right)  _{\mathcal{T}}$ scalar product, the symmetry immediately follows.

\textbf{Coercivity}. Let $u\in S$ and set again $\check{u}:=\left(
A^{S}\right)  ^{1/2}u$. Then%
\[
a^{p,\nu}\left(  u,u\right)  =\left(  P_{p,\nu}^{\Delta t}\left(  \left(
A^{S}\right)  ^{1/2}\Pi_{\operatorname*{f}}^{S}\left(  A^{S}\right)
^{1/2}\right)  \check{u},\check{u}\right)  _{\mathcal{T}}\geq\mu_{p,\nu
}\left\Vert \check{u}\right\Vert _{\mathcal{T}}^{2}%
\]
for%
\[
\mu_{p,\nu}=\inf_{\lambda\in\sigma\left(  \left(  A^{S}\right)  ^{1/2}%
\Pi_{\operatorname*{f}}^{S}\left(  A^{S}\right)  ^{1/2}\right)  }\left\vert
P_{p,\nu}^{\Delta t}\left(  \lambda\right)  \right\vert .
\]

From a similar argument as for the continuity estimate, we obtain%
\[
\mu_{p,\nu}\geq\inf_{0\leq x\leq\left(  1 + \delta_{p,\nu} \right)  \omega_{p,\nu}%
}\left\vert \frac{P_{p,\nu}\left(  x\right)  }{x}\right\vert .
\]
Now, employ Lemma \ref{LemEstPpnue} to obtain the lower bound%
\[
\mu_{p,\nu}\geq\frac{\nu}{\left(  2+\nu\right)  ^{2}p^{2}}.
\]
As in (\ref{ucheckest2}) we obtain $\left\Vert \check{u}\right\Vert
_{\mathcal{T}}^{2}\geq c_{\operatorname*{coer}}\left\Vert u\right\Vert
_{H^{1}\left(  \Omega\right)  }^{2}$ and, in turn,%
\[
a^{p,\nu}\left(  u,u\right)  =\left(  P_{p,\nu}^{\Delta t}\left(  \left(
A^{S}\right)  ^{1/2}\Pi_{\operatorname*{f}}^{S}\left(  A^{S}\right)
^{1/2}\right)  \check{u},\check{u}\right)  _{\mathcal{T}}\geq\frac
{c_{\operatorname*{coer}}}{p^{2}}\frac{\nu}{\left(  2+\nu\right)  ^{2}%
}\left\Vert u\right\Vert _{H^{1}\left(  \Omega\right)  }^{2}.
\]

\end{proof}

For the convergence proof in Section 3.2, the norm of $A^{S,p,\nu}$ and its inverse (as an operator in $L^2$) will play a key role.  Both can be expressed in terms of the largest and smallest eigenvalue of $A^{S,p,\nu}$. 
Moreover, the estimate of the inverse of $A^{S,p,\nu}$ must also take into account its specific coarse/fine block structure. As it is based on a Schur complement representation of the inverse, it also involves the difference $(A^{S,p,\nu})^{-1}-(A^S)^{-1}$.

To estimate the extremal eigenvalues
 of the operator $A^{S,p,\nu}$, we first rewrite it in block form.
Hence, we first introduce some notation and let $V_{i}$ be a finite-dimensional Hilbert space with scalar product $\left( \cdot,\cdot \right)_{i}$ and norm $\left( \cdot,\cdot \right)_{i}^{1/2}$, $i = 1,2$. 
We denote by $I_{i}:V_{i}\rightarrow V_{i}$ the identity operator. 
The adjoint of an operator $K: V_{i} \rightarrow V_{j}$ with respect to the scalar product $\left( \cdot,\cdot \right)_{j}$ is denoted by $K^{\ast}: V_{j} \rightarrow V_{i}$ and characterized by%
\[
\left(  Ku,w\right)  _{j}=\left(  u,K^{\ast}w\right)  _{i}\quad\forall u\in
V_{i}\quad\forall w\in V_{j}.
\]
Recall that $K:V_{i}\rightarrow V_{i}$ is self-adjoint if $K=K^{\ast}$ and positive
definite if%
\begin{equation}
K\text{ is self-adjoint and }\left(  Ku,u\right)  _{i}>0\quad\forall u\in
V_{i}\backslash\left\{  0\right\}  \text{.}\label{posdef}%
\end{equation}
If $K$ satisfies (\ref{posdef}) with \textquotedblleft$>$\textquotedblright\ replaced by \textquotedblleft$\geq$\textquotedblright , $K$ is positive semi-definite.

For any $\delta \in \mathbb{R}$, we shall write $K>\delta I$ and $K\geq\delta I$, if $K-\delta I$ is positive definite or positive semi-definite, respectively.
For criteria and rules in the calculus of positive-definite operators, we refer, e.g., to \cite[\S \ C.1]{hackitengl2nd}.

Next, we consider $2\times2$ block operators on $V_1 \times V_2$ and denote the induced scalar product by%
\begin{equation}
\left[  u,v\right]  =\left(  u_{1},v_{1}\right)  _{1}+\left(  u_{2}%
,v_{2}\right)  _{2}\quad\forall u=\binom{u_{1}}{u_{2}},v=\binom{v_{1}}{v_{2}%
}\in V_{1}\times V_{2}. 
\label{addscprod}%
\end{equation}

Let $B_{i,j}:V_{j}\rightarrow V_{i}$ be bounded linear operators which form
the blocks of the operator $\mathbb{B}:V_{1}\times V_{2}\rightarrow V_{1}\times V_{2}$ via%
\[
\mathbb{B}=\left[
\begin{array}[c]{cc}%
B_{1,1} & B_{1,2}\\
B_{2,1} & B_{2,2}%
\end{array}
\right]  .
\]
Then, we define (formally) the Schur complement of $\mathbb{B}$ as 
\[
\mathfrak{S}
:=B_{1,1}-B_{1,2}B_{2,2}^{-1}B_{2,1},
\]
which is regular if both $\mathbb{B}$ and $B_{2,2}$ are regular \cite[Remark C.63]{hackitengl2nd}. Then, the inverse of $\mathbb{B}$ can be written as
\begin{align}
\mathbb{B}^{-1}  &  =\left[
\begin{array}[c]{cc}%
\mathfrak{S}^{-1} & -\mathfrak{S}^{-1} B_{1,2} B_{2,2}^{-1}\\
-B_{2,2}^{-1}B_{2,1} \mathfrak{S}^{-1} 
& B_{2,2}^{-1}+B_{2,2}^{-1}%
B_{2,1} \mathfrak{S}^{-1} B_{1,2} B_{2,2}^{-1}%
\end{array}
\right] \nonumber\\
&  =\left[
\begin{array}[c]{rr}%
I_{1} & 0\\
-B_{2,2}^{-1}B_{2,1} & I_{2}%
\end{array}
\right]  
\left[
\begin{array}[c]{cc}%
\mathfrak{S}^{-1} & 0\\
0 & B_{2,2}^{-1}%
\end{array}
\right]  
\left[
\begin{array}[c]{rr}%
I_{1} & -B_{1,2}B_{2,2}^{-1}\\
0 & I_{2}%
\end{array}
\right]  . \label{LDLT2}%
\end{align}

Using the notation introduced above, we now write the operators $A^{S}$ and $A^{S,p,\nu}$ in block form. 
Since $ S = S_{\operatorname*{c}} \oplus S_{\operatorname*{f}} $, we can define local versions of $A^{S}$ restricted to either the ``coarse'' or ``fine'' subspace for $s,t\in \left\{ \operatorname*{c},\operatorname*{f} \right\}$ by%
\[
A_{s\,t}^{S}: S_{t}\rightarrow S_{s},
\qquad \left(  A_{s\,t}^{S} u_{t},v_{s} \right)_{\mathcal{T}} = a\left(  u_{t},v_{s} \right)  
\qquad\forall u_{t}\in S_{t}
\text{,\quad}u_{s}\in S_{s},
\]
which yields the block operator 
$\mathbb{A}^{S}: S_{\operatorname*{c}} \times S_{\operatorname*{f}}\rightarrow S_{\operatorname*{c}}\times S_{\operatorname*{f}}$, 
\[
\mathbb{A}^{S}:=\left[
\begin{array}[c]{rr}%
A_{\operatorname*{c}\operatorname*{c}}^{S} & A_{\operatorname*{c}\operatorname*{f}}^{S}\\
A_{\operatorname*{f}\operatorname*{c}}^{S} & A_{\operatorname*{f}\operatorname*{f}}^{S}%
\end{array}
\right]  .
\]
We also equip $S_{\operatorname*{c}} \times S_{\operatorname*{f}}$ with a scalar product as in (\ref{addscprod}): 
\begin{equation}
\left[  u,v\right]  _{\mathcal{T}}
:=\left(  u_{\operatorname*{c}} ,v_{\operatorname*{c}} \right)_{\mathcal{T}}
+\left(  u_{\operatorname*{f}},v_{\operatorname*{f}}\right)_{\mathcal{T}},
\qquad u = \binom{u_{\operatorname*{c}}}{u_{\operatorname*{f}}}, 
v = \binom{v_{\operatorname*{c}}}{v_{\operatorname*{f}}}
\in S_{\operatorname*{c}} \times S_{\operatorname*{f}}.
 \label{defblockscprod}%
\end{equation}
Since $A_{\operatorname*{c}\operatorname*{f}}^{S} = \left( A_{\operatorname*{f}\operatorname*{c}}^{S} \right)^{\ast}$ and both $A_{\operatorname*{c}\operatorname*{c}}^{S}$ and $A_{\operatorname*{f}\operatorname*{f}}^{S}$ are self-adjoint with respect to $(\cdot,\cdot)_{\mathcal{T}}$, 
$\mathbb{A}^{S}$ is self-adjoint with respect to $\left[\cdot,\cdot\right]_{\mathcal{T}}$.
These block operators can also be written in terms of the projections $\Pi^{S}_{\operatorname*{c}}$ and $\Pi^{S}_{\operatorname*{f}}$ from (\ref{defPifPic}) as%
\[
A_{s\,t}^{S} = \Pi^{S}_{s}A^{S}\Pi^{S}_{t}.
\]
Similarly, we associate to $A^{S,p,\nu}$ in (\ref{defASp}) the block operator
\begin{equation}
\mathbb{A}^{S,p,\nu} :=
\left[
\begin{array}[c]{cc}%
A_{\operatorname*{c}\operatorname*{c}}^{S,p,\nu} & A_{\operatorname*{c}\operatorname*{f}}^{S,p,\nu}\\
A_{\operatorname*{f}\operatorname*{c}}^{S,p,\nu} & A_{\operatorname*{f}\operatorname*{f}}^{S,p,\nu}%
\end{array}
\right]. 
\label{defAblackboard}%
\end{equation}

We shall now compute the blocks of $\mathbb{A}^{S,p,\nu}$ explicitly. 
First, we associate to $\Pi_{\operatorname*{f}}^{S}$ the block operator
\begin{equation*}
\mathbb{R}_{\operatorname*{f}}:=
\left[
\begin{array}[c]{rr}%
0 & 0\\
0 & I_{\operatorname*{f}\operatorname*{f}}%
\end{array}
\right],
\end{equation*}
with $I_{\operatorname*{f}\operatorname*{f}}: S_{\operatorname*{f}} \rightarrow S_{\operatorname*{f}}$ the identity and recall from (\ref{defASp}) that
\[
\mathbb{A}^{S,p,\nu} = \mathbb{A}^{S} - \mathbb{A}^{S}  Q_{p,\nu}^{\Delta t} \left( \mathbb{R}_{\operatorname*{f}} \mathbb{A}^{S} \right).
\]
For $j\geq1$, we easily verify by direct computation that
\[
\left(  \mathbb{A}^{S}\mathbb{R}_{\operatorname*{f}}\right)  ^{j}=\left[
\begin{array}
[c]{cc}%
0 & A_{\operatorname*{c}\operatorname*{f}}^{S}\left(  A_{\operatorname*{f}%
\operatorname*{f}}^{S}\right)  ^{j-1}\\
0 & \left(  A_{\operatorname*{f}\operatorname*{f}}^{S}\right)  ^{j}%
\end{array}
\right]  
\]
and therefore
\begin{equation}
\mathbb{A}^{S} \left( \mathbb{R}_{\operatorname*{f}} \mathbb{A}^{S} \right)  ^{j} =
\left(  \mathbb{A}^{S}\mathbb{R}_{\operatorname*{f}}\right)  ^{j}\mathbb{A}^{S}=
\left[
\begin{array}[c]{cc}%
A_{\operatorname*{c}\operatorname*{f}}^{S} \left(  A_{\operatorname*{f} \operatorname*{f}}^{S}\right)  ^{j-1} A_{\operatorname*{f}\operatorname*{c}}^{S} 
& A_{\operatorname*{c}\operatorname*{f}}^{S}\left(  A_{\operatorname*{f} \operatorname*{f}}^{S}\right)  ^{j} \\
\left(  A_{\operatorname*{f}\operatorname*{f}}^{S} \right)^{j} A_{\operatorname*{f}\operatorname*{c}}^{S} 
& \left(  A_{\operatorname*{f}\operatorname*{f}}^{S}\right)  ^{j+1}
\end{array}
\right]  .
\label{APjA}
\end{equation}
From (\ref{defQpnu}) and (\ref{APjA}), we thus obtain
\begin{align*}
&\mathbb{A}^{S}  Q_{p,\nu}^{\Delta t} \left( \mathbb{R}_{\operatorname*{f}} \mathbb{A}^{S} \right)
= \frac{2}{p^2} \sum\limits_{j=1}^{p-1} \alpha_{j}^{p,\nu} \left( \frac{\Delta t}{p} \right)^{2 j} \mathbb{A}^{S} \left( \mathbb{R}_{\operatorname*{f}} \mathbb{A}^{S} \right)^{j} \\
&= \frac{2}{p^2} \sum\limits_{j=1}^{p-1} \alpha_{j}^{p,\nu} \left( \frac{\Delta t}{p} \right)^{2 j}
\left[
\begin{array}[c]{cc}%
A_{\operatorname*{c}\operatorname*{f}}^{S} \left(  A_{\operatorname*{f} \operatorname*{f}}^{S}\right)  ^{j-1} A_{\operatorname*{f}\operatorname*{c}}^{S} & 
A_{\operatorname*{c}\operatorname*{f}}^{S}\left(  A_{\operatorname*{f} \operatorname*{f}}^{S}\right)  ^{j} \\
\left(  A_{\operatorname*{f}\operatorname*{f}}^{S} \right)^{j} A_{\operatorname*{f}\operatorname*{c}}^{S} & 
\left(  A_{\operatorname*{f}\operatorname*{f}}^{S}\right)  ^{j+1}
\end{array}
\right] \\
& = \left[
\begin{array}[c]{@{} c @{} c @{}}
A_{\operatorname*{c} \operatorname*{f}}^{S} \left(  A_{\operatorname*{f} \operatorname*{f}}^{S} \right)^{-1/2} & 0 \\
0 & \left(  A_{\operatorname*{f} \operatorname*{f}}^{S} \right)^{1/2}
\end{array}
\right]  
\left[
\begin{array}[c]{cc}
Q_{p,\nu}^{\Delta t}\left(  A_{\operatorname*{f}\operatorname*{f}}^{S}\right)
& Q_{p,\nu}^{\Delta t}\left(  A_{\operatorname*{f}\operatorname*{f}}^{S}\right) \\
Q_{p,\nu}^{\Delta t}\left(  A_{\operatorname*{f}\operatorname*{f}}^{S}\right)
& Q_{p,\nu}^{\Delta t}\left(  A_{\operatorname*{f}\operatorname*{f}}^{S}\right)
\end{array}
\right]  
\left[
\begin{array}[c]{@{} c @{} c @{}}
\left(  A_{\operatorname*{f} \operatorname*{f}}^{S} \right)^{-1/2} A_{\operatorname*{f}\operatorname*{c}}^{S} & 0\\
0 & \left(  A_{\operatorname*{f}\operatorname*{f}}^{S}\right)  ^{1/2}%
\end{array}
\right]  , 
\end{align*}
which yields the explicit block representation of $\mathbb{A}^{S,p,\nu}$:%
\begin{equation}
\begin{aligned}
&\mathbb{A}^{S,p,\nu}= \\
&\quad\left[
\begin{array}[c]{@{} c @{} c @{}}
A_{\operatorname*{c}\operatorname*{c}}^{S} - A_{\operatorname*{c} \operatorname*{f}}^{S} \left(  A_{\operatorname*{f}\operatorname*{f}}^{S}\right)^{-1/2} Q_{p,\nu}^{\Delta t} \left(  A_{\operatorname*{f} \operatorname*{f}}^{S} \right)  \left(  A_{\operatorname*{f}\operatorname*{f}}^{S}\right)^{-1/2} A_{\operatorname*{f} \operatorname*{c}}^{S} 
&
A_{\operatorname*{c}\operatorname*{f}}^{S} P_{p,\nu}^{\Delta t} \left( A_{\operatorname*{f}\operatorname*{f}}^{S} \right) \\
P_{p,\nu}^{\Delta t}\left(  A_{\operatorname*{f}\operatorname*{f}}^{S}\right)
A_{\operatorname*{f}\operatorname*{c}}^{S} 
& 
\left(  A_{\operatorname*{f} \operatorname*{f}}^{S}\right)^{1/2} P_{p,\nu}^{\Delta t} \left( A_{\operatorname*{f} \operatorname*{f}}^{S} \right)  \left( A_{\operatorname*{f}\operatorname*{f}}^{S}\right)^{1/2}%
\end{array}
\right]  .
\end{aligned}
\label{Apvblackboard}
\end{equation}
Note that both $A_{\operatorname*{f}\operatorname*{f}}^S$ and $A_{\operatorname*{f} \operatorname*{f}}^{S,p,\nu}$ are self-adjoint and coercive because of (\ref{wellposedc}) and Theorem \ref{TheoCont}.
Since%
\begin{align*}
-A_{\operatorname*{c} \operatorname*{f}}^{S,p,\nu} \left(  A_{\operatorname*{f} \operatorname*{f}}^{S,p,\nu} \right)^{-1} 
&= -A_{\operatorname*{c} \operatorname*{f}}^{S} P_{p,\nu}^{\Delta t} \left( A_{\operatorname*{f} \operatorname*{f}}^{S}\right)  \left(  \left(  A_{\operatorname*{f} \operatorname*{f}}^{S} \right)^{1/2}P_{p,\nu}^{\Delta t} \left( A_{\operatorname*{f} \operatorname*{f}}^{S} \right)  \left( A_{\operatorname*{f} \operatorname*{f}}^{S} \right)^{1/2} \right)^{-1} \\
&= -A_{\operatorname*{c}\operatorname*{f}}^{S} \left( A_{\operatorname*{f} \operatorname*{f}}^{S} \right)^{-1},
\end{align*}
the $LDL^{\ast}$ factorization of $\left(
\mathbb{A}^{S,p,\nu}\right)  ^{-1}-\left(  \mathbb{A}^{S}\right)  ^{-1}$
according to (\ref{LDLT2}) is given by%
\begin{align*}
&  \left(  \mathbb{A}^{S,p,\nu}\right)  ^{-1}-\left(  \mathbb{A}^{S}\right)
^{-1} = \\
&  \left[
\begin{array}[c]{@{} c c @{}}%
I_{\operatorname*{c}\operatorname*{c}} & 0\\
- \left( A_{\operatorname*{f}%
\operatorname*{f}}^{S} \right)^{-1} A_{\operatorname*{f}%
\operatorname*{c}}^{S} & I_{\operatorname*{f}\operatorname*{f}}%
\end{array}
\right]  \left[
\begin{array}[c]{@{} c @{} c @{}}%
\left(\mathfrak{S}^{S,p,\nu}\right)^{-1} - \left(\mathfrak{S}^{S}\right)^{-1} & 0\\
0 & \left(  A_{\operatorname*{f}\operatorname*{f}}^{S,p,\nu}\right)
^{-1}-\left(  A_{\operatorname*{f}\operatorname*{f}}^{S}\right)  ^{-1}%
\end{array}
\right]  \left[
\begin{array}[c]{@{} c c @{}}%
I_{\operatorname*{c}\operatorname*{c}} & -A_{\operatorname*{c}%
\operatorname*{f}}^{S} \left( A_{\operatorname*{f}%
\operatorname*{f}}^{S} \right)^{-1}\\
0 & I_{\operatorname*{f}\operatorname*{f}}%
\end{array}
\right]  ,
\end{align*}
where $\mathfrak{S}^{S,p,\nu}$ and $\mathfrak{S}^{S}$ denotes the Schur complements of $\mathbb{A}^{S,p,\nu}$ and $\mathbb{A}^{S}$, respectively. 
A simple calculation shows that $\mathfrak{S}^{S,p,\nu} = \mathfrak{S}^{S}$ and we end up with%
\begin{equation}
\left(  \mathbb{A}^{S,p,\nu}\right)  ^{-1}-\left(  \mathbb{A}^{S}\right)
^{-1}=\left[
\begin{array}[c]{cc}%
0 & 0\\
0 & \left(  A_{\operatorname*{f}\operatorname*{f}}^{S,p,\nu}\right)
^{-1}-\left(  A_{\operatorname*{f}\operatorname*{f}}^{S}\right)  ^{-1}%
\end{array}
\right]  . \label{ablockdiff}%
\end{equation}

\begin{lemma}
\label{LemInvDiff}Let (\ref{wellposedb}), (\ref{wellposedc}),
(\ref{hsmallerh0}), (\ref{ceqCeq}), and (\ref{hfphc}) be satisfied. 
Assume that the \emph{CFL condition} (\ref{condcfl2}) holds. Then%
\[
\left(  \left(  \left(  A^{S,p,\nu}\right)  ^{-1}-\left(  A^{S}\right)
^{-1}\right)  u,v\right)  _{\mathcal{T}}\leq\frac{\nu+1}{2 \nu}\Delta
t^{2}\left\Vert u\right\Vert _{\mathcal{T}}\left\Vert v\right\Vert
_{\mathcal{T}}\quad\forall u,v\in S.
\]

\end{lemma}

\begin{proof}
From (\ref{wellposedc}) and Theorem \ref{TheoCont} we conclude that $\left(
A^{S}\right)  ^{-1}$ and $\left(  A^{S,p,\nu}\right)  ^{-1}$ exist. 
We use (\ref{ablockdiff}) and the above expressions for the corresponding block matrices to
obtain%
\begin{align*}
\left(  \left(  \left(  A^{S,p,\nu}\right)  ^{-1}-\left(  A^{S}\right)
^{-1}\right)  u,v\right)  _{\mathcal{T}}  &  =\left(  \left(
A_{\operatorname*{f}\operatorname*{f}}^{S}\right)  ^{-1/2}\left(  \left(
P_{p,\nu}^{\Delta t}\left(  A_{\operatorname*{f}\operatorname*{f}}^{S}\right)
\right)  ^{-1}-I\right)  \left(  A_{\operatorname*{f}\operatorname*{f}}%
^{S}\right)  ^{-1/2}u,v\right)  _{\mathcal{T}}\\
&  \leq\lambda_{\operatorname*{f}\operatorname*{f}}\left\Vert u\right\Vert
_{\mathcal{T}}\left\Vert v\right\Vert _{\mathcal{T}}\qquad\forall u,v\in S,
\end{align*}
where%
\[
\lambda_{\operatorname*{f}\operatorname*{f}}:=\sup_{x\in\sigma\left(
A_{\operatorname*{f}\operatorname*{f}}^{S}\right)  }\left\vert \frac
{1-P_{p,\nu}^{\Delta t}\left(  x\right)  }{xP_{p,\nu}^{\Delta t}\left(
x\right)  }\right\vert =\Delta t^{2}\sup_{y\in\sigma\left(  \Delta
t^{2}A_{\operatorname*{f}\operatorname*{f}}^{S}\right)  }\left\vert
\frac{1-y^{-1}P_{p,\nu}\left(  y\right)  }{P_{p,\nu}\left(  y\right)
}\right\vert .
\]
By Lemma \ref{LemEstPif} and (\ref{condcfl2}), the maximal eigenvalue of $\Delta t^{2}A_{\operatorname*{f} \operatorname*{f}}^{S}$ appearing in the supremum can be bounded from above by%
\begin{equation}
\Delta t^{2}C_{\operatorname*{cont}}C_{\operatorname*{inv}}^{2} \left(\frac{p}{h_{\operatorname*{c}}}\right)^{2}
\overset{(\ref{condcfl2})}{\leq} \left(  2+\nu/p^{2}\right) \omega_{p,\nu}. 
\label{stuppbound}%
\end{equation}
Hence, we may apply Lemma \ref{Lemdiffquot} to obtain%
\[
\lambda_{\operatorname*{f}\operatorname*{f}}\leq\frac{\nu+1}{2\nu}\Delta
t^{2}.
\]%
\end{proof}

Starting from the block respresentation (\ref{Apvblackboard}) we now estimate the extreme eigenvalues of the stabilized LTS bilinear form $a^{p,\nu}$.
\begin{theorem}
\label{Coreigsys}
Let (\ref{wellposed}), (\ref{hsmallerh0}), (\ref{ceqCeq}), (\ref{hfphc}), (\ref{CFLtotal}) be satisfied.
Denote by $\lambda_{p,\nu} > 0$ the eigenvalues of the (symmetric and coercive) bilinear form $a^{p,\nu}$, i.e. 
\[%
a^{p,\nu}\left(  \eta_{p,\nu},v\right)  
= \lambda_{p,\nu} \left(\eta_{p,\nu},v\right)_{\mathcal{T}} 
\quad \forall v\in S,
\]
with corresponding eigenfunctions $\eta_{p,\nu}$ orthonormalized w.r.t. $(\cdot,\cdot)_{\mathcal{T}}$.

Then, the largest eigenvalue satisfies%
\begin{equation}
\Delta t^2 \lambda_{p,\nu}^{\max}\leq4-\frac{\nu}{\nu+1} \label{estlmaxp}%
\end{equation}
and the smallest eigenvalue satisfies%
\[
\lambda_{p,\nu}^{\min}\geq\frac{1}{2}\min\left\{  \frac
{c_{\operatorname*{coer}}}{C_{\operatorname*{eq}}^{2}},\frac{2 \nu}{\left(
\nu+1\right)  \Delta t^2}\right\} =: c_\nu  .
\]

\end{theorem}%

\begin{proof} 
\textbf{(i) Upper bound for the largest eigenvalue}

First, we derive upper bounds for the individual blocks in (\ref{Apvblackboard}).

\textbf{Estimate of }$A_{\operatorname*{f}\operatorname*{f}}^{S,p,\nu}$\textbf{:}

We have%
\begin{equation}
\left(  \left(  A_{\operatorname*{f}\operatorname*{f}}^{S}\right)
^{1/2}P_{p,\nu}^{\Delta t}\left(  A_{\operatorname*{f}\operatorname*{f}}%
^{S}\right)  \left(  A_{\operatorname*{f}\operatorname*{f}}^{S}\right)
^{1/2}u_{\operatorname*{f}},u_{\operatorname*{f}}\right)  _{\mathcal{T}}%
\leq \frac{\kappa_{\operatorname*{f}\operatorname*{f}}}{\Delta t^2} \left\Vert u_{\operatorname*{f}%
}\right\Vert _{\mathcal{T}}^{2}, 
\label{kappaff}%
\end{equation}
where%
\[
\kappa_{\operatorname*{f}\operatorname*{f}}
:=\sup_{x\in\sigma\left(  \Delta t^{2}A_{\operatorname*{f}\operatorname*{f}}%
^{S}\right)  }\left\vert P_{p,\nu}\left(  x\right)  \right\vert .
\]
We again use the bound (\ref{stuppbound}) for the maximal eigenvalue of
$\Delta t^{2}A_{\operatorname*{f}\operatorname*{f}}^{S}$ and conclude from
Lemma \ref{LemPpnue}%
\[
\kappa_{\operatorname*{f}\operatorname*{f}}\leq\frac{2\left(  2+\nu\right)
}{1+\nu}=4-\frac{2\nu}{\nu+1}.
\]
This leads to%
\[
\left(  A_{\operatorname*{f}\operatorname*{f}}^{S}\right)  ^{1/2}P_{p,\nu
}^{\Delta t}\left(  A_{\operatorname*{f}\operatorname*{f}}^{S}\right)  \left(
A_{\operatorname*{f}\operatorname*{f}}^{S}\right)  ^{1/2}\leq\frac{2}{\Delta
t^{2}}\frac{\left(  2+\nu\right)  }{1+\nu}I_{\operatorname*{f}%
\operatorname*{f}}.
\]

\textbf{Estimate of }
$A_{\operatorname*{c}\operatorname*{c}}^{S,p,\nu} = 
A_{\operatorname*{c}\operatorname*{c}}^{S} - 
A_{\operatorname*{c}\operatorname*{f}}^{S} \left(  A_{\operatorname*{f}\operatorname*{f}}^{S}\right)^{-1/2} Q_{p,\nu}^{\Delta t} \left(  A_{\operatorname*{f}\operatorname*{f}}^{S}\right) \left(  A_{\operatorname*{f}\operatorname*{f}}^{S} \right)^{-1/2} A_{\operatorname*{f}\operatorname*{c}}^{S}$\textbf{:}

\textbf{1) Estimate of }$A_{\operatorname*{c}\operatorname*{c}}^{S}$\textbf{:}

By continuity of $a$, we have for all $u_{\operatorname*{c}}\in S_{\operatorname*{c}}$%
\[
\left(  A_{\operatorname*{c}\operatorname*{c}}^{S}u_{\operatorname*{c}%
},u_{\operatorname*{c}}\right)  _{\mathcal{T}}=\left(  A^{S}%
u_{\operatorname*{c}},u_{\operatorname*{c}}\right)  _{\mathcal{T}}\leq
C_{\operatorname*{cont}}\left\Vert u_{\operatorname*{c}}\right\Vert
_{H^{1}\left(  \Omega\right)  }^{2}.
\]
From (\ref{globalinvcoarse}), we thus conclude that
\begin{equation}
\left(  A_{\operatorname*{c}\operatorname*{c}}^{S}u_{\operatorname*{c}%
},u_{\operatorname*{c}}\right)  _{\mathcal{T}}\leq C_{\operatorname*{cont}%
}C_{\operatorname*{inv}}^{2}h_{\operatorname*{c}}%
^{-2}\left\Vert u_{\operatorname*{c}}\right\Vert _{\mathcal{T}}^{2}.
\label{estass}%
\end{equation}

\textbf{2) Estimate of }$A_{\operatorname*{c}\operatorname*{f}}%
^{S}\left(  A_{\operatorname*{f}\operatorname*{f}}^{S}\right)  ^{-1/2}%
Q_{p,\nu}^{\Delta t}\left(  A_{\operatorname*{f}\operatorname*{f}}^{S}\right)
\left(  A_{\operatorname*{f}\operatorname*{f}}^{S}\right)  ^{-1/2}%
A_{\operatorname*{f}\operatorname*{c}}^{S}$\textbf{:}

By symmetry of $A_{\operatorname*{f} \operatorname*{f}}^S$ with respect to $(\cdot, \cdot)_{\mathcal{T}}$, we have%
\begin{align*}
\left(  \left(  A_{\operatorname*{f}\operatorname*{f}}^{S}\right)
^{-1/2}A_{\operatorname*{f}\operatorname*{c}}^{S}u_{\operatorname*{c}%
},v_{\operatorname*{f}}\right)  _{\mathcal{T}}  &  =\left(
A_{\operatorname*{f}\operatorname*{c}}^{S}u_{\operatorname*{c}},\left(
A_{\operatorname*{f}\operatorname*{f}}^{S}\right)  ^{-1/2}v_{\operatorname*{f}%
}\right)  _{\mathcal{T}}=a\left(  u_{\operatorname*{c}},\left(
A_{\operatorname*{f}\operatorname*{f}}^{S}\right)  ^{-1/2}v_{\operatorname*{f}%
}\right) \\
&  \leq C_{\operatorname*{cont}}\left\Vert u_{\operatorname*{c}}\right\Vert
_{H^{1}\left(  \Omega\right)  }\left\Vert \left(  A_{\operatorname*{f}%
\operatorname*{f}}^{S}\right)  ^{-1/2}v_{\operatorname*{f}}\right\Vert
_{H^{1}\left(  \Omega\right)  } \\
&  \overset{\text{(\ref{globalinvcoarse})}}{\leq}C_{\operatorname*{cont}%
}C_{\operatorname*{inv}}h_{\operatorname*{c}}^{-1}\left\Vert
u_{\operatorname*{c}}\right\Vert _{\mathcal{T}}\left\Vert \left(
A_{\operatorname*{f}\operatorname*{f}}^{S}\right)  ^{-1/2}v_{\operatorname*{f}%
}\right\Vert _{H^{1}\left(  \Omega\right)  }.
\end{align*}
Since we can bound the last term in (\ref{estass}) as%
\begin{align*}
c_{\operatorname*{coer}}\left\Vert \left(  A_{\operatorname*{f}%
\operatorname*{f}}^{S}\right)  ^{-1/2}v_{\operatorname*{f}}\right\Vert
_{H^{1}\left(  \Omega\right)  }^{2}  &  \leq a\left(  \left(
A_{\operatorname*{f}\operatorname*{f}}^{S}\right)  ^{-1/2}v_{\operatorname*{f}%
},\left(  A_{\operatorname*{f}\operatorname*{f}}^{S}\right)  ^{-1/2}%
v_{\operatorname*{f}}\right) \\
&  =\left(  A_{\operatorname*{f}\operatorname*{f}}^{S}\left(
A_{\operatorname*{f}\operatorname*{f}}^{S}\right)  ^{-1/2}v_{\operatorname*{f}%
},\left(  A_{\operatorname*{f}\operatorname*{f}}^{S}\right)  ^{-1/2}%
v_{\operatorname*{f}}\right)  _{\mathcal{T}}=\left\Vert v_{\operatorname*{f}%
}\right\Vert _{\mathcal{T}}^{2},
\end{align*}
we obtain%
\begin{equation}
\left(  \left(  A_{\operatorname*{f}\operatorname*{f}}^{S}\right)
^{-1/2}A_{\operatorname*{f}\operatorname*{c}}^{S}u_{\operatorname*{c}%
},v_{\operatorname*{f}}\right)  _{\mathcal{T}}\leq\frac
{C_{\operatorname*{cont}}}{\sqrt{c_{\operatorname*{coer}}}}%
C_{\operatorname*{inv}} h_{\operatorname*{c}}^{-1}\left\Vert
u_{\operatorname*{c}}\right\Vert _{\mathcal{T}}\left\Vert v_{\operatorname*{f}%
}\right\Vert _{\mathcal{T}}. \label{OffDiaga}%
\end{equation}
Similarly, we derive the upper bound%
\begin{equation}
\left(  A_{\operatorname*{c}\operatorname*{f}}^{S}\left(  A_{\operatorname*{f}%
\operatorname*{f}}^{S}\right)  ^{-1/2}v_{\operatorname*{f}}%
,u_{\operatorname*{c}}\right)  _{\mathcal{T}}
\leq\frac{C_{\operatorname*{cont}}}{\sqrt{c_{\operatorname*{coer}}}}
C_{\operatorname*{inv}}h_{\operatorname*{c}}^{-1}\left\Vert u_{\operatorname*{c}}\right\Vert
_{\mathcal{T}}\left\Vert v_{\operatorname*{f}}\right\Vert _{\mathcal{T}}.
\label{OffDiagb}%
\end{equation}
To estimate $Q_{p,\nu}^{\Delta t}\left(  A_{\operatorname*{f}\operatorname*{f}}^{S}\right)$, we proceed as in (\ref{kappaff}) to obtain%
\begin{align}
\left(  Q_{p,\nu}^{\Delta t}\left(  A_{\operatorname*{f}\operatorname*{f}}%
^{S}\right)  u_{\operatorname*{f}},u_{\operatorname*{f}}\right)
_{\mathcal{T}}  &  =\left(  u_{\operatorname*{f}},u_{\operatorname*{f}%
}\right)  _{\mathcal{T}}-\left(  P_{p,\nu}^{\Delta t}\left(
A_{\operatorname*{f}\operatorname*{f}}^{S}\right)  u_{\operatorname*{f}%
},u_{\operatorname*{f}}\right)  _{\mathcal{T}}\nonumber\\
&  \overset{(\ref{Pdeltatest})}{\leq}  
\left\Vert u_{\operatorname*{f}}\right\Vert
_{\mathcal{T}}^{2}. \label{OffDiagc}%
\end{align}
The combination of (\ref{estass}), (\ref{OffDiaga}), (\ref{OffDiagb}),
(\ref{OffDiagc}) yields the upper bound
\[
\left(  A_{\operatorname*{c}\operatorname*{c}}^{S,p,\nu}u_{\operatorname*{c}%
},u_{\operatorname*{c}}\right)  _{\mathcal{T}}\leq C_{\operatorname*{cont}%
}C_{\operatorname*{inv}}^{2}h_{\operatorname*{c}}%
^{-2}\left(  1 +  
\frac{C_{\operatorname*{cont}}}{c_{\operatorname*{coer}}%
}  \right)  \left\Vert
u_{\operatorname*{c}}\right\Vert _{\mathcal{T}}^{2}.
\]

\textbf{Estimate of }$A_{\operatorname*{c}\operatorname*{f}}^{S,p,\nu}$\textbf{:}

Following a similar argument as in (\ref{kappaff}), we obtain
\begin{align*}
\left(  A_{\operatorname*{c}\operatorname*{f}}^{S}P_{p,\nu}^{\Delta t}\left(
A_{\operatorname*{f}\operatorname*{f}}^{S}\right)  u_{\operatorname*{f}%
},v_{\operatorname*{c}}\right)  _{\mathcal{T}}  
&  = a\left(  P_{p,\nu}^{\Delta
t}\left(  A_{\operatorname*{f}\operatorname*{f}}^{S}\right)
u_{\operatorname*{f}},v_{\operatorname*{c}}\right) \\
&  \leq C_{\operatorname*{cont}}C_{\operatorname*{inv}}%
^{2}h_{\operatorname*{c}}^{-2}\left\Vert P_{p,\nu}^{\Delta t}\left(
A_{\operatorname*{f}\operatorname*{f}}^{S}\right)  u_{\operatorname*{f}%
}\right\Vert _{\mathcal{T}}\left\Vert v_{\operatorname*{c}}\right\Vert
_{\mathcal{T}}\\
&  \overset{(\ref{Pdeltatest})}{\leq}C_{\operatorname*{cont}%
}C_{\operatorname*{inv}}^{2}h_{\operatorname*{c}}%
^{-2} 
\left\Vert u_{\operatorname*{f}%
}\right\Vert _{\mathcal{T}}\left\Vert v_{\operatorname*{c}}\right\Vert
_{\mathcal{T}}.
\end{align*}

\textbf{Estimate of }$A_{\operatorname*{f}\operatorname*{c}}^{S,p,\nu}$\textbf{:}

As in the previous case, we conclude that%
\[
\left(  A_{\operatorname*{f}\operatorname*{c}}^{S,p,\nu}v_{\operatorname*{c}%
},u_{\operatorname*{f}}\right)  _{\mathcal{T}}\leq C_{\operatorname*{cont}%
}C_{\operatorname*{inv}}^{2}h_{\operatorname*{c}}%
^{-2} 
\left\Vert u_{\operatorname*{f}%
}\right\Vert _{\mathcal{T}}\left\Vert v_{\operatorname*{c}}\right\Vert
_{\mathcal{T}}%
\]
holds.

\textbf{Final estimate:}

From (\ref{Apvblackboard}) together with the above four estimates, we obtain using Young's inequality%
\begin{align*}
a^{p,\nu}\left(  u,u\right)  &=  \left[  \mathbb{A}^{S,p,\nu}\binom
{u_{\operatorname*{c}}}{u_{\operatorname*{f}}},\binom{u_{\operatorname*{c}}%
}{u_{\operatorname*{f}}}\right]  _{\mathcal{T}} \\
&\leq  C_{\operatorname*{cont}}C_{\operatorname*{inv}}%
^{2}h_{\operatorname*{c}}^{-2}\left(  1+
\frac{C_{\operatorname*{cont}}%
}{c_{\operatorname*{coer}}}
\right)  \left\Vert u_{\operatorname*{c}}\right\Vert _{\mathcal{T}}^{2} \\
&\quad  +2C_{\operatorname*{cont}}C_{\operatorname*{inv}}%
^{2}h_{\operatorname*{c}}^{-2} 
\left\Vert
u_{\operatorname*{f}}\right\Vert _{\mathcal{T}}\left\Vert u_{\operatorname*{c}%
}\right\Vert _{\mathcal{T}}+\frac{2}{\Delta t^{2}}\frac{\left(  2+\nu\right)
}{1+\nu}\left\Vert u_{\operatorname*{f}}\right\Vert _{\mathcal{T}}^{2} \\
&\leq  C_{\operatorname*{cont}}C_{\operatorname*{inv}}%
^{2}h_{\operatorname*{c}}^{-2}\left( 2 +\frac{C_{\operatorname*{cont}}%
}{c_{\operatorname*{coer}}}\right)    
\left\Vert u_{\operatorname*{c}}\right\Vert _{\mathcal{T}}^{2}\\
&\quad  +\left(  C_{\operatorname*{cont}}C_{\operatorname*{inv}}^{2}h_{\operatorname*{c}}^{-2} 
+\frac
{2}{\Delta t^{2}}\frac{\left(  2+\nu\right)  }{1+\nu}\right)  \left\Vert
u_{\operatorname*{f}}\right\Vert _{\mathcal{T}}^{2}\\
&\overset{(\ref{Rfest})}{\leq}  \left(  C_{\operatorname*{cont}}C_{\operatorname*{inv}}^{2}
\left(  3 + \frac{C_{\operatorname*{cont}}%
}{c_{\operatorname*{coer}}}\right)
h_{\operatorname*{c}}^{-2} + \frac{2}{\Delta t^{2}}\frac{\left(  2+\nu\right)  }{1+\nu}\right)
\left\Vert u\right\Vert _{\mathcal{T}}^{2}.
\end{align*}

From this, we conclude by (\ref{CFLtotal}) that the maximal eigenvalue $\lambda_{p,\nu}^{\max}$
satisfies
\begin{align*}
\Delta t^{2}\lambda_{p,\nu}^{\max}  &  \leq C_{\operatorname*{cont}} C_{\operatorname*{inv}}^{2}
\left(  3 + \frac{C_{\operatorname*{cont}}%
}{c_{\operatorname*{coer}}}\right)  
\left(  \frac{\Delta t}{h_{\operatorname*{c}}}\right)  ^{2}%
+2\frac{\left(  2+\nu\right)  }{1+\nu}\\
&  \leq4-\frac{\nu}{\nu+1}.
\end{align*}

\textbf{(ii) Lower bound for the smallest eigenvalue}

From the (trivial) identity,%
\[
\left(  \mathbb{A}^{S,p,\nu}\right)  ^{-1}=\left(  \mathbb{A}^{S}\right)
^{-1}+\left(  \left(  \mathbb{A}^{S,p,\nu}\right)  ^{-1}-\left(
\mathbb{A}^{S}\right)  ^{-1}\right),
\]
combined with (\ref{lowevAs}) and Lemma \ref{LemInvDiff}, we immediately obtain%
\[
\left(  \mathbb{A}^{S,p,\nu}\right)  ^{-1}\leq\left(  \frac
{C_{\operatorname*{eq}}^{2}}{c_{\operatorname*{coer}}}+\frac{\nu+1}{2 \nu}\Delta
t^2\right)  \mathbb{I},
\]
where $\mathbb{I}$ denotes the identity operator.
\end{proof}

\begin{corollary}
Let the assumptions of Theorems \ref{TheoCont} hold, together with the CFL condition (\ref{CFLtotal}). 
Then, the discrete energy
\begin{equation}
E^{n+1/2} = \frac{1}{2} \left\{ \left( \frac{u_S^{(n+1)}-u_S^{(n)}}{\Delta t} , \frac{u_S^{(n+1)}-u_S^{(n)}}{\Delta t} \right)_{\mathcal{T}} + a^{p,\nu} \left( u_S^{(n+1)} ,  u_S^{(n)} \right) \right\}
\label{DefDiscNrj}
\end{equation}
is non-negative and conserved by the stabilized LF-LTS method (\ref{leap_frog_lts_fd}).
\end{corollary}

\begin{proof}
The assumptions imply those for Theorem 3.7, where
 positive lower bounds on the eigenvalues of the bilinear form $a^{p,\nu}$
were derived, and thus $E^{n+1/2} \geq 0$.
From the symmetry of $a^{p,\nu}$, proved in Theorem \ref{TheoCont}, we then conclude that the energy is conserved -- see e.g.~\cite{COHEN}.
\end{proof}

\subsection{Error equation and estimates}

{To derive a priori error estimates for the stabilized LTS/FE-Galerkin solution of (\ref{leap_frog_lts_fd}), 
we first introduce%
\begin{equation*}
v_{S}^{\left(  n+1/2\right)  }:=\frac{u_{S}^{\left(  n+1\right)  }%
-u_{S}^{\left(  n\right)  }}{\Delta t}, 
\end{equation*}
and rewrite (\ref{leap_frog_lts_fd}) as a one-step method
\begin{equation}%
\begin{aligned}
\left(  v_{S}^{\left(  n+1/2\right)  },q\right)_{\mathcal{T}}   &  = \left(  v_{S}^{\left(
n-1/2\right)  },q\right)_{\mathcal{T}}  -\Delta t\, a^{p,\nu}\left(  u_{S}^{\left(  n\right)
},q\right)   \; & &\forall q\in
S,\\
\left(
u_{S}^{\left(  n+1\right)  },r\right)_{\mathcal{T}}   &  = \left(  u_{S}^{\left(  n\right)
},r\right)_{\mathcal{T}} + \Delta t\left(  v_{S}^{\left(  n+1/2\right)  },r\right)_{\mathcal{T}}  \; & &\forall r\in S,\\
\left(  u_{S}^{\left(  0\right)  },w\right)_{\mathcal{T}}   &  = \left(  u_{0},w\right)_{\mathcal{T}} & &\forall w\in S,\\
\left(  v_{S}^{\left(  1/2\right)  },w\right)_{\mathcal{T}}   &  =\left(  v_{0},w\right)_{\mathcal{T}}
- \frac{\Delta t}{2} a\left(
u_{0},w\right)   \;& &\forall w\in S.
\end{aligned}
\label{eq1}%
\end{equation}
}

The first two equations in (\ref{eq1}) correspond to the one-step iteration
\begin{equation*}
\left[
\begin{array}{c}
v_S^{(n+1/2)} \\
u_S^{(n+1)}
\end{array}
\right] = \mathfrak{B} \left[
\begin{array}{c}
v_S^{(n-1/2)} \\
u_S^{(n)}
\end{array}
\right],\quad n \geq 1
\end{equation*}
with
\begin{equation*}
\mathfrak{B}:=
\left[
\begin{array}[c]{cc}%
I_{S} & -\Delta t A^{S,p,\nu}\\
\Delta t \, I_{S} & I_{S}-\Delta t^{2} A^{S,p,\nu}%
\end{array}
\right]  .
\end{equation*}
and $A^{S,p,\nu}$ as in (\ref{defASp}).

Next, we denote the error by
\[
\mathbf{e}^{\left(  n+1\right)  }:=
\begin{bmatrix}
e_{v}^{\left(n+1/2\right)  } \\
e_{u}^{\left(  n+1\right)  }
\end{bmatrix} := 
\begin{bmatrix}
v\left(  t_{n+1/2}\right)  - v_{S}^{\left(
n+1/2 \right)  } \\
u\left(t_{n+1}\right) - u_{S}^{\left(  n+1 \right)  }
\end{bmatrix},
\]
where $u$ is the solution of \eqref{waveeq}-\eqref{waveeqic} and $v$ the
solution of the corresponding first-order formulation: {Find $u,v: [0,T]
\rightarrow V$ such that%
\begin{align*}%
\left(  \dot{v},w\right)  +a\left(  u,w\right)   &  = 0
& \quad\forall\,w\in V,\quad t>0,\\
\left(  v,w\right)   &  =\left(  \dot{u},w\right)  & \quad\forall\,w\in V,\quad
t>0,
\end{align*}
with initial conditions $u(0)=u_{0}$ and $v(0)=v_{0}$.}

We shall now split the error into a semi-discrete and a fully discrete contribution.
To do so, we introduce the first-order formulation of the
semi-discrete problem (\ref{spacedisc}). Find $u_{S},v_{S}:\left[  0,T\right]
\rightarrow S$ such that%
\[%
\begin{array}[c]{cl}%
\left.
\begin{array}[c]{rcl}%
\left(  \dot{v}_{S},w\right)_{\mathcal{T}}  +a\left(  u_{S},w\right) & = & 0 \\
\left(  v_{S},w\right)_{\mathcal{T}}  & = & \left(  \dot{u}_{S},w\right)_{\mathcal{T}}
\end{array}
\right\}  & \forall w\in S,\quad t>0,\\
\begin{array}[c]{rcl}%
\qquad u_{S}\left(  0\right) & = & r_{S}  u_{0} \\
\qquad v_{S}\left(  0\right) & = & r_{S}  v_{0}.
\end{array}
\end{array}
\]
Hence, we may write 
\begin{equation}
\mathbf{e}^{\left(  n+1\right)  } = 
\mathbf{e}_{S}^{\left(  n+1\right)  }+\mathbf{e}_{S,\Delta t}^{\left(n+1\right)  }
\label{defspliterror}
\end{equation}
with%
\begin{align*}
\mathbf{e}_{S}^{\left(  n+1\right)  }  &  :=\left[
\begin{array}[c]{c}%
e_{v,S}^{\left(  n+1/2\right)  }\\
e_{u,S}^{\left(  n+1\right)  }%
\end{array}
\right]  :=\left[
\begin{array}[c]{c}%
v\left(  t_{n+1/2}\right)  -v_{S}\left(  t_{n+1/2}\right) \\
u\left(  t_{n+1}\right)  -u_{S}\left(  t_{n+1}\right)
\end{array}
\right]  ,\\
\mathbf{e}_{S,\Delta t}^{\left(  n+1\right)  }  &  :=\left[
\begin{array}[c]{c}%
e_{v,S,\Delta t}^{\left(  n+1/2\right)  }\\
e_{u,S,\Delta t}^{\left(  n+1\right)  }%
\end{array}
\right]  :=\left[
\begin{array}[c]{c}%
v_{S}\left(  t_{n+1/2}\right)  -v_{S}^{\left(  n+1/2\right)  }\\
u_{S}\left(  t_{n+1}\right)  -u_{S}^{\left(  n+1\right)  }%
\end{array}
\right]  .
\end{align*}

Following similar arguments as in \cite[section 3.2]{grote_sauter_1}, we derive a recurrence relation for $\mathbf{e}_{S,\Delta t}^{(n+1)}$, which can be solved and eventually yields the explicit error representation
\begin{align}
\left[
\begin{array}[c]{c}%
e_{v,S,\Delta t}^{\left(  n+1/2\right)  }\\
e_{u,S,\Delta t}^{\left(  n+1\right)  }%
\end{array}
\right]   
&  = \mathfrak{B}^{n}\left[
\begin{array}[c]{c}%
e_{v,S,\Delta t}^{\left(  1/2\right)  }\\
e_{u,S,\Delta t}^{\left(  1\right)  }%
\end{array}
\right]  
+ \Delta t \sum_{\ell=1}^{n-1}\mathfrak{B}^{\ell}
\left( 
\mbox{\boldmath$ \sigma$}^{\left( n-\ell \right) } - 
\mbox{\boldmath$ \sigma$}^{\left(  n-\ell+1\right)  } \right) \label{errorreprfu}\\
&\quad  +\Delta t \mbox{\boldmath$ \sigma$}^{\left(  n\right)  } - \Delta t\mathfrak{B}^{n} \mbox{\boldmath$ \sigma$}^{\left(  1\right)  }, \nonumber%
\end{align}
where
\begin{equation}%
\mbox{\boldmath$ \sigma$}^{\left(  n\right)  }
= \frac{1}{\Delta t} \left[
\begin{array}[c]{c}%
-\dfrac{u_{S}\left(  t_{n+1}\right)  -u_{S}\left(  t_{n}\right)  }{\Delta
t}+v_{S}\left(  t_{n+1/2}\right) \\
u_{S}\left(  t_{n}\right)  
+ \left(A^{S,p,\nu}\right)^{-1}   \dfrac{v_{S}\left(
t_{n+1/2}\right)  -v_{S}\left(  t_{n-1/2}\right)  }{\Delta t} 
\end{array}
\right]. \label{defboldsigmab}
\end{equation}

\subsubsection{Stability}

The convergence analysis consists of an estimate for the stability of the iteration operator $\mathfrak{B}$  and a consistency estimate. 
We begin with the stability analysis.

\begin{theorem}
[Stability]\label{TheoStabLeapFrog}
Let the CFL condition (\ref{CFLtotal}) be
satisfied. Then the leapfrog scheme (\ref{leap_frog_lts_fd}) is stable,%
\[
\left\Vert v_{S}^{\left(  n{+1/2}\right)  }\right\Vert_{\mathcal{T}} +\left\Vert
u_{S}^{\left(  n\right)  }\right\Vert_{\mathcal{T}} \leq C_{0}\left(  \left\Vert
v_{S}^{\left(  1/2\right)  }\right\Vert_{\mathcal{T}} +\left\Vert u_{S}^{\left(  1\right)
}\right\Vert_{\mathcal{T}} \right)  ,
\]
where $C_{0}$ is independent of $n$, $\Delta t$, $h$, and $T$.
\end{theorem}

\begin{proof}
Using the eigensystem of $a^{p,\nu}$ introduced in Theorem \ref{Coreigsys}, the proof follows along the same lines as the proof of \cite[Theorem 15]{grote_sauter_1}.
In the present case, however, the eigenfunctions $\eta_{p,\nu}$ are orthonormal with respect to the scalar product $\left(\cdot,\cdot\right)_{\mathcal{T}}$.
\end{proof}

\subsubsection{Error Estimates}

In this section we first estimate the discrete error $e_{u,S,\Delta t}^{\left(  n+1\right)  }$. 
Standard estimates on the semi-discrete error and the error due to mass-lumping then lead to an estimate of the total error $e_{u}^{(n+1)}$.

\begin{theorem}
\label{Theotimedisc}
Assume that (\ref{wellposed}), (\ref{hsmallerh0}), (\ref{interpolationestimate}), (\ref{ceqCeq}),
(\ref{hfphc}) and (\ref{CFLtotal}) hold and  
let the solution of the semi-discrete equation (\ref{spacedisc}) satisfy 
$u_{S}\in W^{5,\infty}\left(  \left[  0,T\right]  ;L^{2}\left(\Omega\right)  \right)  $. 
Then the fully discrete solution $u_{S}^{\left(  n+1\right)  }$ of (\ref{leap_frog_lts_fd}) satisfies the error estimate%
\[
\left\Vert e_{u,S,\Delta t}^{\left(  n+1\right)  }\right\Vert \leq C\Delta
t^{2}\left(  1+T\right)  \mathcal{M}\left(  u_{S}\right)
\]
with
\begin{equation*}
\mathcal{M}\left(  u_{S} \right)  := \max_{2\leq\ell\leq
5}\left\Vert \partial_{t}^{\ell}u_{S}\right\Vert _{L^{\infty}\left(  \left[
0,T\right]  ; L^2(\Omega)  \right)  }  
\end{equation*}
and a constant $C$, which is independent of $n$, $\Delta t$, $T$, $h$, $p$ and $u_{S}$.
\end{theorem}

\begin{proof}
The proof follows the lines of the proof of \cite[Theorem 16]{grote_sauter_1}, yet with
a few modifications to obtain estimates (\ref{estsigman}) and (\ref{estdiffsigmam}) below.
In particular, the proof now requires the new estimates of inverse operators in Lemma \ref{LemInvDiff}.

We apply the stability estimate to the second component of the error
representation (\ref{errorreprfu}). From Theorem \ref{TheoStabLeapFrog} and
(\ref{defboldsigmab}) we obtain\footnote{For a pair of functions
$\mathbf{v}=\left(  v_{1},v_{2}\right)  ^{\intercal}\in S \times S$ we use the
notation $\left\Vert \mathbf{v}\right\Vert _{\ell^{1}}:=\left\Vert
v_{1}\right\Vert_{\mathcal{T}} +\left\Vert v_{2}\right\Vert_{\mathcal{T}} $.}%
\begin{align}
\left\Vert e_{u,S,\Delta t}^{\left(  n+1\right)  }\right\Vert_{\mathcal{T}}  &  \leq
C_{0}\left\Vert \mathbf{e}_{S,\Delta t}^{\left(  1\right)  }\right\Vert
_{\ell^{1}} + C_{0}\, \Delta t \sum_{\ell=1}^{n-1} \left\Vert 
\mbox{\boldmath$ \sigma$}%
^{\left(  n - \ell \right)} - 
\mbox{\boldmath$ \sigma$}%
^{\left(  n - \ell + 1 \right)  } \right\Vert _{\ell^{1}} \nonumber\\
&\quad + C_{0}\,\Delta t\left\Vert
\mbox{\boldmath$ \sigma$}%
^{\left(  1\right)  }\right\Vert _{\ell^{1}} + \Delta t\left\Vert
\mbox{\boldmath$ \sigma$}%
^{\left(  n\right)  }\right\Vert _{\ell^{1}}.\label{eestmainform}
\end{align}
First, we consider the last term on the right-hand side of (\ref{eestmainform}). If we define
\begin{equation*}
\psi_S^{(k)} = u_{S}\left(  t_{k}\right)  + \left(A^{S,p,\nu}\right)^{-1}  \ddot{u}%
_{S}\left(  t_{k}\right) 
\end{equation*}
for all $k \in \left\{ 1, 3/2, 2, 5/2,\ldots,n \right\}$, we obtain by a Taylor argument and Theorem \ref{Coreigsys},
\begin{equation*}
\mbox{\boldmath$ \sigma$}^{\left(  n\right)  } = \dfrac{1}{\Delta t} 
\begin{bmatrix}
0 \\
\psi_{S}^{(n)}
\end{bmatrix} 
+ \dfrac{\Delta t}{24} \mathcal{R}_{n}^{\operatorname{I}},
\end{equation*}
where
\[
\left\Vert \mathcal{R}_{n}^{\operatorname{I}}\right\Vert _{\ell^{1}}%
\leq \left(  1+\frac{1}{c_{\nu}}\right)  \mathcal{M}%
_{n}\left(  u_{S}\right)
\]
with $c_\nu$ as defined in Theorem \ref{Coreigsys} and%
\[
\mathcal{M}_{n}\left(  u_{S} \right)  := \max_{2\leq\ell\leq5} \left\{ \max_{t \in \left[  t_{n-1/2},t_{n+2}\right]} \left\Vert \partial_{t}^{\ell}u_{S}(t) \right\Vert_{\mathcal{T}   } \right\}  .
\]
Since $u_S = - \left( A^S \right)^{-1} \ddot{u}_S$ (cf. (\ref{spacedisc})), we observe that
\begin{equation}
\psi_S^{(k)} = \left( \left( A^{S,p,\nu} \right)^{-1} - \left( A^S \right)^{-1} \right)  \ddot{u}_{S}\left(  t_{k}\right) .
\label{psiSk}
\end{equation}
Thus we can apply Lemma \ref{LemInvDiff} to estimate the norm of $\psi_S^{(k)}$ and it follows
\begin{equation}
\Delta t\left\Vert
\mbox{\boldmath$ \sigma$}%
^{\left(  n\right)  }\right\Vert _{\ell^{1}} \leq \dfrac{\Delta t^2}{24} \left( 24 \dfrac{\nu + 1}{2 \nu} + \dfrac{1}{c_\nu} + 1 \right) \mathcal{M}_{n}\left(  u_{S} \right).
\label{estsigman}
\end{equation}

Next, we investigate the summands in the second term of the right side of (\ref{eestmainform}), 
\begin{align*}
&  \mbox{\boldmath$ \sigma$}^{\left(  m\right)  } - \mbox{\boldmath$ \sigma$}^{\left(  m+1\right)  } =\\
& \!\!\left[  \!\!\!
\begin{array}[c]{c}%
\dfrac{ u_{S}\left(t_{m+2}\right) - 2 u_{S}\left(t_{m+1}\right) + u_{S}\left(t_{m}\right) }{\Delta t^{2}}
+ \dfrac{ v_{S}\left(t_{m+1/2}\right) - v_{S}\left(t_{m+3/2}\right) }{\Delta t}\\
\dfrac{ u_{S}\left(t_{m}\right) - u_{S}\left(t_{m+1}\right) }{\Delta t} 
- \left( A^{S,p,\nu} \right)^{-1} \! \left( 
\dfrac{ v_{S}\left(t_{m+3/2}\right) - 2 v_{S}\left(t_{m+1/2}\right) + v_{S}\left(t_{m-1/2}\right) }{\Delta t^{2}} 
\right)
\end{array}
\!\!\!\right].
\end{align*}
Again, by a Taylor argument and Theorem \ref{Coreigsys}, we deduce for any $m \geq1$%
\begin{equation*}
\mbox{\boldmath$ \sigma$}^{\left(  m\right)  }-%
\mbox{\boldmath$ \sigma$}^{\left(  m+1\right)  } = \left[
\begin{array}[c]{c}%
0\\
- \dot{\psi}_{S}^{\left(m+1/2\right)}
\end{array}
\right]  
+\frac{\left(  \Delta t\right)  ^{2}}{24}\mathcal{E}_{m}^{\operatorname{I}} 
\end{equation*}
with%
\[
\left\Vert \mathcal{E}_{m}^{\operatorname{I}}\right\Vert _{\ell^{1}}%
\leq \left(  4+\frac{2}{c_{\nu}}\right)  \mathcal{M}_{m}\left(  u_{S} \right).
\]
We now apply (\ref{psiSk}) and Lemma \ref{LemInvDiff} to $\dot{\psi}_{S}^{\left(m+1/2\right)}$ which yields the estimate
\begin{equation}
\left\| 
\mbox{\boldmath$ \sigma$}^{\left(  m\right)  } -\mbox{\boldmath$ \sigma$}^{\left(  m+1\right)  } 
\right\|_{\ell^{1}} 
\leq \dfrac{\Delta t^2}{24} \left( 12 \dfrac{\nu + 1}{\nu} + \dfrac{2}{c_\nu} + 4 \right) \mathcal{M}_{m}\left(  u_{S} \right).
\label{estdiffsigmam}
\end{equation}

Inserting (\ref{estsigman}) and (\ref{estdiffsigmam}) into (\ref{eestmainform}) leads to%
\begin{align}
\left\Vert e_{u,S,\Delta t}^{\left(  n+1\right)  }\right\Vert_{\mathcal{T}}  &\leq
C_{0} \left\Vert \mathbf{e}_{S,\Delta t}^{\left(  1\right)  }\right\Vert_{\ell^{1}} + C_{0} \frac{\Delta t^{2}}{24} \left( 12 \dfrac{\nu + 1}{\nu} + \dfrac{2}{c_\nu} + 4 \right) \Delta t \sum_{\ell=1}^{n-1}\mathcal{M}_{n-\ell}\left(  u_{S} \right)
\nonumber\\
& + \frac{ \Delta t^{2} }{24} \left( 24 \dfrac{\nu + 1}{\nu} + \dfrac{1}{c_\nu} + 1 \right)  \left(  \mathcal{M}_{n}\left(  u_{S} \right)  + C_{0}\mathcal{M}_{1}\left(  u_{S} \right)  \right) \nonumber\\
&  \overset{(\ref{ceqCeq})}{\leq} C_{0}\left\Vert \mathbf{e}_{S,\Delta t}^{\left(  1\right)
}\right\Vert _{\ell^{1}} \label{ruspu}\\
& + \frac{ \Delta t^{2} }{24} C_{\operatorname*{eq}} \left(
C_{0}T \left( 12 \dfrac{\nu + 1}{\nu} + \dfrac{2}{c_\nu} + 4 \right)  + \left( 12 \dfrac{\nu + 1}{\nu} + \dfrac{1}{c_\nu} + 1 \right) \left( 1+C_{0} \right) \right)  \mathcal{M}\left(  u_{S} \right).
\nonumber
\end{align}
To estimate the initial error $\mathbf{e}_{S,\Delta t}^{\left(
1\right)  }$, we again use a Taylor argument as in \cite[Proof of Theorem 16]{grote_sauter_1}, which  leads to
\begin{equation} \label{initerr1}
\left\Vert u_{S}\left(  t_{1}\right)  -u_{S}^{\left(  1\right)  }\right\Vert_{\mathcal{T}}
\leq\frac{3}{2}\Delta t^{3} C_{\operatorname*{eq}} \mathcal{M}\left(  u_{S} \right)
\end{equation}
and
\begin{equation} \label{initerr2}
\left\Vert v_{S}\left(  t_{1/2}\right)  -v_{S}^{\left(  1/2\right)
}\right\Vert_{\mathcal{T}}  \leq\frac{3}{2} \Delta t^2 C_{\operatorname*{eq}} \mathcal{M}\left(  u_{S} \right)  .
\end{equation}
In summary, we have estimated the initial error by%
\begin{equation}
\left\Vert \mathbf{e}_{S,\Delta t}^{\left(  1\right)  }\right\Vert _{\ell^{1}%
} \leq \frac{3}{2} \Delta t^2 \left(  1+\Delta t\right)
C_{\operatorname*{eq}} \mathcal{M}\left(  u_{S} \right)  . \label{initerror}%
\end{equation}
The combination of (\ref{ruspu}), (\ref{initerror}) and (\ref{ceqCeq}) concludes the proof.%
\end{proof}

Theorem \ref{Theotimedisc} can be combined with known error estimates for the
semi-discrete error $\mathbf{e}_{S}^{\left(  n+1\right)  }$ to obtain an error
estimate of the total error.

\begin{theorem}
\label{TheoMain}
Assume that (\ref{wellposed}), (\ref{hsmallerh0}), (\ref{interpolationestimate}), (\ref{ceqCeq}), (\ref{hfphc}) and (\ref{CFLtotal}) hold 
and that the exact solution of (\ref{waveeq}) satisfies $ u\in W^{8,\infty}\left(  \left[  0,T\right] ; H^{m+1}\left(  \Omega\right)  \right)  $.
Then, the corresponding fully discrete Galerkin FE formulation with local
time-stepping (\ref{leap_frog_lts_fd}) has a unique solution $u_{S}^{\left(
n+1\right)  }$ which satisfies the error estimate%
\[
\left\Vert u(t_{n+1})-u_{S}^{\left(  n+1\right)  }\right\Vert \leq C\left(
1+T\right)  \left(  h^{m+1}+\Delta t^{2}\right)  \mathcal{Q}\left(
u \right)
\]
with%
\[
\mathcal{Q}\left(  u \right)  := \max_{2\leq\ell\leq5} \left\{ \left(
1+C_{\ell}^{\prime}h^{m+1}\left(  1+T\right)  \right)  \left\Vert u\right\Vert
_{W^{\ell+3,\infty}\left(  \left[  0,T\right]  ;H^{m+1}\left(  \Omega\right)
\right)  } \right\}
\]
and constants $C_{\ell}^{\prime}$ which are independent of $n$, $\Delta t$,
$h$, $p$, and the final time $T$.
\end{theorem}

\begin{proof}
The existence of the semi-discrete solution $u_{S}$ follows from \cite[Theorem 3.1]{Baker}, which directly implies the existence of the fully discrete stabilized LTS-Galerkin FE solution. 
Next, we split the total error
\[
\mathbf{e}^{\left(  n+1\right)  }=\left(  v\left(  t_{n+1/2}\right)
-v_{S}^{\left(  n+1/2\right)  },u\left(  t_{n+1}\right)  -u_{S}^{\left(
n+1\right)  }\right)  ^{\intercal}%
\]
into a semi-discrete and a fully discrete contribution, as in (\ref{defspliterror}). 

Before applying Theorem \ref{Theotimedisc} to bound the discrete error contribution, we must ensure that the semi-discrete solution $u_S$ is sufficiently regular.
In \cite{MuellerSchwab16} the regularity of solutions to wave equations is investigated in an
abstract setting for an evolution triplet of Hilbert spaces; 
hence, we can apply these results to the semi-discrete formulation (\ref{spacedisc}) with mass-lumping. 
Thus, we conclude that the semi-discrete solution $u_{S}$ inherits the same temporal regularity from $u\in W^{8,\infty}\left(  \left[0,T\right]  ;H^{m+1}\left(  \Omega\right)  \right)  $, 
i.e., $u_{S}\in W^{8,\infty}\left(  \left[  0,T\right]  ;H^{1}\left(  \Omega\right)  \right)$
and therefore we can apply Theorem \ref{Theotimedisc}. 

To estimate the remaining error 
\[
\mathbf{e}_{S}^{\left(  n+1\right)  } = 
\begin{bmatrix}
v\left(  t_{n+1/2}\right)
-v_{S}\left(  t_{n+1/2}\right)  \\
u\left(  t_{n+1}\right)  -u_{S}\left(
t_{n+1}\right)
\end{bmatrix}
\]
in the semi-discrete solution $u_{S}$ with mass-lumping, we shall apply \cite[Thm. 4.1]{Baker2}.
Hence, we now verify the assumptions stated in that Theorem. 
Since the space $S$ consists of piecewise polynomials, any $\varphi\in S$ satisfies 
$\left.\varphi\right\vert_{\tau} \in C^{m+1}\left(  \tau\right)$ for all $\tau \in \mathcal{T}$, which corresponds to condition \cite[(2.4)]{Baker2};
note that the index $k$ in \cite[Prop. 4.1]{Baker2} is denoted by $m$ in our paper. 
Assumption \cite[(2.5)]{Baker2} is equivalent to (\ref{interpolationestimate}).

Next, we verify the conditions in \cite[Thm. 4.1]{Baker2} related to the accuracy of mass-lumping. 
From \cite[Lem. 5.2]{mass_lumping_2d}, we conclude that for all
$\chi,\psi \in S$ and $1\leq r,\mu\leq m-1$ it holds%
\[
\left\vert \left(  \chi,\psi\right)  -\left(  \chi,\psi\right)  _{\mathcal{T}%
}\right\vert \leq Ch^{r+\mu}\left(  \sum_{\tau\in\mathcal{T}}\left\Vert
\chi\right\Vert _{H^{r+\mu}\left(  \tau\right)  }^{2}\right)  ^{1/2}\left(
\sum_{\tau\in\mathcal{T}}\left\Vert w\right\Vert _{H^{\mu}\left(  \tau\right)
}^{2}\right)  ^{1/2},
\]
which is condition \cite[(3.5)]{Baker2} for $1\leq r,\mu\leq m-1$ and $q=2$ 
(Note that $H^{r+\mu}\left(  \Omega\right)  \hookrightarrow C^{0}\left(
\overline{\Omega}\right)  $ for $d \leq 3$.). 
For $q\geq2$, and the same range of $r,\mu$, condition \cite[(3.5)]{Baker2} follows from the continuous embedding
$W^{r+\mu,q}\left(  \tau\right)  \hookrightarrow H^{r+\mu}\left(  \tau\right)  $.
Inspection of the proof of \cite[Lem. 5.2]{mass_lumping_2d} shows that the claim also holds for $r=0$ 
(by choosing $\widehat{\Pi}_{p-1}$ as the zero operator in \cite[(5.10)]{mass_lumping_2d}). 
Finally, the choice $\mu=1,2$ is always allowed since (\ref{ml_exactnesscondition}) ensures that \cite[(5.10)]{mass_lumping_2d} holds.

Since the Galerkin operator $A^{S}$ is defined via the (unperturbed) bilinear
form $\left(  A^{S}u,v\right)  _{\mathcal{T}}:=a\left(  u,v\right)  $,
conditions \cite[(3.6) and (3.7)]{Baker2} are automatically satisfied. 
Condition \cite[(3.8)]{Baker2} is also satisfied, because the right-hand side in (\ref{waveeq}) is zero, whereas \cite[(3.9)]{Baker2} follows from (\ref{ceqCeq}). 
Finally, the coercivity of the bilinear form in \cite[(3.10)]{Baker2} follows from (\ref{wellposedc}).
The hypotheses on regularity \cite[(4.22)]{Baker2} are satisfied by assumption.

Hence, we can use \cite[Theorem 4.1]{Baker2} to obtain
\begin{equation}
\left\Vert \partial_{t}^{\ell}\left(  u-u_{S}\right)  \right\Vert _{ L^{\infty} \left(  \left[  0,T\right]  ;L^{2}\left(  \Omega\right)  \right)  } \leq
C_{\ell}h^{m+1} \sum_{k=\ell}^{\ell+3}  \left\Vert \partial_{t}^{k}u\right\Vert _{L^{2}\left(  \left[
0,T\right]  ;H^{m+1}\left(  \Omega\right)  \right)  }  
\label{Bakerest}%
\end{equation}
for $\ell=0$. 
Inspection of the proof in \cite{Baker2} shows that the estimate also holds for
$\ell\geq1$, provided the right-hand side in (\ref{Bakerest}) exists.
By applying a H\"{o}lder inequality to the summands of the right-hand side in (\ref{Bakerest}), we obtain%
\[
\left\Vert \partial_{t}^{k}u\right\Vert _{L^{2}\left(  \left[
0,T\right]  ;H^{m+1}\left(  \Omega\right)  \right)  }\leq\sqrt{T}\left\Vert
\partial_{t}^{k}u\right\Vert _{L^{\infty}\left(  \left[  0,T\right]
;H^{m+1}\left(  \Omega\right)  \right)  },
\]
which implies for $\ell\leq 5$ that%
\[
\left\Vert \partial_{t}^{\ell}\left(  u-u_{S}\right)  \right\Vert _{L^{\infty
}\left(  \left[  0,T\right]  ;L^{2}\left(  \Omega\right)  \right)  }\leq
C_{\ell}^{\prime}h^{m+1} (1+T) \left\Vert u\right\Vert
_{W^{\ell+3,\infty}\left(  \left[  0,T\right]  ;H^{m+1}\left(  \Omega\right)
\right)  }%
\]
holds for a constant $C_{\ell}^{\prime}$ independent of the final time $T$. 
From Theorem \ref{Theotimedisc}, we thus obtain
\[
\left\Vert u(t_{n+1})-u_{S}^{\left(  n+1\right)  }\right\Vert \leq C\left(
1+T\right)  \left(  h^{m+1}+\Delta t^{2}\right)  \max\left\{  \mathcal{M}%
\left(  u_{S} \right)  ,\left\Vert u\right\Vert _{W^{3,\infty}\left(
\left[  0,T\right]  ;H^{m+1}\left(  \Omega\right)  \right)  }\right\}  .
\]

Finally, it remains to estimate $\mathcal{M}\left(  u_{S} \right)  $ in terms of $u$. 
A triangle inequality yields%
\begin{align*}
\left\Vert \partial_{t}^{\ell}u_{S}\right\Vert _{L^{\infty}\left(  \left[
0,T\right]  ; L^2(\Omega)  \right)  }  &  \leq
\left\Vert \partial_{t}^{\ell}u\right\Vert _{L^{\infty}\left(  \left[  0,T\right]
; L^2(\Omega) \right)  } 
+ \left\Vert \partial_{t}^{\ell} \left( u_{S}-u\right)  \right\Vert _{L^{\infty}\left(  \left[  0,T\right];L^2(\Omega)  \right)  }\\
&  \leq C_{\operatorname*{eq}} \left(  1+C_{\ell}^{\prime}h^{m+1}\left(  1+T\right)  \right)
\left\Vert u\right\Vert _{W^{\ell+3,\infty}\left(  \left[  0,T\right]
;H^{m+1}\left(  \Omega\right)  \right)  }.
\end{align*}
Therefore, we conclude that
\[
\max_{2\leq\ell\leq5}\left\Vert \partial_{t}^{\ell
}u_{S}\right\Vert _{L^{\infty}\left(  \left[  0,T\right]  ; L^2(\Omega)  \right)  } \leq C_{\operatorname*{eq}} \mathcal{Q}\left(  u\right)
\]
holds, which, together with the triangle inequality, leads to the assertion.
\end{proof}

\section{Numerical results}
Here we present a series of numerical examples, which illustrate the accuracy, enhanced stability and energy conservation of the \emph{stabilized LF-LTS} method from Section \ref{sec:stablflts}. 
First, we apply the stabilized LF-LTS method to a smooth initial condition and verify that it achieves the expected rate of convergence and conserves the energy.
Next, we purposely choose a critical, unstable time-step and initial condition for the original LF-LTS method without stabilization \cite{DiazGrote09} to demonstrate that the new stabilized version nonetheless remains stable for all time.

For these experiments, we consider the wave equation (\ref{model problem}) in the one-dimensional unit interval $\Omega = (0,1)$ with homogeneous Dirichlet boundary conditions, i.e. $\Gamma = \Gamma_D$, and $c \equiv 1$.
The computational domain $\Omega$ separates into a coarse part, $\Omega_{\operatorname*{c}} = (0,0.9)$, and a ``locally refined" part, $\Omega_{\operatorname*{f}} = [0.9,1)$. 
Inside $\Omega_{\operatorname*{c}}$ and $\Omega_{\operatorname*{f}}$, we use an equidistant mesh with mesh size $h_{\operatorname*{c}}$ or $h_{\operatorname*{f}}$, respectively, with $h_{\operatorname*{f}} = h_{\operatorname*{c}}/p$, $p \geq 2$.
For the spatial discretization, we always use piecewise linear $H^1$-conforming finite elements with mass-lumping, as described in Section \ref{RemMasslumping}.
This simple setting is appropriate to 
demonstrate the \emph{stability} and \emph{convergence} properties of the stabilized LF-LTS method.

To conclude this section, we consider an application of the stabilized LF-LTS
method on a two-dimensional domain with a reentrant corner where the solution
exhibits a characteristic singular behavior. We verify that optimal convergence
rates are retained if the mesh is suitably graded towards the reentrant corner.

\subsection{Convergence and energy conservation}
\label{SecNumStabLTS}

First, we consider a smooth solution of (\ref{model problem}) with initial condition
\begin{equation}
\left\{
\begin{aligned}
u_0(x) &= \exp\left(- 400 \left( x - \dfrac{1}{2} \right)^2 \right), \\
v_0(x) &= 0.
\end{aligned}
\right.
\label{IniCondGauss}
\end{equation}
We now apply the stabilized LF-LTS method with $h_{\operatorname*{c}}=0.01$, $h_{\operatorname*{f}} = h_{\operatorname*{c}}/1000$ in $\Omega_{\operatorname*{f}}$. 
Hence for each global time-step $\Delta t$, the LF-LTS method takes $p=1000$ local time-steps of size $\Delta t / p$ inside $\Omega_{\operatorname*{f}}$.
Clearly, a value as high as $p=1000$ is not necessarily intended for practical use but merely to demonstrate the robustness of the stabilized LF-LTS method with respect to $p$.

In Fig. \ref{FigNumSolT2}, we show the exact solution at $t=2$ together with the stabilized LF-LTS numerical solution with $\nu = 0.01$.
For the sake of comparison, we also show the numerical solution obtained with the original LF-LTS method \cite{DiazGrote09} without stabilization ($\nu = 0$).
Both numerical solutions, be it with or without stabilization, 
were computed using a time-step $\Delta t = \operatorname*{e}^{-\nu} \, h_{\operatorname*{c}}$ and
coincide at this scale with the exact solution.
\begin{figure}
\centering
\includegraphics[width=0.49\textwidth]{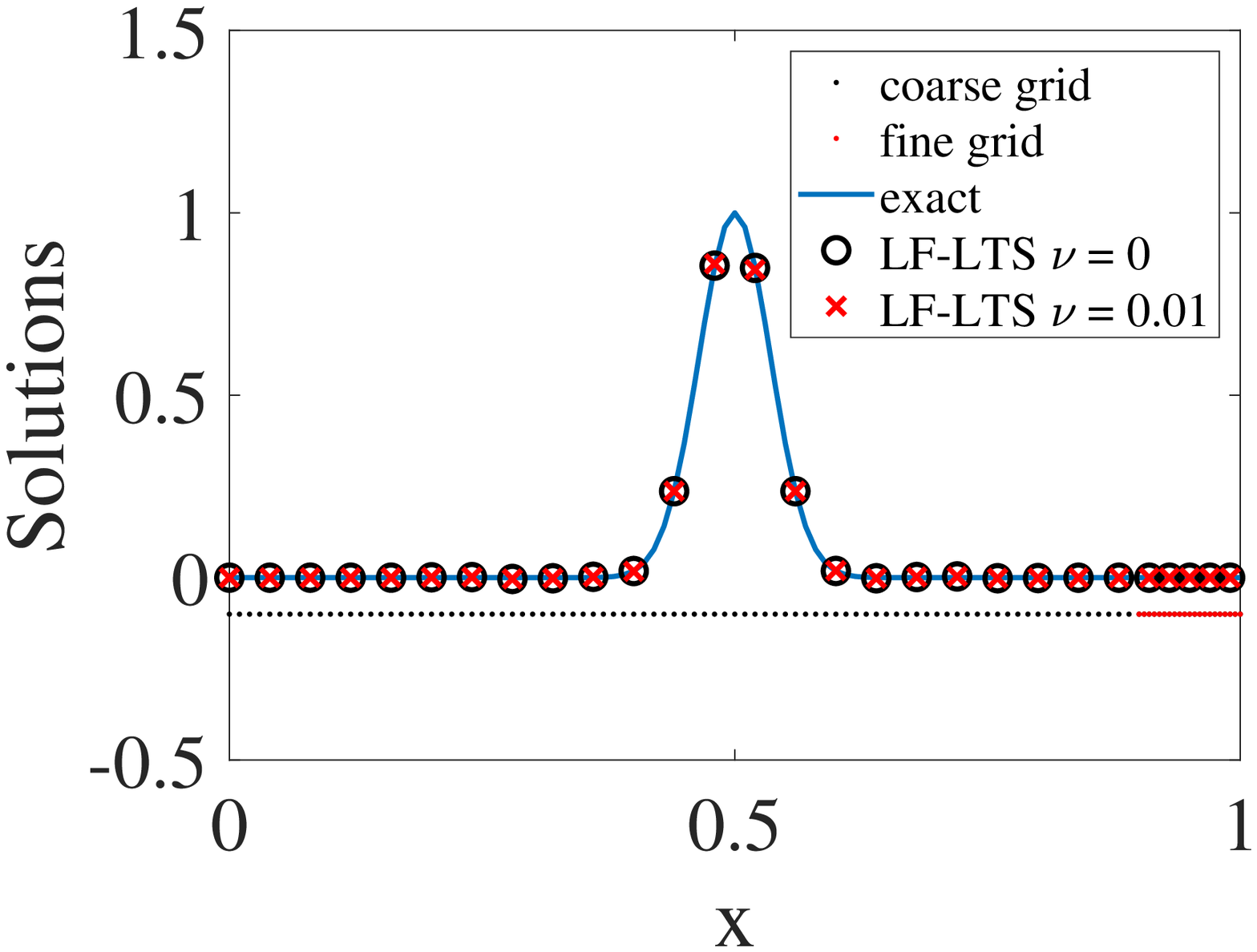}
\includegraphics[width=0.49\textwidth]{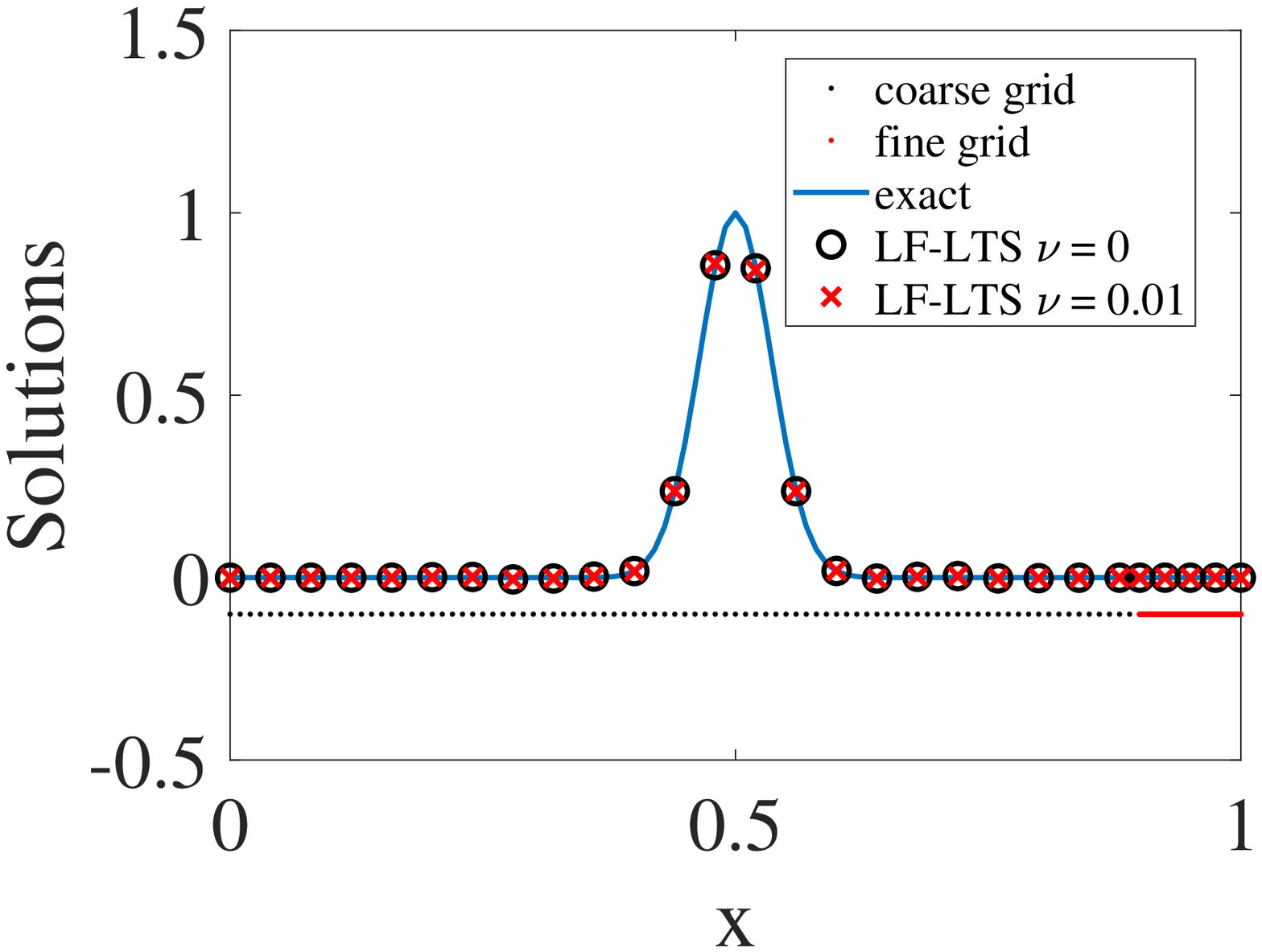}
\caption{LF-LTS solutions of (\ref{model problem}) at $t=2$ for $p=2$ (left) or $p=1000$ (right) either with ($\nu=0.01$, crosses) or without ($\nu=0$, circles) stabilization together with the exact solution (solid line). The dots below the $x$ axis mark the grid points in $\Omega_{\operatorname*{c}} = [0,0.9]$ and $\Omega_{\operatorname*{f}} = [0.9,1]$. }
\label{FigNumSolT2}
\end{figure}

Next, we test convergence of the stabilized scheme for a varying number of local time-steps $p$ on a sequence of locally refined meshes. 
As shown in Fig. \ref{FigError}, the absolute $L^2$-error at the final time $t=2$ converges as $\mathcal{O}(h_{\operatorname*{c}}^2)$ independently of $p$. 
Hence, the stabilized LF-LTS method indeed achieves the optimal convergence rate proved in Theorem \ref{TheoMain}.
\begin{figure}
\centering
\includegraphics[width=0.75\textwidth]{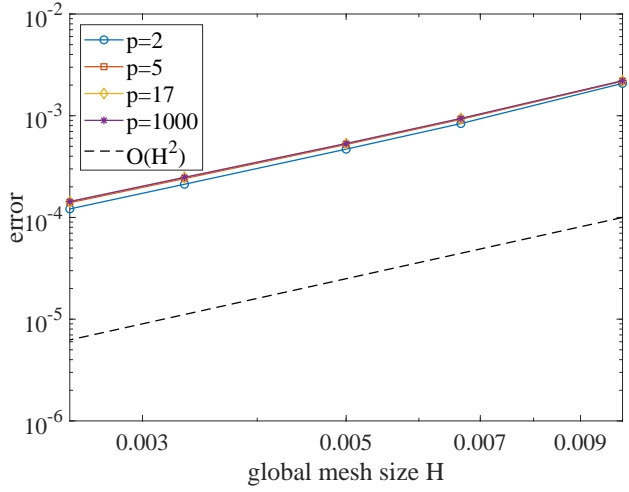}
\caption{Absolute $L^2$-error vs. $H = h_{\operatorname*{c}}$ for $\mathbb{P}_1$ finite elements for $\nu=0.01$ and $p=2,5,17,1000$.}
\label{FigError}
\end{figure}

In Fig. \ref{FigNrj}, we display the time evolution of the energy and relative deviation from its constant value for $h_{\operatorname*{c}} = 1/320$ and $p=2$ until $T=100$, that is, for over 30'000 (global) time-steps. 
During the entire simulation, the discrete energy $E^{n+1/2}$ defined in (\ref{DefDiscNrj}) indeed remains constant within machine precision.
\begin{figure}
\centering
\includegraphics[width=0.49\textwidth]{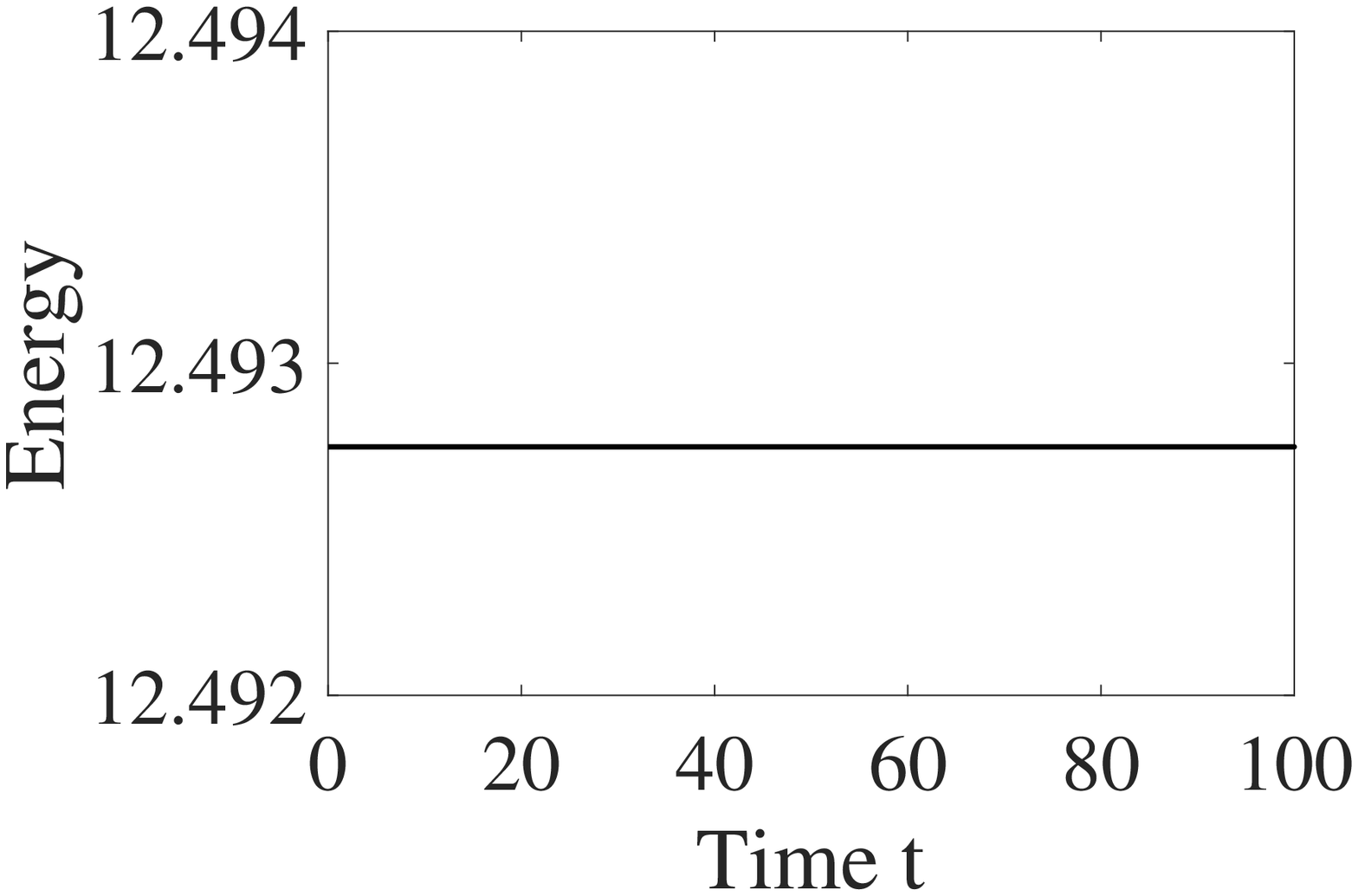}
\includegraphics[width=0.49\textwidth]{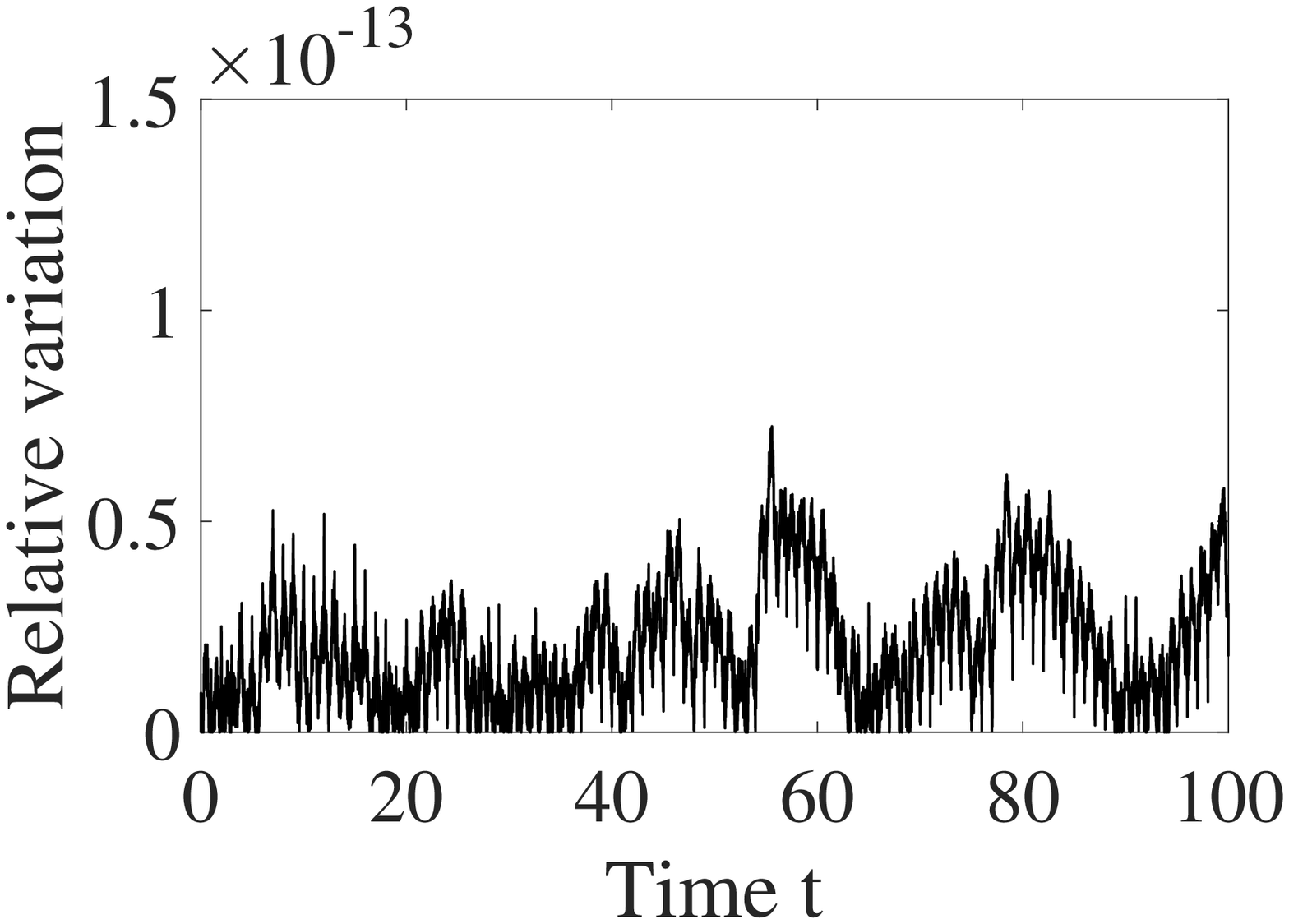}
\caption{Time evolution of the discrete energy $E^{n+1/2}$ as in (\ref{DefDiscNrj}) (left) and of its relative variation $\left| E^{n+1/2}/E^{1/2} - 1 \right|$ (right).}
\label{FigNrj}
\end{figure}

Finally, we determine experimentally the maximal time-step $\Delta t = \Delta t_{\operatorname*{LFLTS}(p,\nu)}$, for which the stabilized LF-LTS scheme remains stable, or more precisely, for which all eigenvalues of $\Delta t^2 {A}^{S,p,\nu}$ lie in $(0,4)$.
Ideally, that time-step should essentially be equal to the optimal time-step,
$\Delta t_{\operatorname*{opt}} = 2 / \sqrt{\lambda^{\max}}$, 
for the standard leapfrog scheme on an equidistant mesh; 
here $\lambda^{\max}$ denotes the largest (discrete) eigenvalue $\lambda$ of the bilinear form $a$, i.e.
\[
a(\eta,v) = \lambda\, (\eta,v)_{\mathcal{T}}, \qquad \forall v \in S,
\]
with corresponding eigenfunction $\eta\in S$.
As shown in Table \ref{TabCFLnue}, the allowed maximal time-step only slightly decreases for small values of $\nu$, such as $\nu=0.01$, where we lose only 0.3 \% over  $\Delta t_{\operatorname*{opt}}$.
Even for a stabilization parameter as large as $\nu=0.5$, the CFL number decreases by less than 13 \%.
In all cases, the critical eigenvalue of $\Delta t^2 {A}^{S,p,\nu}$ remains strictly below $4$, as shown in the last row of Table \ref{TabCFLnue}. 
The negligible reduction of the allowed maximal time-step indicates that the stability condition \eqref{CFLtotal}, needed for the analysis and increasingly stringent for small $\nu$, in fact seems irrelevant in practice.

\begin{table}
\centering
\caption{The maximal time-step $\Delta t_{\operatorname*{LFLTS}(p,\nu)}$ in percent of $\Delta t_{\operatorname*{opt}}$ and the minimal distance from the stability limit at 1 of $\Delta t^2 \lambda_{p,\nu}/4$ for $\Delta t \leq \Delta t_{\operatorname*{LFLTS}(p,\nu)}$.}
\begin{tabular}{ c||c|c|c|c }
$\nu$ & 0.001 & 0.01 & 0.05 & 0.5 \\
\hline
\hline
\rule{0pt}{18pt}
$\dfrac{\Delta t_{\operatorname*{LFLTS}(p,\nu)}}{\Delta t_{\operatorname*{opt}}}$ & 99.9 \% & 99.7 \% & 98.4 \% & 87.6 \%
\rule[-13pt]{0pt}{18pt} \\
\hline
\rule{0pt}{18pt}
$\min \left\{ 1 - \dfrac{\Delta t^2}{4} \lambda_{p,\nu} \right\}$ & 0.0005 & 0.0049 & 0.0239 & 0.1758
\rule[-13pt]{0pt}{18pt}
\end{tabular}
\label{TabCFLnue}
\end{table}

\subsection{Stabilized vs. non-stabilized LTS scheme}
\label{SecNumCompa}

The LF-LTS method is stable for any particular $\Delta t$, 
as long as all the discrete eigenvalues of $(\Delta t^2/4) A^{S,p,\nu}$ strictly lie in $(0,4)$.
Hence to illustrate the stabilizing effect of $\nu$, we now study the behavior of those eigenvalues either with ($\nu>0$) or without ($\nu=0$) stabilization
for a coarse mesh size of $h_{\operatorname*{c}} = 1/40$ and a locally refined part with $h_{\operatorname*{f}} = h_{\operatorname*{c}}/p,$ $p=3$. 

For $\nu=0$, the eigenvalues shown in the left column of Fig. \ref{FigEVDtApnue} cut across the stability threshold line at $1$ only for $\Delta t > \Delta t_{\operatorname*{opt}} \sim h_{\operatorname*{c}}$. 
However, at certain discrete values of $\Delta t$ as small as $\Delta t_{\operatorname*{crit}} = 0.5011 \cdot \Delta t_{\operatorname*{opt}}$, the first potentially unstable, critical time-step, some of the eigenvalue curves in fact are tangent to the stability thresholds at zero or one.
Those unstable, critical time-steps $\Delta t$ are isolated and hard to detect: 
choose any time-step slightly smaller, or larger, and the LTS-LF method will indeed remain perfectly stable.
Still, for these particular choices of $\Delta t$, the LF-LTS method (with $\nu = 0$) may become (linearly) unstable. 

In contrast for $\nu = 0.001$, we observe in the right column of Figure \ref{FigEVDtApnue} that all eigenvalues remain strictly inside the open interval $(0,1)$.
By setting $\nu=0.001$, we have thus removed all potentially unstable critical time-steps present in the LF-LTS method without stabilization.
\begin{figure}
\centering
\includegraphics[width=0.49\textwidth]{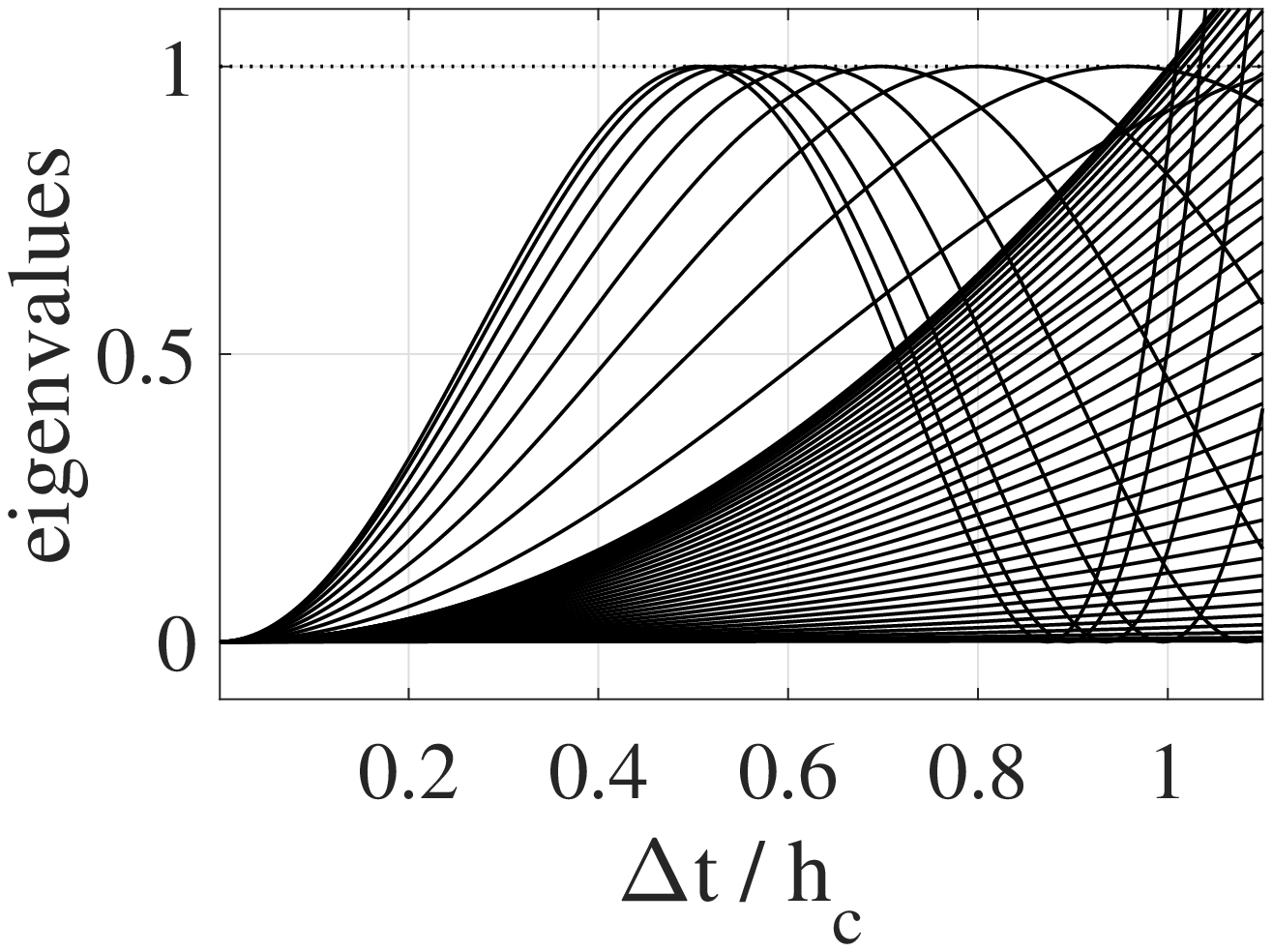}
\includegraphics[width=0.49\textwidth]{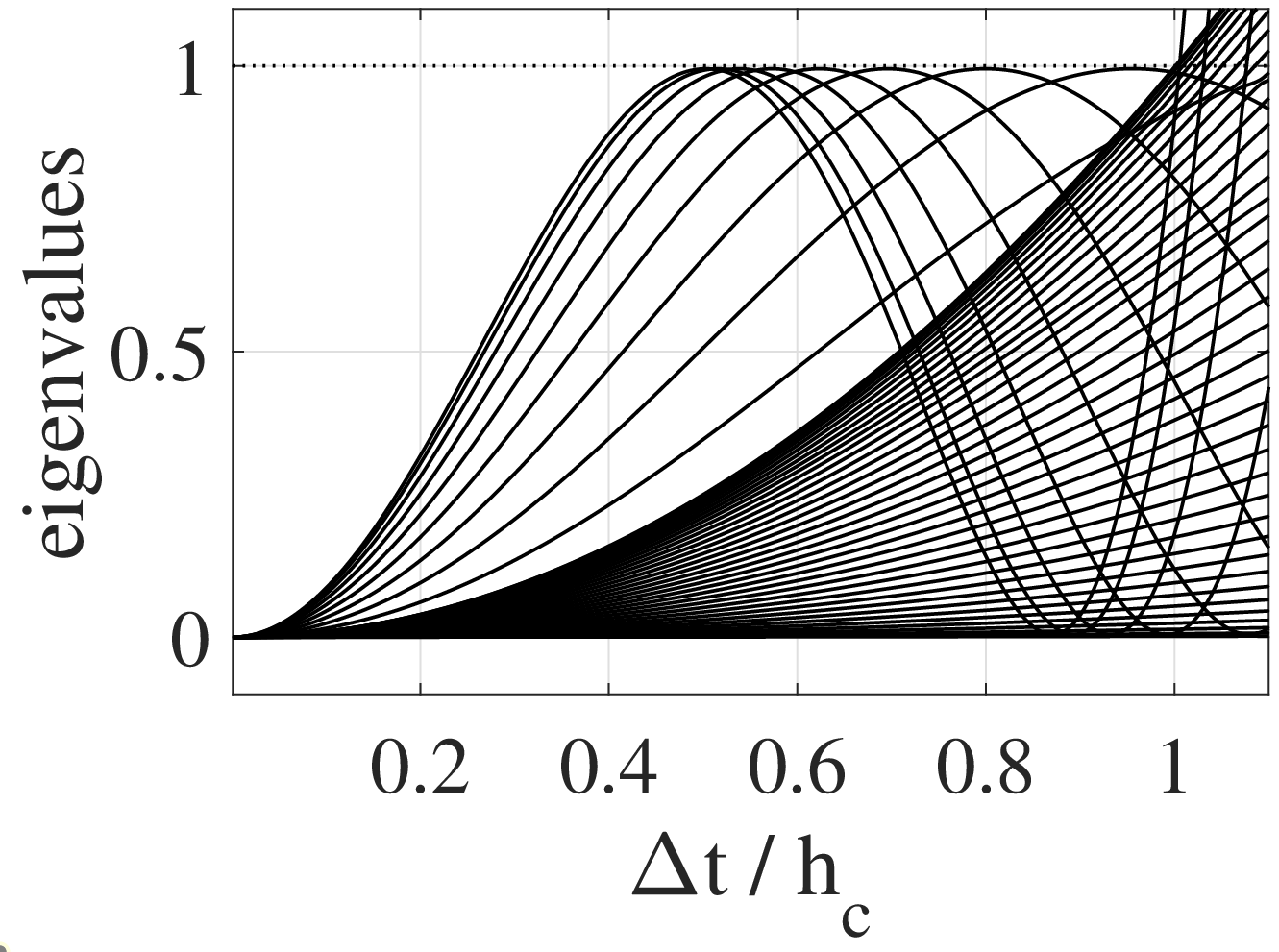}
\includegraphics[width=0.49\textwidth]{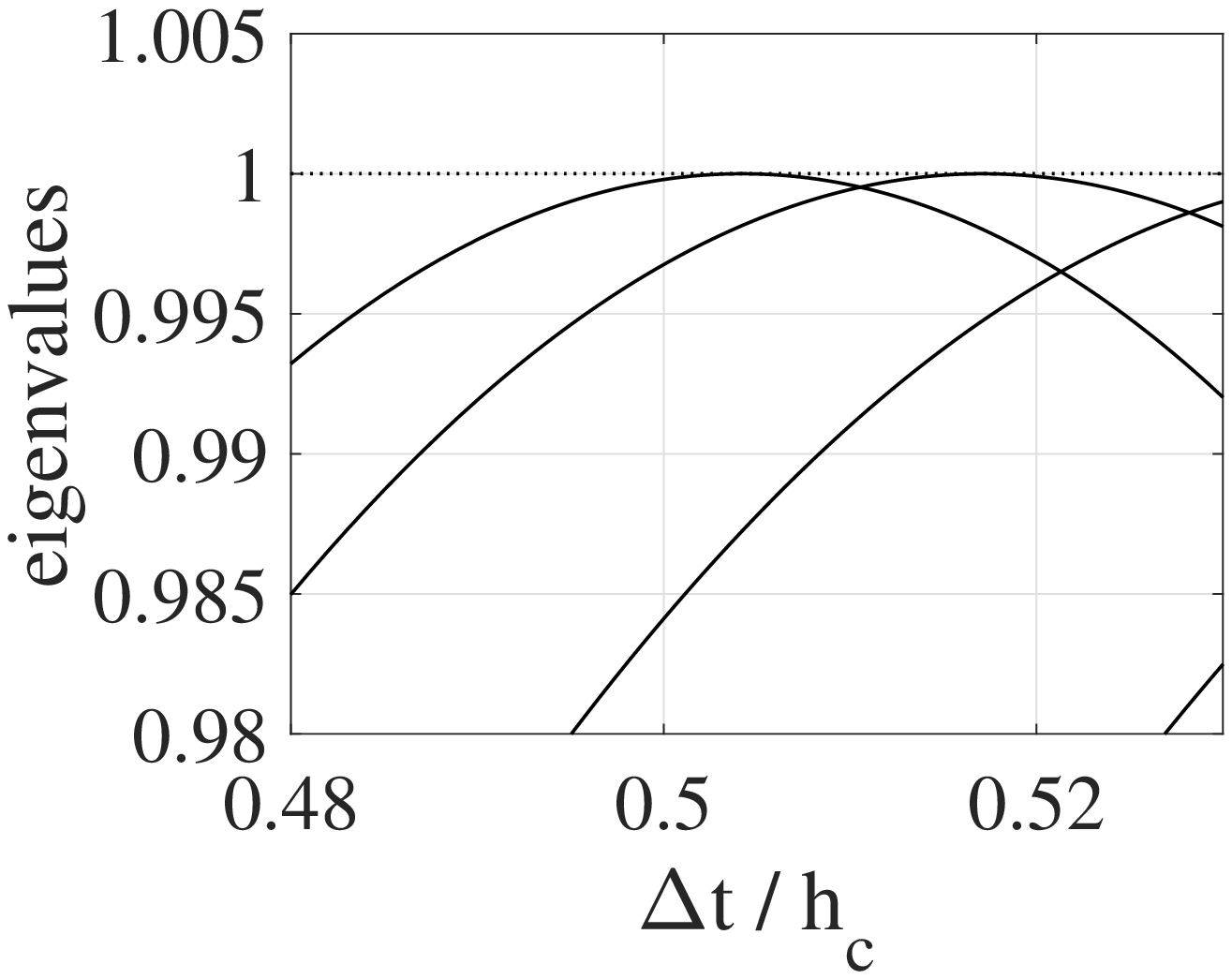}
\includegraphics[width=0.49\textwidth]{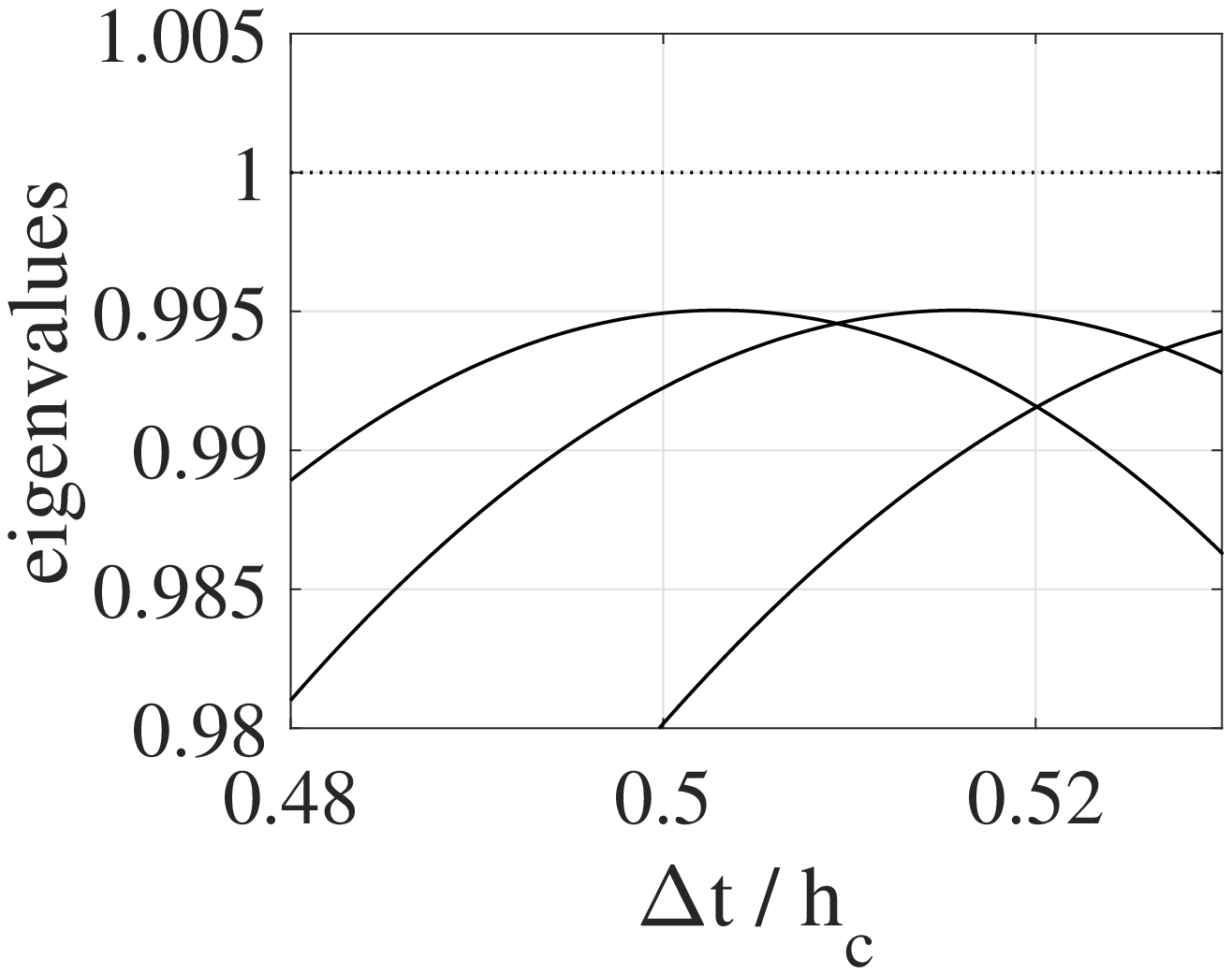}
\includegraphics[width=0.49\textwidth]{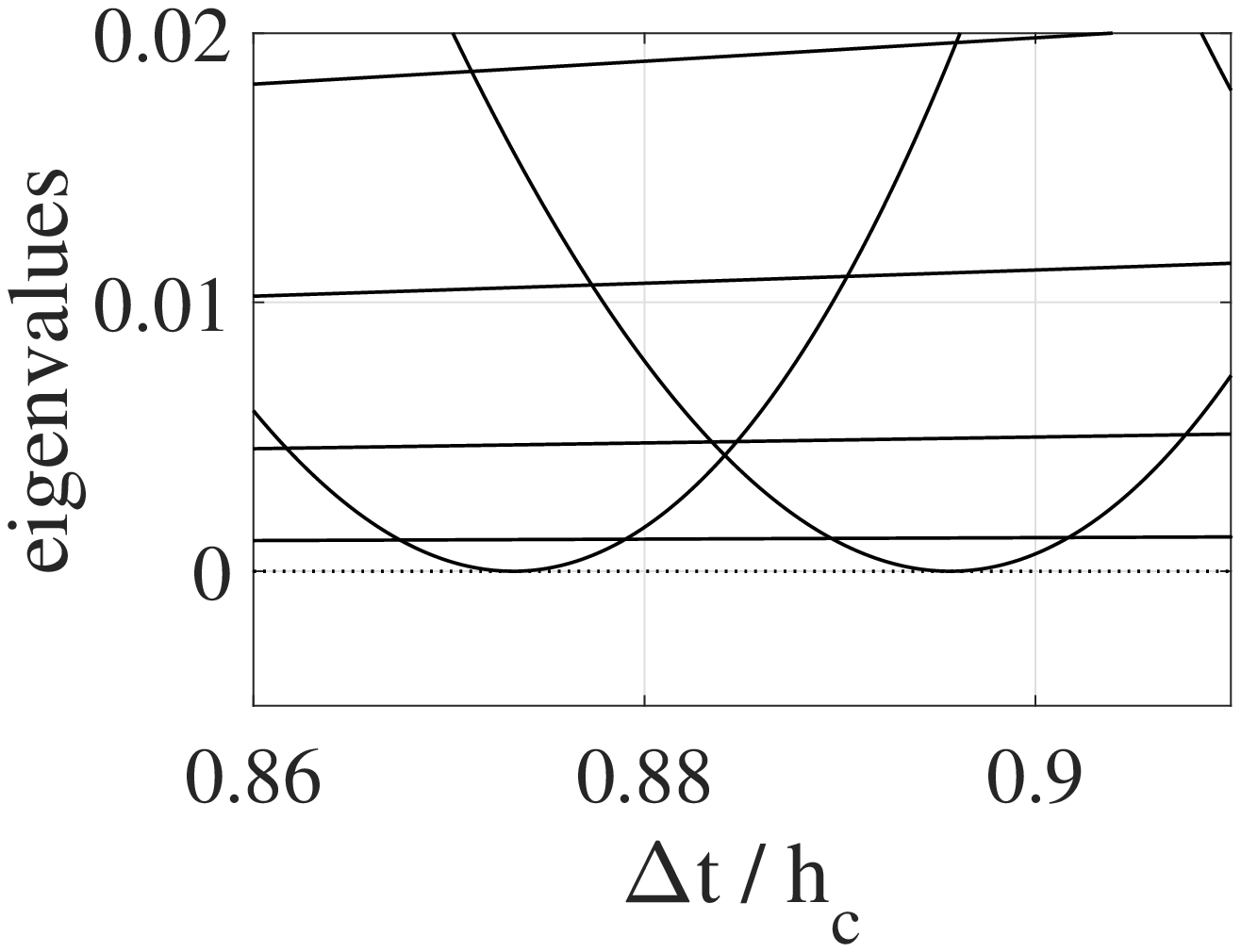}
\includegraphics[width=0.49\textwidth]{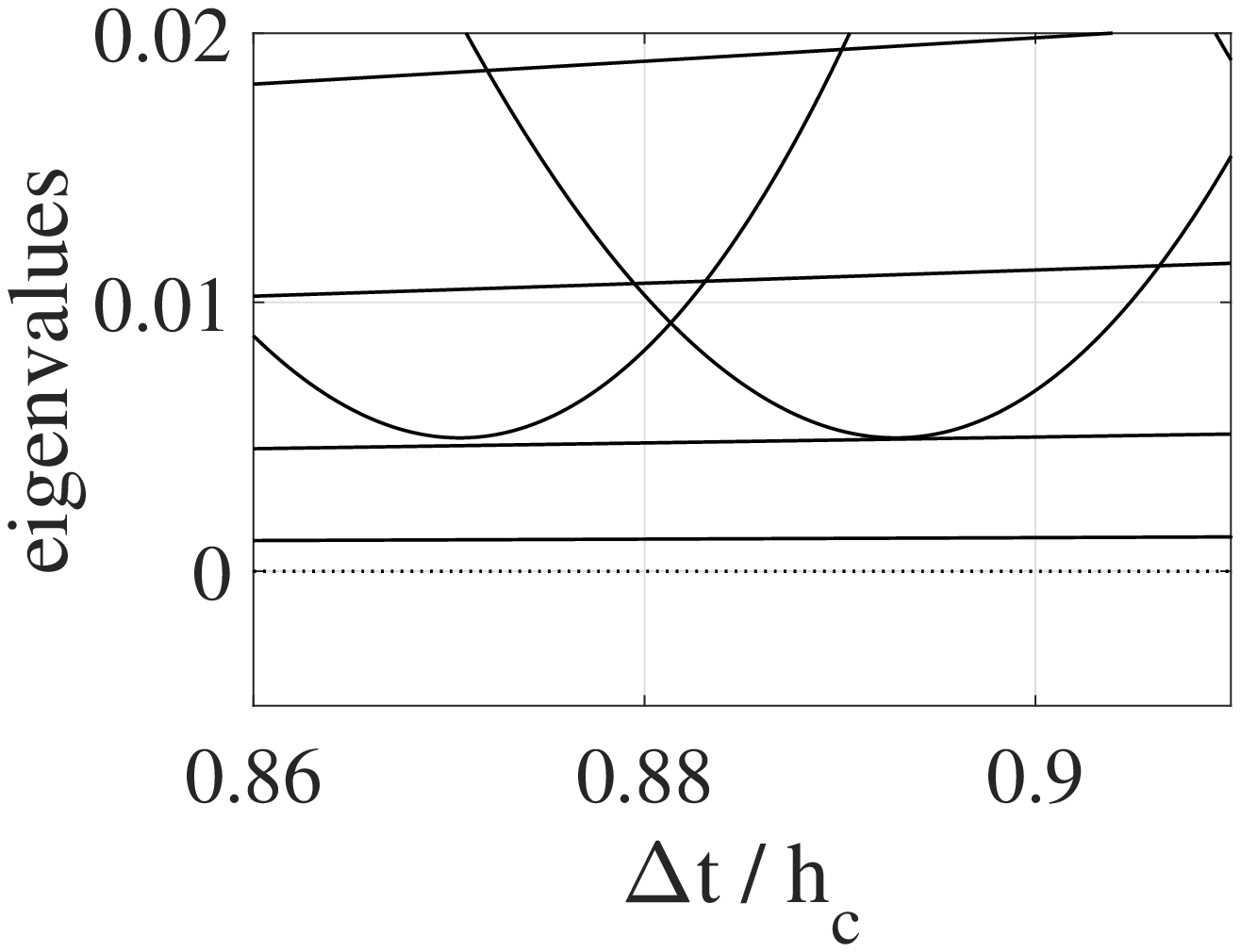}
\caption{Top: Eigenvalues of $(\Delta t^2/4) A_{S,p,\nu}$ vs. CFL ratio $\Delta t/h_{\operatorname*{c}}$ for $p=3$, $\nu = 0$ (left), $\nu = 0.001$ (right). Middle and bottom: Zoom on particular critical points from the top row.}
\label{FigEVDtApnue}
\end{figure}

In general, the value of the first potentially critical time-step $\Delta t_{\operatorname*{crit}}$ in the LF-LTS method without stabilization depends on $p$. 
For $p=2$, for instance, we have $\Delta t_{\operatorname*{crit}} = 0.7104 \cdot h_{\operatorname*{c}}$.
To purposely trigger the potentially unstable behavior of the LF-LTS method when $\nu=0$, we now set the initial condition $u_0(x)$ to the corresponding eigenfunction $\eta_{p,\nu}^{\operatorname*{crit}}$ of $\Delta t^2 A^{S,p,\nu}$ for $p=2$ and $\Delta t = \Delta t_{\operatorname*{crit}}$.
Shown on the left of Fig. \ref{Figvcrit}, we observe that $\eta_{p,\nu}^{\operatorname*{crit}}$ is highly oscillatory.

On the right of Fig. \ref{Figvcrit}, we observe that the numerical solution of the LF-LTS method without stabilization indeed grows linearly with time until $T = 1000$; it then starts to grow even faster probably due to round-off. 
In contrast, the stabilized LF-LTS method with $\nu = 0.01$ remains bounded for all time even when initialized with the same unstable eigenmode.
Hence, the stabilization has entirely removed any potentially unstable behavior at no extra cost, while the CFL restriction on $\Delta t$ essentially remains the same---see Table \ref{TabCFLnue}.
\begin{figure}
\centering
\includegraphics[width=0.49\textwidth]{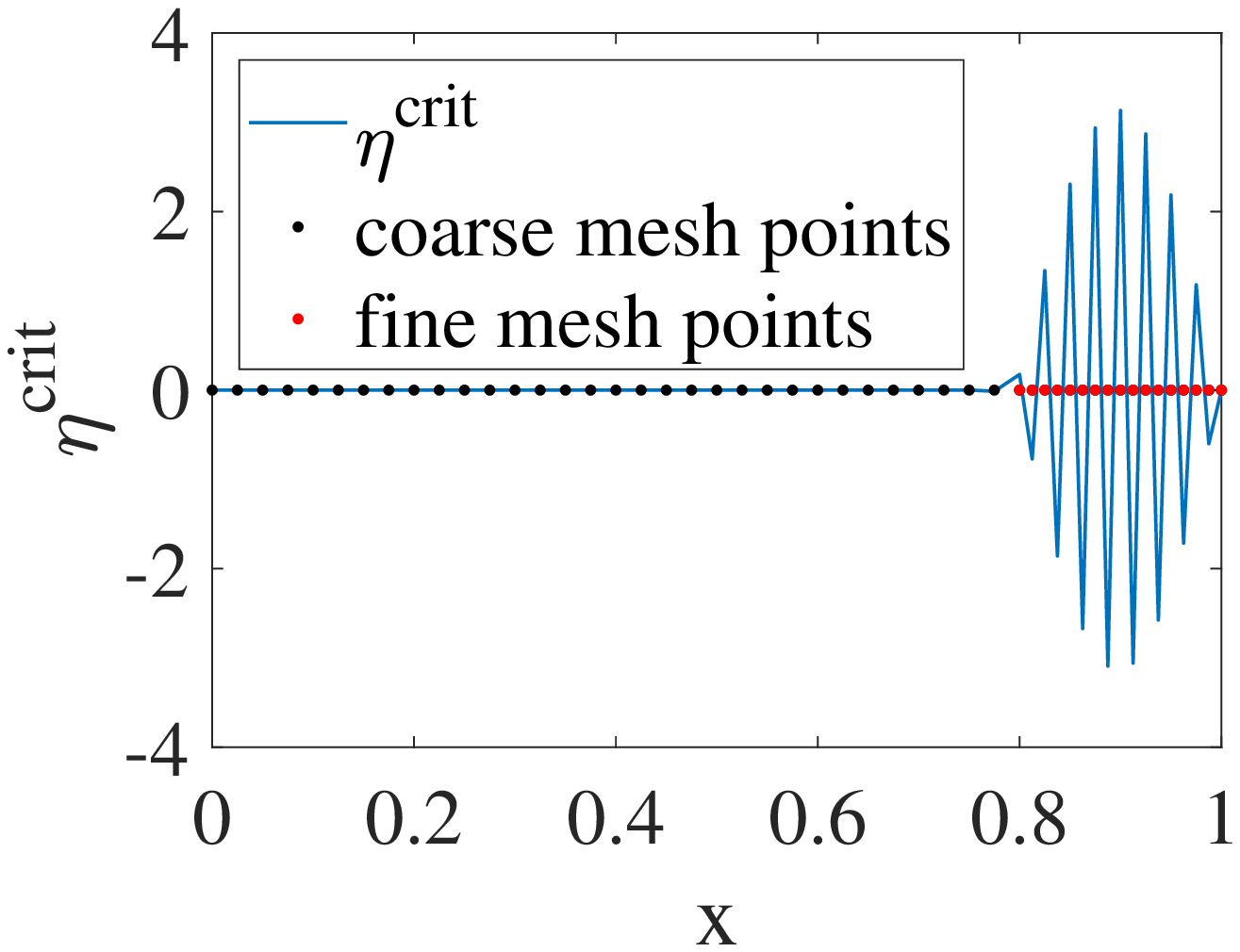}
\includegraphics[width=0.49\textwidth]{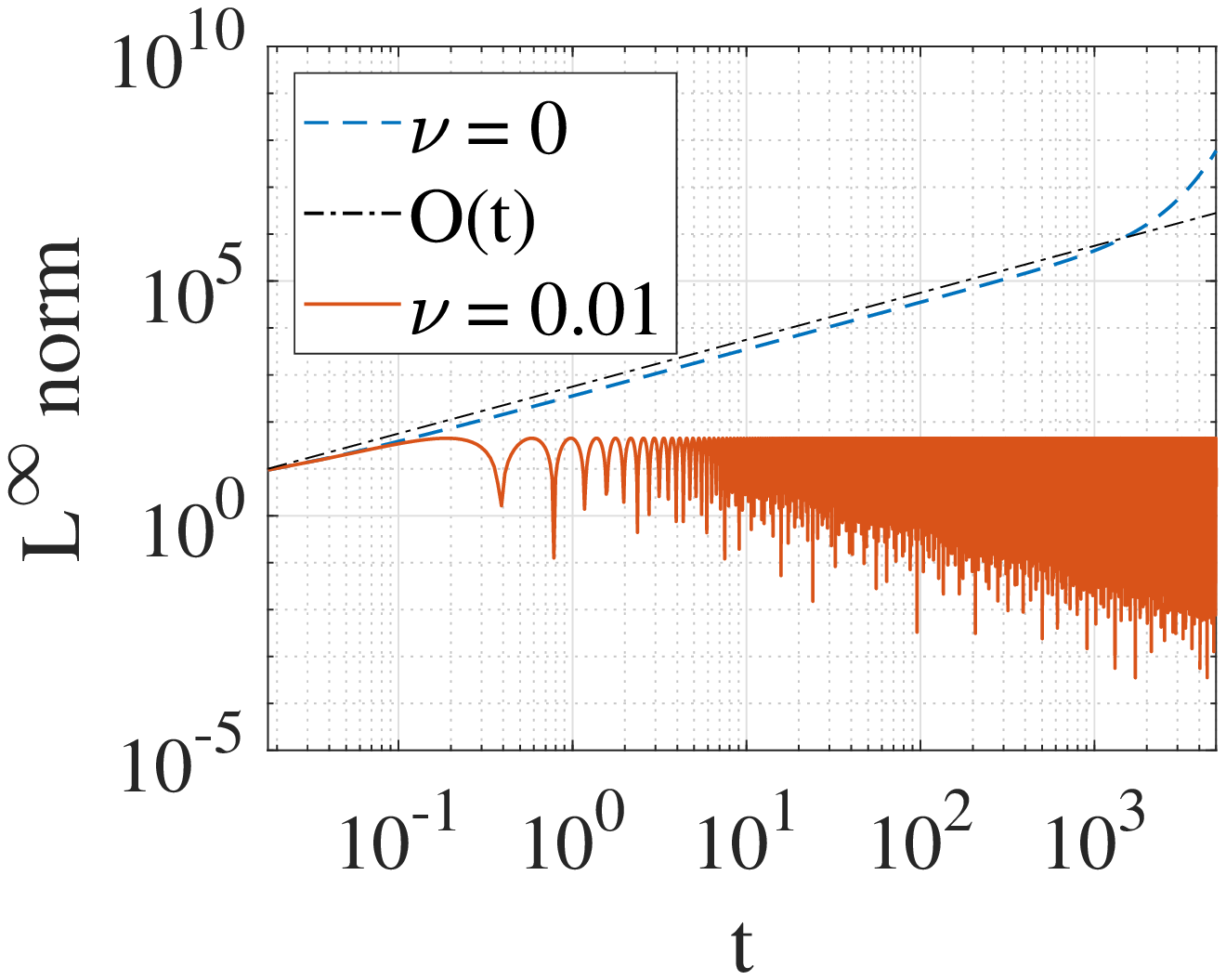}
\caption{Left: Eigenfunction $\eta_{p,\nu}^{\operatorname*{crit}}$ associated with a critical (unstable) time-step $\Delta t = \Delta t_{\operatorname*{crit}}$ on a locally refined mesh. Right: Evolution of $\left\| u(t,\cdot) \right\|_\infty$ vs.~time $t$ for $\nu = 0$ (dashed blue) and $\nu=0.01$ (solid orange) with dash-dotted reference line $\mathcal{O}\left(t\right)$.}
\label{Figvcrit}
\end{figure}

As illustrated in Fig. \ref{Figvcrit}, the LF-LTS method without stabilization may become unstable at particular discrete values of the time-steps.
We now ask whether purposely choosing a critically unstable time-step can prevent the LF-LTS method with $\nu=0$ from achieving optimal convergence.
Hence, we apply the LF-LTS method with $\nu=0$ and initial condition (\ref{IniCondGauss}) for a sequence of meshes of decreasing mesh size $h_{\operatorname*{c}}$ while choosing on each mesh a corresponding critical time-step $\Delta t_{\operatorname*{crit}}$.
Even for those somewhat extreme parameter settings, we observe in Fig. \ref{FigConvDtcrit} that optimal convergence of the LF-LTS\,($p=2,\nu=0$) method is in fact preserved on a finite time interval. 
\begin{figure}
\centering
\includegraphics[width=0.75\textwidth]{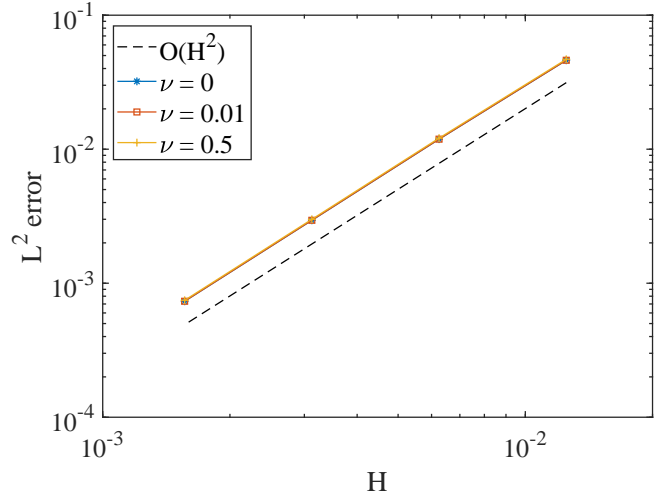}
\caption{Absolute $L^2$ error at time $t=2$ vs. $H = h_{\operatorname*{c}}$ for $\Delta t = \Delta t_{\operatorname*{crit}}$ on each mesh, $p=2$, $\nu = 0,0.01,0.5$, with dashed reference line $\mathcal{O}\left(H^2\right)$.}
\label{FigConvDtcrit}
\end{figure}

\subsection{Accuracy and convergence on graded meshes}
Finally, we consider a two-dimensional L-shaped domain with a reentrant corner, \[
\Omega = (0,1)^2 \setminus \left ( [0.5,1) \times (0.5,1] \right).
\]
Due to characteristic singularities of the solution at reentrant corners, uniform meshes
generally do not yield optimal convergence rates
 \cite{MuellerSchwab16}. 
A common remedy to restore the accuracy and achieve optimal convergence rates is to use (a-priori) graded meshes, see, e.g., \cite[Chapt. 3.3.7]{Schwab}, toward the reentrant corner.
Hence, we  first partition $\Omega$ into six triangles of equal size with a common vertex at the center $(0.5,0.5)$.
Then on every edge $e$ connected to the center, we allocate $N+1$ points at distance 
\begin{equation*}
|e| \left(\frac{k}{N}\right)^\beta, \quad k = 0,1,\ldots,N,
\end{equation*}
from it, where $\beta \geq 1$ is a fixed grading parameter; the larger $\beta$, the more strongly the triangles will cluster near the reentrant corner, whereas for $\beta = 1$ the mesh is uniform throughout $\Omega$.
All other vertices within the same $k$-th layer are distributed uniformly, as shown in Fig. \ref{FigGradedMesh} for a graded mesh with $\beta = 1.5$ and $N=10$.

Now, we consider (\ref{model problem}) in $\Omega$ with homogeneous Dirichlet boundary conditions and the initial conditions
\begin{equation}
\left\{
\begin{aligned}
u_0(x) &= 0, \\
v_0(x) &= -w(x),
\end{aligned}
\right.
\label{IniCondElliptic}
\end{equation}
where $w$ is the (singular) solution of the elliptic problem $\Delta w = 100$ in $\Omega$ with homogeneous Dirichlet boundary conditions, numerically computed on a highly refined mesh.

Again, we solve this problem numerically using $\mathbb{P}_1$-FE and the stabilized LF-LTS method with $\nu = 0.01$ for time integration on a sequence of graded meshes. On every mesh, we let
 the locally refined region $\Omega_{\operatorname*{f}}$ consist of those elements which lie inside the nearest $\left\lfloor\sqrt{N}\right\rfloor$ layers from the reentrant corner.
Hence, the number of local time-steps, $p$, depends on $N$ and thus varies from one mesh to another.
The numerical solution for $\beta=1.5$ and $N = 320$ is shown at time $t=0.3$ on the right of Fig. \ref{FigGradedMesh}.
\begin{figure}
\centering
\includegraphics[width=0.42\textwidth]{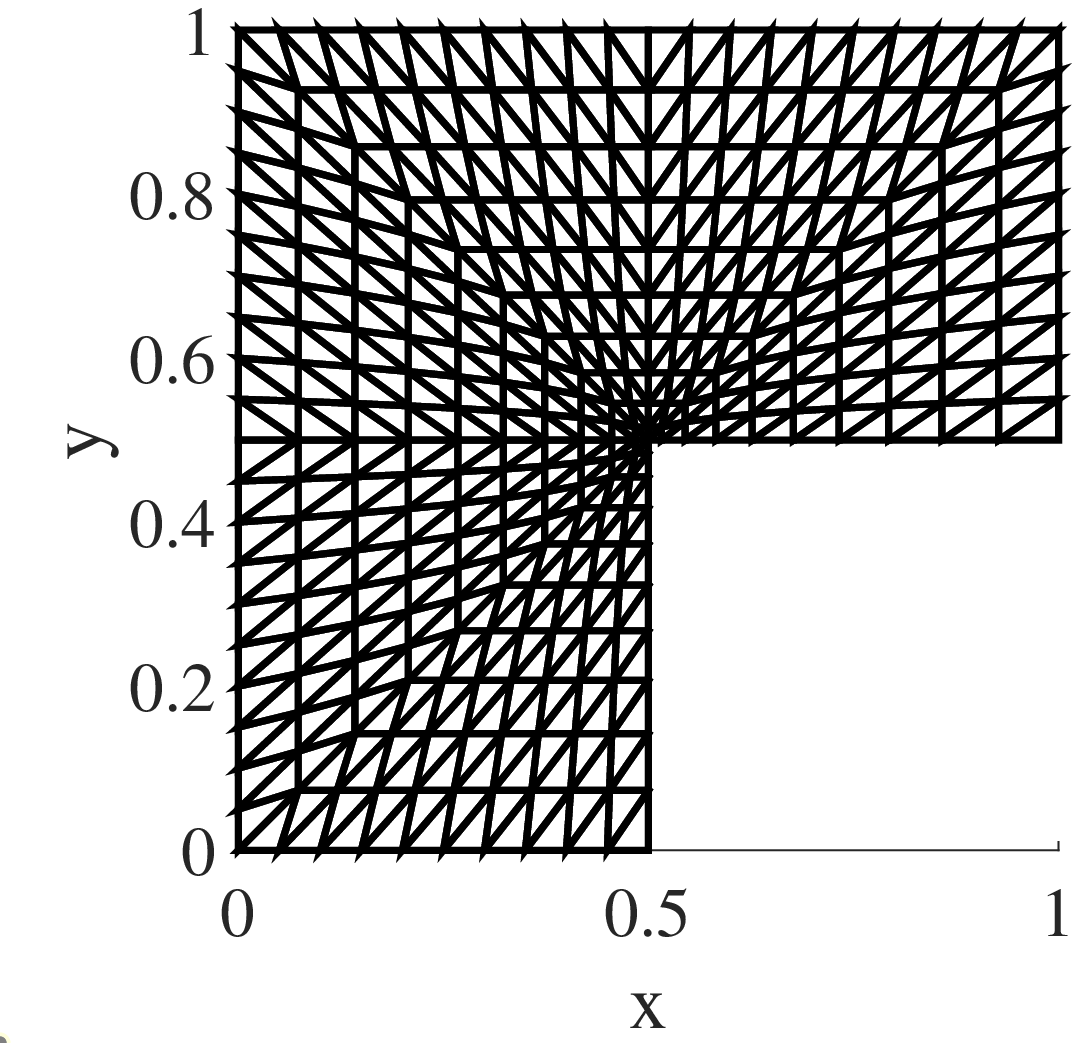}
\includegraphics[width=0.56\textwidth]{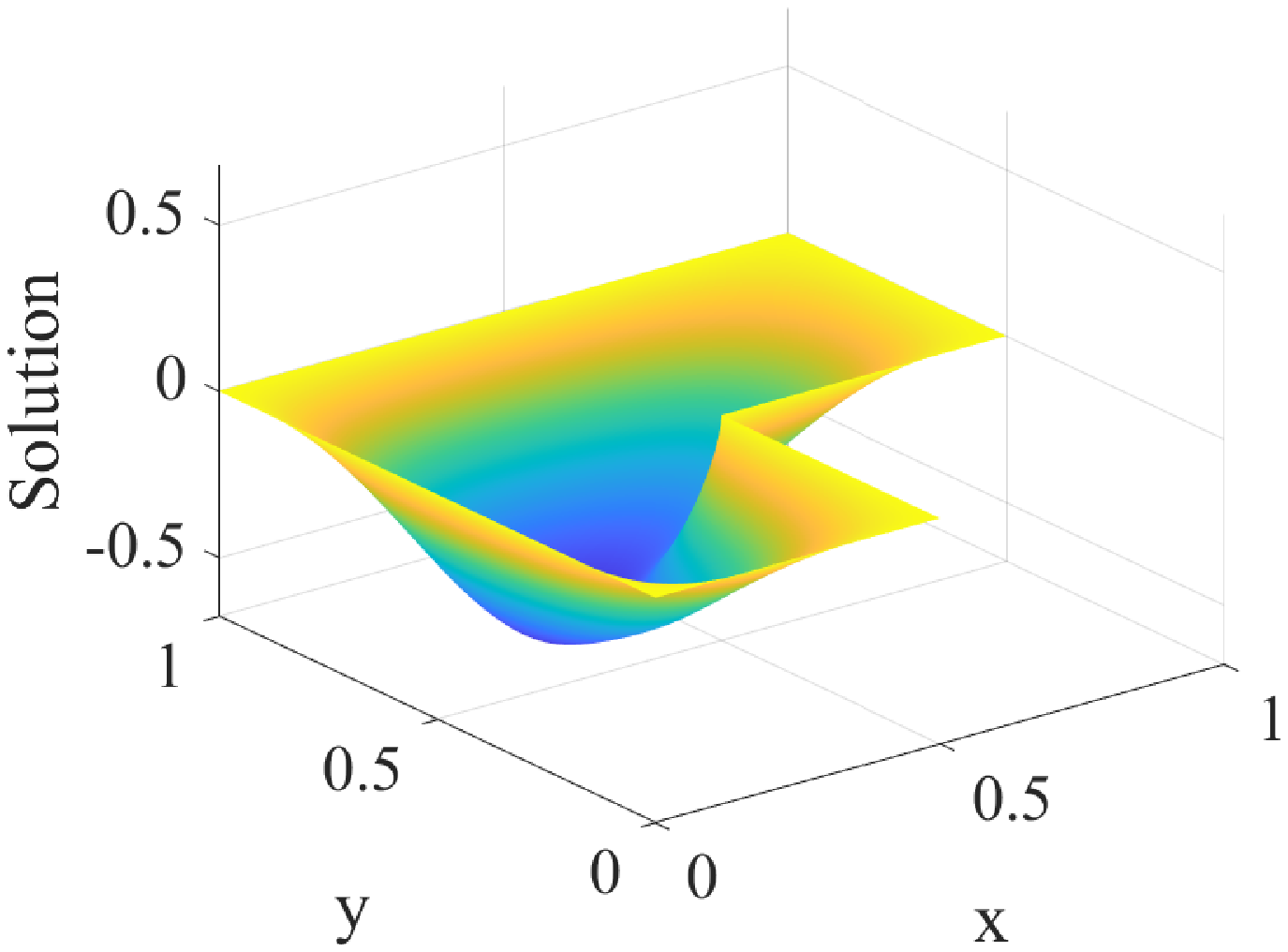}
\caption{Left: Graded mesh with grading parameter $\beta=1.5$. Right: LF-LTS solution of (\ref{model problem}) for initial conditions (\ref{IniCondElliptic}) at $t=0.3$}
\label{FigGradedMesh}
\end{figure}

Next, we study the convergence of the stabilized LF-LTS scheme on a sequence of graded meshes with  $\beta = 1.6$ and $h = \max_\tau h_\tau$.
As shown in Fig. \ref{FigConvLshape}, the error converges optimally as $\mathcal{O}(h)$ with respect to the $H^1$-norm and as $\mathcal{O}\left(h^2\right)$ with respect to the $L^2$-norm.
The $L^2$- and $H^1$-errors are both computed at the final time $t=0.3$ with respect to a reference solution on a finer mesh.

For the sake of comparison, we also display the errors obtained with a standard LF method on a sequence of uniform meshes. As expected, the LF method with a uniform FE discretization fails to achieve
the optimal convergence rates due to the corner singularity. Moreover, for any mesh size $h$, the 
LF-LTS method on the corresponding graded mesh is more accurate than the standard LF method on a uniform mesh.

\begin{figure}
\centering
\includegraphics[width=0.49\textwidth]{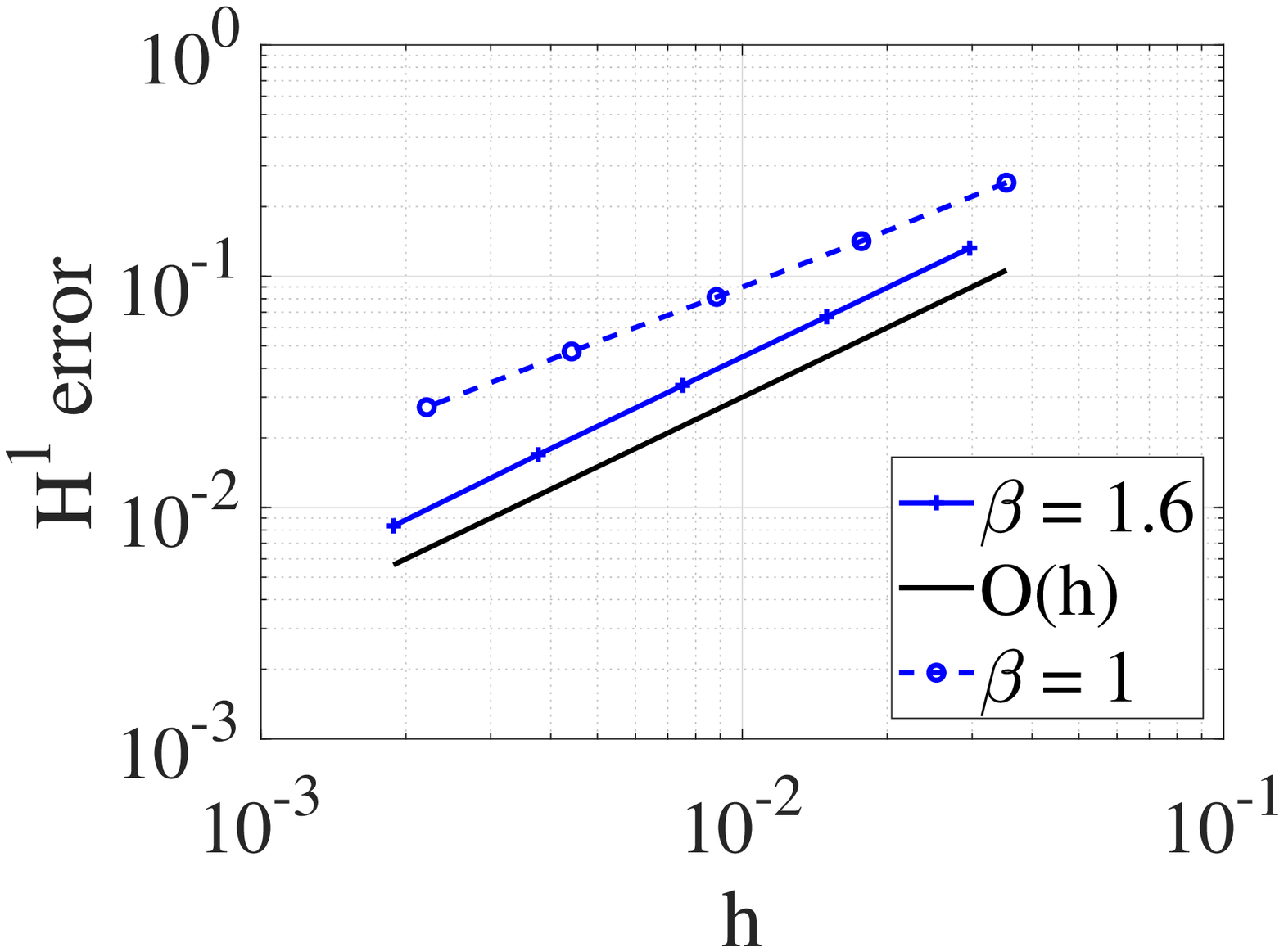}
\includegraphics[width=0.49\textwidth]{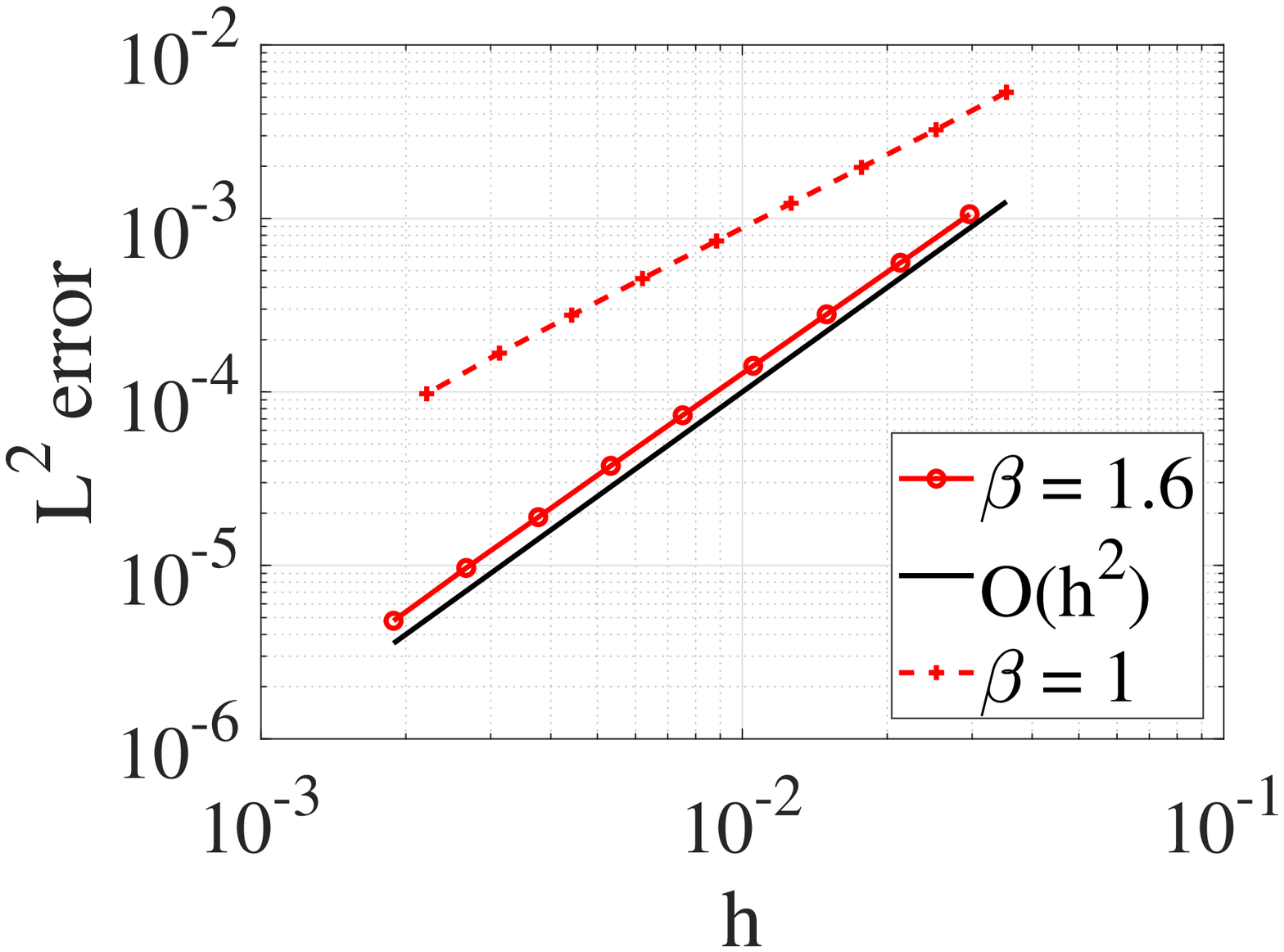}
\caption{L-shaped domain: Absolute errors vs. $h = \max_\tau h_\tau$ at $t=0.3$ either for graded
meshes with  $\beta = 1.6$ and LF-LTS time integration (solid line), or for uniform
meshes (with $\beta = 1$) and standard LF time integration (dashed line).
Left: $H^1$-error with reference line $\mathcal{O}(h)$; right: $L^2$-error with reference line $\mathcal{O}\left(h^2\right)$.}
\label{FigConvLshape}
\end{figure}

\section{Concluding remarks}
\label{sec:conclusion}
We have devised a stabilized version of the leapfrog based local time-stepping (LF-LTS) 
method from \cite{DiazGrote09}, which removes all discrete critical time-steps where the original method could potentially become unstable. 
Still, the new stabilized LF-LTS method retains all advantages of the original scheme:
it is fully explicit, second-order accurate and based on a ``leapfrog-like'' three-term recurrence relation. 
Thus, it is inherently parallel when combined with a standard FE discretization with mass-lumping or 
a discontinuous Galerkin FE method \cite{GSS06} while it also conserves a discrete energy (\ref{DefDiscNrj}).
Moreover, for sufficiently small stabilization $\nu>0$, the CFL stability condition of the
LF-LTS method remains essentially unchanged even in the presence of local refinement. 

The full stabilized LF-LTS algorithm is listed in Section 2.4; it 
is based on a nonsymmetric operator splitting -- see Remark \ref{RemarkLFCLTS}.
In Theorem \ref{TheoMain}, we have proved optimal $L^2$ convergence rates under a CFL condition (\ref{CFLtotal}) independent of the ``coarse''-to-``fine'' mesh size ratio and the number of local time-steps $p$.
Our numerical experiments corroborate those convergence rates even for $p$ as large as one thousand and validate the conservation of energy down to machine precision. 
Moreover, when combined with a FE discretization on a graded mesh, the LF-LTS method achieves optimal convergence rates even for singular solutions in a two-dimensional L-shaped domain.

Although we have restricted ourselves here to homogeneous wave equations for the sake of simplicity, 
the new stabilized LF-LTS algorithm extends to general nonzero forcing \cite{GroteMitkova10}.

\appendix

\section{Estimates for Chebyshev Polynomials}\label{AppendixA}

In this appendix, we will prove some estimate for the Chebyshev polynomials of
the first kind. We recall that
\begin{equation}
T_{p}^{\left(  m\right)  }\left(  1\right)  =%
{\displaystyle\prod\limits_{\ell=0}^{m-1}}
\frac{\left(  p^{2}-\ell^{2}\right)  }{\left(  2\ell+1\right)  }%
\qquad\text{and\quad}\left\Vert T_{p}^{(m)}\right\Vert _{L^{\infty}\left(
\left[  -1,1\right]  \right)  }=T_{p}^{\left(  m\right)  }\left(  1\right)  ,
\label{Tpproperties}%
\end{equation}
where the first relation follows from \cite[(1.97)]{Rivlin} and the second one
from \cite[Theorem 2.24]{Rivlin}, see also \cite[Corollary 7.3.1]%
{SauterSchwab2010}.

\begin{lemma}
\label{Tpmext}Let $\delta_{p,\nu}=1+\nu/p^{2}$ as in (\ref{defpolys1}). Then%
\begin{equation}
\left\Vert T_{p}^{\left(  m\right)  }\right\Vert _{L^{\infty}\left(
-\delta_{p,\nu},\delta_{p,\nu}\right)  }\leq\frac{p^{2m}\operatorname*{e}%
^{\nu}}{%
{\displaystyle\prod\limits_{\ell=0}^{m-1}}
\left(  2\ell+1\right)  }. \label{Tpmmax}%
\end{equation}
For $x=1+\varepsilon\in\left[  1,\delta_{p,\nu}\right]  $ and $p\geq m+1$ it
holds
\[
T_{p}^{\left(  m\right)  }\left(  x\right)  \geq\left(  1+\frac{\left(
p^{2}-m^{2}\right)  }{\left(  2m+1\right)  }\varepsilon\right)
{\displaystyle\prod\limits_{\ell=0}^{m-1}}
\frac{p^{2}-\ell^{2}}{2\ell+1}%
\]
and%
\begin{equation}
T_{p}\left(  \delta_{p,\nu}\right)  \geq1+\nu. \label{Tplower}%
\end{equation}

\end{lemma}

\begin{proof}
For $x\in\left[  -1,1\right]  $, we conclude from (\ref{Tpproperties}) that%
\[
\left\vert T_{p}^{\left(  m\right)  }\left(  x\right)  \right\vert \leq%
{\displaystyle\prod\limits_{\ell=0}^{m-1}}
\frac{\left(  p^{2}-\ell^{2}\right)  }{\left(  2\ell+1\right)  }
\]
holds. For $x\in\left(  -\delta_{p,\nu},\delta_{p,\nu}\right)  \backslash
\left[  -1,1\right]  $, we employ (cf. \cite[18.5.7]{NIST:DLMF})%
\[
T_{p}^{\left(  m\right)  }\left(  x\right)  =p\sum_{k=m}^{p}\left(  -2\right)
^{k}\left(  -1\right)  ^{m}\frac{\left(  p+k-1\right)  !}{\left(  p-k\right)
!\left(  2k\right)  !}\frac{k!}{\left(  k-m\right)  !}\left(  1-x\right)
^{k-m}.
\]
We set $x=1+\varepsilon$ for $0<\varepsilon<\nu/p^{2}$ and obtain%
\begin{align}
T_{p}^{\left(  m\right)  }\left(  x\right)   &  =\sum_{k=m}^{p}2^{k}\frac
{k!}{\left(  k-m\right)  !\left(  2k\right)  !}\varepsilon^{k-m}%
{\displaystyle\prod\limits_{\ell=0}^{k-1}}
\left(  p^{2}-\ell^{2}\right) \label{Tpmx}\\
& \leq\sum_{k=m}^{p}2^{k}\frac{k!}{\left(
k-m\right)  !\left(  2k\right)  !}\left(  \frac{\nu}{p^{2}}\right)  ^{k-m}%
{\displaystyle\prod\limits_{\ell=0}^{k-1}}
\left(  p^{2}-\ell^{2}\right) \nonumber\\
&  \leq p^{2m}\sum_{k=m}^{p}\frac{\nu^{k-m}}{\left(  k-m\right)  !%
{\displaystyle\prod\limits_{\ell=0}^{k-1}}
\left(  2\ell+1\right)  }\leq p^{2m}\sum_{k=0}^{p-m}\frac{\nu^{k}}{k!%
{\displaystyle\prod\limits_{\ell=0}^{k+m-1}}
\left(  2\ell+1\right)  }\nonumber\\
&  \leq\frac{p^{2m}}{%
{\displaystyle\prod\limits_{\ell=0}^{m-1}}
\left(  2\ell+1\right)  }\sum_{k=0}^{\infty}\frac{\nu^{k}}{k!%
{\displaystyle\prod\limits_{\ell=0}^{k-1}}
\left(  2\ell+1\right)  }\leq\frac{p^{2m}\operatorname*{e}^{\nu}}{%
{\displaystyle\prod\limits_{\ell=0}^{m-1}}
\left(  2\ell+1\right)  }.\nonumber
\end{align}
By symmetry of $T_{p}^{\left(  m\right)  }$, we conclude that the same upper
bound holds for $\left\vert T_{p}^{\left(  m\right)  }\left(  x\right)
\right\vert $ in the range $x=-1-\varepsilon$.

Next, we will obtain a lower bound for $T_{p}^{\left(  m\right)  }\left(
x\right)  $ for $x=1+\varepsilon$ and $0<\varepsilon<\nu/p^{2}$. Note that (\ref{Tpmx}) provides a representation with positive terms
so that the truncation of the sum after two terms (for $p\geq m+1$) leads to%
\begin{align*}
T_{p}^{\left(  m\right)  }\left(  x\right)   &  \geq2^{m}\frac{m!}{\left(
2m\right)  !}\left(  1+\frac{\left(  p^{2}-m^{2}\right)  }{\left(
2m+1\right)  }\varepsilon\right)
{\displaystyle\prod\limits_{\ell=0}^{m-1}}
\left(  p^{2}-\ell^{2}\right) \\
&  \geq\left(  1+\frac{\left(  p^{2}-m^{2}\right)  }{\left(  2m+1\right)
}\varepsilon\right)
{\displaystyle\prod\limits_{\ell=0}^{m-1}}
\frac{p^{2}-\ell^{2}}{2\ell+1}.
\end{align*}%
\end{proof}

\begin{lemma}
\label{Lembound}
Let $p\geq1$.\quad

\begin{enumerate}
\item $T_{p}$ is monotonically increasing in $\left[  1,\infty\right]  $.

\item It holds that
\begin{equation*}
\left\| T_p^\prime \right\|_{L^\infty(-\delta_{p,\nu},\delta_{p,\nu})} = T_p^\prime \left(\delta_{p,\nu}\right).
\end{equation*}

\item Let $\omega_{p,\nu}$ be as in (\ref{defpolys1}). Then,%
\begin{equation}
2p^{2}\operatorname*{e}\nolimits^{-\nu}\leq\omega_{p,\nu}\leq2p^{2}.
\label{estomegafb}%
\end{equation}

\end{enumerate}
\end{lemma}

\begin{proof}
Part 1 and 2 follows from the fact that the sequence of Chebyshev polynomials
$\left(  T_{p}\right)  _{p}$ form a Sturm's chain on $\left[  -1,1\right]  $.

Part 3. The estimate from above is a direct consequence of Part 1 and
(\ref{Tpproperties}).

For the lower estimate we employ%
\begin{equation}
T_{p}^{\prime}\left(  \delta_{p,\nu}\right)  \geq T_{p}^{\prime}\left(
1\right)  =p^{2}.\label{Tpxa}%
\end{equation}
From Lemma \ref{Tpmext} we conclude that%
\begin{equation}
T_{p}\left(  \delta_{p,\nu}\right)  \leq\operatorname*{e}\nolimits^{\nu}.\label{Tpxb}%
\end{equation}
The combination of (\ref{Tpxa}) and (\ref{Tpxb}) leads to the assertion.%
\end{proof}

\begin{lemma}
\label{LemEstPpnue}Let $\nu\in\left[  0,1/2\right]  $. It holds%
\begin{subequations}
\label{Pdeltatest}
\end{subequations}%
\begin{align}
\sup_{0\leq x\leq2\delta_{p,\nu}\omega_{p,\nu}}\left\vert \frac{P_{p,\nu
}\left(  x\right)  }{x}\right\vert  &\leq  1,\tag{%
\ref{Pdeltatest}%
a}\label{Pdeltatesta}\\
\inf_{0\leq x\leq\left( 1 + \delta_{p,\nu} \right)  \omega_{p,\nu}}%
\frac{P_{p,\nu}\left(  x\right)  }{x}  &  \geq\frac{2\nu}{\left(
2+\nu\right)  ^{2}\omega_{p,\nu}}\geq\frac{\nu}{\left(  2+\nu\right)
^{2}p^{2}}. \tag{%
\ref{Pdeltatest}%
b}\label{Pdeltatestb}%
\end{align}
\end{lemma}

For the proof of (\ref{Pdeltatesta}) we refer to  \cite[Theorem 5.1, proof part (ii)]{CarHocStu19},
whereas  (\ref{Pdeltatestb}) is implied by \cite[Theorem 5.2]{CarHocStu19}.

\begin{lemma}
\label{LemPpnue}
It holds%
\begin{align*}
\sup_{0\leq x\leq\left(  2+\nu/p^{2}\right)  \omega_{p,\nu}}\left\vert
P_{p,\nu}\left(  x\right)  \right\vert  &  \leq\frac{2\left(  2+\nu\right)
}{1+\nu} = 4 - \frac{2 \nu}{1 + \nu},\\
\inf_{x\in\left(  c_{\operatorname*{coer}}\Delta t^{2},\left(  2+\nu
/p^{2}\right)  \omega_{p,\nu}\right)  }P_{p,\nu}\left(  x\right)   &  \geq
\min\left\{  c_{\operatorname*{coer}}\frac{\Delta t^{2}}{\left(
\nu+1\right)  },\frac{2\nu}{\nu+1}\right\}  .
\end{align*}
\end{lemma}

\begin{proof}
Since the upper estimate directly follows from \cite[Theorem 5.2]{CarHocStu19}, we only need to prove the lower bound.

For $\nu/p^{2}\leq x/\omega_{p , \nu}<2+\nu/p^{2}$, we conclude from
(\ref{Tplower}) that%
\begin{equation*}
P_{p,\nu}\left(  x\right)  =2\left(  1-\frac{T_{p}\left(  \delta_{p,\nu}%
-\frac{x}{\omega_{p,\nu}}\right)  }{T_{p}\left(  \delta_{p,\nu}\right)
}\right)  \geq2\left(  1-\frac{1}{1+\nu}\right)  =\frac{2\nu}{\nu+1}.
\end{equation*}

For $c_{\operatorname*{coer}}\Delta t^{2}\leq x/\omega_{p,\nu}\leq$ $\nu
/p^{2}$, 
we employ%
\begin{equation}
\inf_{0\leq x\leq \nu/p^{2}  \omega_{p,\nu}}\frac{P_{p,\nu
}\left(  x\right)  }{x}=\frac{1}{\omega_{p,\nu}}\inf_{0\leq y\leq \nu/p^{2}%
}q_{p,\nu}\left(  y\right)  \label{infproof1}%
\end{equation}
for
\[
q_{p,\nu}\left(  y\right)
=\frac{2}{y}\left(  1-\frac{T_{p}\left(  \delta_{p,\nu}-y\right)  }{T_{p}\left(
\delta_{p,\nu}\right)  }\right)  .
\]
We employ a Taylor argument to obtain%
\[
q_{p,\nu}\left(  y\right)  =\frac{2T_{p}^{\prime}\left(  \delta_{p,\nu
}\right)  -yT_{p}^{\prime\prime}\left(  \delta_{p,\nu}-\xi\right)  }%
{T_{p}\left(  \delta_{p,\nu}\right)  }\geq2\frac{T_{p}^{\prime}\left(
\delta_{p,\nu}\right)  }{T_{p}\left(  \delta_{p,\nu}\right)  }-y\left\vert
\frac{T_{p}^{\prime\prime}\left(  \delta_{p,\nu}-\xi\right)  }{T_{p}\left(
\delta_{p,\nu}\right)  }\right\vert
\]
for $0\leq\xi\leq y$. The combination of (\ref{Tpxa}) and (\ref{Tpxb}) leads
to%
\begin{align*}
q_{p,\nu}\left(  y\right)   
&  \geq 2p^{2} \operatorname*{e}\nolimits^{-\nu} - y\frac{\left\vert T_{p}^{\prime\prime}\left(  \delta_{p,\nu}-\xi\right)\right\vert }{T_{p}\left(  \delta_{p,\nu}\right)  } 
\overset{\text{Lem.\ref{Tpmext}}}{\geq}2p^{2} \operatorname*{e}\nolimits^{-\nu} \left(1-y\frac{p^{2}\operatorname*{e}\nolimits^{2\nu}}{6}\right) \\
&  \geq2p^{2}\operatorname*{e}\nolimits^{-\nu}\left(  1-\frac{\nu \operatorname*{e}\nolimits^{2\nu}}{6}\right)  .
\end{align*}
For $0 \leq \nu \leq 1/2$, it holds that%
\[
2p^{2}\operatorname*{e}\nolimits^{-\nu}\left(  1-\frac{\nu\operatorname*{e}%
\nolimits^{2\nu}}{6}\right)  \geq\frac{p^{2}}{\nu+1},%
\]
which implies%
\begin{equation}
q_{p,\nu}\left(  y\right)  \geq\frac{p^{2}}{\nu+1}\quad\forall y\in\left[
0,\nu/p^{2}\right]  . \label{infproof2}%
\end{equation}
Now, we combine (\ref{infproof1}) and (\ref{infproof2}) to get%
\[
P_{p,\nu}\left(  x\right)  \geq
\frac{x}{\nu+1}\frac{p^{2}}{\omega_{p,\nu}}\geq
\frac{c_{\operatorname*{coer}}\Delta t^{2} p^2}{\nu+1} \overset{p\geq 1}{\geq}%
c_{\operatorname*{coer}}\frac{\Delta t^{2}}{\left(  \nu+1\right)  }.
\]%
\end{proof}

\begin{lemma}
\label{Lemdiffquot}
Let $0\leq\nu\leq 1/2$. Then%
\begin{equation*}
\sup_{0<y<\left(  2+\nu/p^{2}\right)  \omega_{p,\nu}}\left\vert \frac
{1-y^{-1}P_{p,\nu}\left(  y\right)  }{P_{p,\nu}\left(  y\right)  }\right\vert
\leq\frac{\nu+1}{2\nu}.
\end{equation*}
\end{lemma}

\begin{proof}
From (\ref{Pdeltatestb}) we conclude that%
\begin{equation}
\frac{P_{p,\nu}\left(  y\right)  }{y}\geq\frac{\nu}{\left(  2+\nu\right)
^{2}p^{2}} > 0 \quad\forall\, 0 < y\leq\left(  2+\frac{\nu}{p^{2}}\right)
\omega_{p,\nu}, 
\label{Ppnuelow3}%
\end{equation}
so that%
\[
\frac{1-y^{-1}P_{p,\nu}\left(  y\right)  }{P_{p,\nu}\left(  y\right)  }%
\leq\frac{1}{P_{p,\nu}\left(  y\right)  }.
\]
On the other hand, the estimate (\ref{Pdeltatesta}) implies%
\[
1-y^{-1}P_{p,\nu}\left(  y\right)  \geq 0
\]
and we arrive at%
\begin{equation}
\left\vert \frac{1-y^{-1}P_{p,\nu}\left(  y\right)  }{P_{p,\nu}\left(
y\right)  }\right\vert \leq\frac{1}{P_{p,\nu}\left(  y\right)  }\quad
\forall\, 0 < y\leq\left(  2+\frac{\nu}{p^{2}}\right)  \omega_{p,\nu}.
\label{estfinaltech1}%
\end{equation}
For $\nu/p^{2}\leq x/\omega_{p\not , \nu}<2+\nu/p^{2}$, we conclude from
(\ref{Tplower}) that%
\begin{equation}
P_{p,\nu}\left(  x\right)  =2\left(  1-\frac{T_{p}\left(  \delta_{p,\nu}%
-\frac{x}{\omega_{p,\nu}}\right)  }{T_{p}\left(  \delta_{p,\nu}\right)
}\right)  \geq2\left(  1-\frac{1}{1+\nu}\right)  =\frac{2\nu}{\nu+1}.
\label{estPpnuelow2}%
\end{equation}
We combine (\ref{estfinaltech1}) with (\ref{estPpnuelow2}) to obtain%
\[
\left\vert \frac{1-y^{-1}P_{p,\nu}\left(  y\right)  }{P_{p,\nu}\left(
y\right)  }\right\vert \leq\frac{\nu+1}{2\nu}.
\]
For $0<y/\omega_{p,\nu}<\nu/p^{2}$, a Taylor argument leads to%
\[
\frac{P_{p,\nu}\left(  y\right)  }{y}=1-\frac{y}{\omega_{p,\nu}^{2}}%
\frac{T_{p}^{\prime\prime}\left(  \delta_{p,\nu}-\xi\right)  }{T_{p}\left(
\delta_{p,\nu}\right)  }%
\]
for some $\xi\in\left(  0,\frac{y}{\omega_{p,\nu}}\right)  $. The combination
of (\ref{Tpmmax}), (\ref{Tplower}), and (\ref{estomegafb}) leads to%
\begin{align*}
\left\vert \frac{1-y^{-1}P_{p,\nu}\left(  y\right)  }{P_{p,\nu}\left(y\right)  }\right\vert 
&= \frac{y}{\omega_{p,\nu}^{2}} \frac{\left\vert T_{p}^{\prime\prime}\left(  \delta_{p,\nu}-\xi\right)  \right\vert } {T_{p}\left(  \delta_{p,\nu}\right)  P_{p,\nu}\left(  y\right)} \\
&\leq \frac{1}{\omega_{p,\nu}^{2}} \frac{p^{4}\operatorname*{e}^{\nu}}{3\left(1+\nu\right)  } \frac{y}{P_{p,\nu}\left(  y\right)  }
\leq\frac{\operatorname*{e}^{3\nu}}{12\left(  1+\nu\right)  } \frac{y}{P_{p,\nu}\left(  y\right)  }.
\end{align*}
From (\ref{infproof1}) and (\ref{infproof2}) one derives%
\[
\frac{P_{p,\nu}\left(  y\right)  }{y}\geq\frac{1}{\omega_{p,\nu}}\frac{p^{2}%
}{\nu+1}\overset{\text{(\ref{estomegafb})}}{\geq}\frac{1}{2\left(
\nu+1\right)  }.
\]
Thus,%
\[
\left\vert \frac{1-y^{-1}P_{p,\nu}\left(  y\right)  }{P_{p,\nu}\left(
y\right)  }\right\vert \leq\frac{\operatorname*{e}^{3\nu}}{6}.
\]
For $0\leq\nu\leq 1/2$ it holds 
\[
\frac{\nu+1}{2\nu}\geq\frac{\operatorname*{e}%
^{3\nu}}{6}
\] 
and the assertion follows.%
\end{proof}

\section{Equivalence of the algorithm}
\label{secAppB}

We prove the equivalence of definition (\ref{leap_frog_lts_fd}) and the algorithm in section \ref{sec:stablflts} by using the following recursion for $P_{p,\nu}^{\Delta t}$, which is also noticed in \cite{CarHocStu19}.
\begin{lemma}
\label{LemRecurPpnue}
Let the polynomials $P_{p,\nu,k}(x) \in \mathbb{P}_k$, $k = 0,1,\ldots,p$ be defined by
\begin{equation}
\left\{
\begin{aligned}
P_{p,\nu,0}^{\Delta t}(x) &:= 0, \\
P_{p,\nu,1}^{\Delta t}(x) &:= \dfrac{2}{\omega_{p,\nu} \delta_{p,\nu}}, \\
P_{p,\nu,k+1}^{\Delta t}(x) &:= 2 \beta_{k+1/2} \left( \delta_{p,\nu} - \dfrac{\Delta t^2 x}{\omega_{p,\nu}} \right) P_{p,\nu,k}^{\Delta t}(x) - \beta_k P_{p,\nu,k-1}^{\Delta t}(x)\\
&\qquad  + \dfrac{4}{\omega_{p,\nu}} \beta_{k+1/2},
\end{aligned}
\right.
\label{RecurPDtpnue}
\end{equation}
with coefficients $\beta_k$ and $\beta_{k+1/2}$ defined as in (\ref{Def_beta}).
Then 
\[
P_{p,\nu}^{\Delta t}(x) = P_{p,\nu,p}^{\Delta t}(x).
\]
\end{lemma}

\begin{proof}
The result follows from straightforward computation using the standard recursion for Chebyshev polynomials.
\end{proof}

\begin{theorem}
Definition (\ref{leap_frog_lts_fd}) and the Stabilized LF-LTS Galerkin FE algorithm in section \ref{sec:stablflts} are equivalent.
\label{TheoEquivLTSAlgo}
\end{theorem}

\begin{proof}
Let $z_{S,k}^{(n)}$ be defined as in the algorithm. In the following, we want to show that
\begin{equation}
z_{S,k}^{(n)} = 
u_{S}^{(n)} - \dfrac{\Delta t^2}{2} A^{S} P_{p,\nu,k}^{\Delta t} \left( \Pi_{\operatorname*{f}}^{S} A^{S} \right) u_{S}^{(n)}
\label{zSk_explicit}
\end{equation}
for every $k = 0,1,\ldots,p$. 

For $k =0,1$, this follows simply by inserting (\ref{RecurPDtpnue}). Hence, we may deduce by induction, 
\begin{align*}
z_{S,k+1}^{(n)} &= \left( 1 + \beta_{k} \right) z_{S,k}^{(n)} - \beta_{k} z_{S,k-1}^{(n)} \\
&\quad - \left( \dfrac{\Delta t}{p} \right)^2 \left( \dfrac{2 p^2}{\omega_{p,\nu}} \right) \beta_{k+ 1/2} 
\left( A^{S} \Pi_{\operatorname*{c}}^{S} u_{S}^{(n)} + A^{S} \Pi_{\operatorname*{f}}^{S} z_{S,k}^{(n)}  \right)& \\
&= \left( 1 + \beta_{k} \right) u_{S}^{(n)} - \left( 1 + \beta_{k} \right) \dfrac{\Delta t^2}{2} A^{S} P_{p,\nu,k}^{\Delta t} \left( \Pi_{\operatorname*{f}}^{S} A^{S} \right) u_{S}^{(n)} - \beta_{k}  u_{S}^{(n)} \\
&\quad + \beta_{k} \dfrac{\Delta t^2}{2} A^{S} P_{p,\nu,k-1}^{\Delta t} \left( \Pi_{\operatorname*{f}}^{S} A^{S} \right) u_{S}^{(n)} - \left( \dfrac{\Delta t}{p} \right)^2 \left( \dfrac{2 p^2}{\omega_{p,\nu}} \right) \beta_{k+ 1/2} A^{S} u_{S}^{(n)} \\
&\quad + \dfrac{\Delta t^2}{2} \left( \dfrac{\Delta t}{p} \right)^2 \left( \dfrac{2 p^2}{\omega_{p,\nu}} \right) \beta_{k+ 1/2} A^{S} \Pi_{\operatorname*{f}}^{S} A^{S} P_{p,\nu,k}^{\Delta t} \left( \Pi_{\operatorname*{f}}^{S} A^{S} \right) u_{S}^{(n)}.
\end{align*}
Now observe that from the well-known recursion for Chebyshev polynomials of first kind, $T_{k+1}(x) = 2 x T_k(x) - T_{k-1}(x)$, it follows that
\[
1 + \beta_{k} = 2 \delta_{p,\nu} \beta_{k+1/2}.
\]
From this, we conclude
\begin{align*}
z_{S,k+1}^{(n)} &= u_{S}^{(n)} - \dfrac{\Delta t^2}{2} A^{S} \left( 2 \left( \delta_{p,\nu} I^S - \dfrac{\Delta t^2}{\omega_{p,\nu}} \Pi_{\operatorname*{f}}^{S} A^{S}  \right) \beta_{k+1/2} P_{p,\nu,k}^{\Delta t} \left( \Pi_{\operatorname*{f}}^{S} A^{S} \right) \right. \\
&\qquad\qquad\qquad \left. - \beta_{k} P_{p,\nu,k-1}^{\Delta t} \left( \Pi_{\operatorname*{f}}^{S} A^{S} \right) + \dfrac{4}{\omega_{p,\nu}} \beta_{k+1/2} I^{S} \right) u_{S}^{(n)} \\
&\overset{\text{(\ref{RecurPDtpnue})}}{=} u_{S}^{(n)} - \dfrac{\Delta t^2}{2} A^{S} P_{p,\nu,k+1}^{\Delta t} \left( \Pi_{\operatorname*{f}}^{S} A^{S} \right) u_{S}^{(n)}.
\end{align*}
In summary, we have proven that
\begin{align*}
u_{S}^{(n+1)} &= - u_{S}^{(n-1)} + 2 \left(u_{S}^{(n)} - \dfrac{\Delta t^2}{2} A^{S} P_{p,\nu}^{\Delta t} \left( \Pi_{\operatorname*{f}}^{S} A^{S} \right) u_{S}^{(n)}\right),
\end{align*}
which implies the assertion.
\end{proof}

\section*{Acknowledgements} 
We thank Constantin Carle and Marlis Hochbruck for fruitful discussions, in particular, about the proof of Lemma \ref{Tpmext}.

This work was supported by the Swiss National Science Foundation under grant SNF 200020-188583. 

\bibliographystyle{amsplain}

\end{document}